%% file: arXiv_v2.tex
\documentclass[11pt,twoside,reqno]{amsart}
\usepackage{setspace}
\usepackage{amsmath,hyperref,amstext,amssymb,amsthm,mathrsfs,eurosym,dsfont}
\usepackage{comment}
\usepackage{xcolor}
\usepackage{enumerate}   
\usepackage{enumitem}
\usepackage{tikz}
\usepackage{geometry}
\newtheorem{theorem}{Theorem}[section]
\newtheorem{lemma}[theorem]{Lemma}
\newtheorem{proposition}[theorem]{Proposition}
\theoremstyle{definition}
\newtheorem{defin}[theorem]{Definition}
\newtheorem{notation}[theorem]{Notation}
\newtheorem{assumption}[theorem]{Assumption}

\newtheorem{remark}[theorem]{Remark}

\numberwithin{equation}{section}
\allowdisplaybreaks
\input{newcommands}

\def \red {\textcolor{red}}

\title[Weak equilibria of a mean-field market model under asymmetric information]{Weak equilibria of a mean-field market model under asymmetric information}

\author[A. Cecchin]{Alekos Cecchin}

\author[M. Fischer]{Markus Fischer}

\author[C. Fontana]{Claudio Fontana}

\author[G. Lanaro]{Giacomo Lanaro}

\address{Department of Mathematics ``Tullio Levi-Civita'', University of Padova, Italy}

\email{acecchin@math.unipd.it, fischer@math.unipd.it, fontana@math.unipd.it, glanaro@math.unipd.it}


\thanks{
This work is partially supported by 
the PRIN 2022 Project 2022BEMMLZ ``Stochastic control and games and the role of information''. 
A.C. and C.F. are also supported by INdAM-GNAMPA Project 2025 ``Stochastic control and MFG under asymmetric information: methods and applications'' and the PRIN 2022 PNRR Project P20224TM7Z  ``Probabilistic methods for energy transition''. A.C. also acknowledges support  from project MeCoGa ``mean-field control and games'' of the University of Padova through the program STARS@UNIPD - NextGenerationEU
}

\date{May 1, 2026}

\begin{document}

\begin{abstract}
We investigate how asymmetric information affects equilibrium price formation in an economy with many interacting agents. Motivated by a finite-player model with two populations of asymmetrically informed agents, we study its mean-field limit when one population observes an additional stochastic factor which is inaccessible to the other. The resulting equilibrium condition involves the conditional expectation of the adjoint process and, therefore, differs from standard mean-field formulations based on the state process. We prove existence of mean-field equilibria in probabilistic weak sense by combining discretization and weak convergence arguments with a lifting procedure tailored to preserve compatibility in the limit. Under additional assumptions, we obtain a conditional asymptotic justification of the mean-field price as an approximation of the finite-player market clearing relation. Finally, we illustrate how, in the case of a single informed agent, her strategy can be characterized in terms of the equilibrium.
\end{abstract}

\maketitle

\section{Introduction}
\label{section: introduction}

The study of decision-making in large populations of interacting agents is a central problem in stochastic control and game theory. In many applications, a key challenge in determining players' strategies arises from the heterogeneity of the information available to each individual. In particular, interactions among players are determined by the differences in the information accessible to each player as well as how this information propagates throughout the system. To analyze these issues, we consider a framework in which players of two different types make decisions based on the different sources of randomness they can access.
In our model, players are divided into two subpopulations, one of informed players and one of less informed (standard) players. Players of the informed population have access to a common stochastic factor which remains unobservable to the other subpopulation. This additional source of randomness impacts the decision-making process of the informed players, affecting, through their interaction in equilibrium, also the collective behavior of the standard players. 
{Our primary aim is to analyze how an additional source of payoff-relevant information, observed only by the informed subpopulation, affects equilibrium price formation. By contrast, the paper is not intended to model in full generality how dealers’ inventory constraints transmit client order flow into prices.}

The main objective of this paper is to study the existence of mean-field equilibria in such a framework and explore what these equilibria reveal about the extent of common information in the market. 
Our work is motivated by the problem of price formation in a market populated by finitely many rational agents with different levels of information and trading the same asset. 
In this model, each agent solves a stochastic optimal control problem that depends on the asset price process, which is initially assumed to be exogenously given. The interaction among the agents is given by the equilibrium price, which is determined by the market clearing condition. This condition imposes a constraint on the optimal controls of all agents in the market and provides the equilibrium price through a fixed point. Enforcing this constraint, we obtain a stochastic system populated by two different types of players, characterized by a highly recursive structure. As a consequence, proving the existence of a solution to this model is a complex task. To address this challenge, we adopt an approach based on  mean-field game (MFG) theory, which provides a powerful framework for analyzing large-scale strategic interactions by studying the statistical behavior of a population of infinitely many interacting agents. 

Introduced in \cite{lasry2007mean} and \cite{huang2006large}, MFGs describe the limiting behavior of stochastic differential games with a large number of interacting players. We establish an analogy between the equilibrium of the system we are considering and the concept of weak equilibrium in MFGs; see \cite{carmona2016mean, lacker2016general}. Our goal is to prove the existence of solutions to an equation that depends on the collective behavior of the players, capturing the interdependence of individual strategies in response to the common information available. Once the existence of solutions is established, the resulting mean-field equilibrium reflects the impact of information asymmetry.

In many applications of MFG theory to economic and financial problems (\cite{gomes2015economic, carmona2020applications}), the model is defined by the exogenous evolution of key system variables. Examples include optimal trading strategies \cite{lehalle2019mean}, portfolio liquidation \cite{fu2021mean}, and energy market optimization \cite{alasseur2020extended, alasseur2023mean}. 
In our framework, the evolution of the system is determined endogenously by the strategic interactions between players. This problem is related to several works which explore equilibrium price formation in MFG markets under the market clearing condition \cite{gueant:hal-01393104, gomes2020mean, gomes2021mean, fujii2022mean, fujii2022strong, fujii2023equilibrium, fujii2022equilibrium}. These studies analyze the interaction between individual agents and equilibrium dynamics, incorporating various sources of randomness, including idiosyncratic noise and common noise. The mean-field approach has been also applied to the study of price formation in electricity market \cite{feron2020price, feron2021price, aid2022equilibrium}, in markets subject to random supply \cite{gomes2023random}, and in the context of pollution regulation through cap and trade mechanism \cite{shrivats2022mean,del2024mean}.

A common limitation of the works mentioned above is that the equilibrium price defined by the market clearing condition is supposed to be adapted to the common flow of information. As emphasized in \cite{fujii2022mean}, this assumption is inconsistent with the market clearing condition in a finitely populated market, while it may be reasonable in the mean-field limit, where the idiosyncratic noises of each agent, assumed to be pairwise independent, vanish. This reasoning does not apply to our framework: the key difference is the presence of the additional source of randomness shared by the informed subpopulation, which does not vanish in the mean-field limit. As a consequence, we derive an equation for the equilibrium price, whose solution cannot be obtained using well-known arguments such as the continuation method developed in \cite{peng1999fully}. 
Hence, our main result is to prove existence of {\em weak} equilibria for the mean-field model we introduce. 
Following the approach described in \cite{carmona2016mean,carmona2018probabilistic2},  we construct the weak solution as the limit in distribution of a sequence of discretized problems. This approach is applied also in \cite{tangpi2024optimal} in the context of optimal bubble riding models, as well as in \cite{bergault2024mean} to prove existence of weak equilibria for a Stackelberg game with an informed major player.
Compared to the above papers, the main difficulty of our work is that the interaction arises as the conditional expectation of the costate process, instead of the state process. 
{This feature cannot be addressed by a straightforward extension of the weak MFG framework of \cite{carmona2016mean,carmona2018probabilistic2} and creates additional difficulties both in preserving compatibility under weak convergence and in proving that the limiting price process satisfies the desired consistency condition.}

Related works on information asymmetry in MFG include \cite{sen2016mean, fu2020mean, ma2018kyle, bertucci2022mean, bensoussan2021mean, casgrain2019trading, casgrain2020mean}, which study various aspects of heterogeneous information structures. More recently, \cite{bergault2024mean2, bergault2024mean} analyze different models in which informed agents exploit privileged information to influence the market. 
{To the best of our knowledge, equilibrium price formation in a mean-field market with heterogeneous information, where the additional source of randomness carried by the informed population does not vanish in the limit and enters the equilibrium equation through the adjoint process, has not been addressed in the existing literature.}

\subsection*{Main contributions and outline of the paper}

In \S \ref{section: setup}, we introduce the stochastic framework and the equilibrium conditions in the finite-dimensional setting, motivating these choices by the problem of price formation. We also recall some preliminaries about FBSDEs and Pontryagin's principle in \S \ref{section: FBSDE from SMP}.

In \S \ref{section: mfl}, we present and motivate the introduction of the mean-field limit of the system, which is the main object of this paper. The notion of weak lifted mean-field equilibrium for a problem with asymmetric information is given in Definition \ref{def: mf price process}. As mentioned above, the difference with the previous literature is that the equilibrium condition for the weak formulation involves the conditional expectation of the costate processes, which might be discontinuous; see \eqref{eqn: mc price process mf v2}. 
{This requires a careful treatment of the compatibility condition and leads to a lifted definition of equilibrium, in which the environment is not required to be adapted to the original common noises.
A key point is that this lifting is not merely formal: it is needed to prove that compatibility is preserved in the weak limit and to identify the limiting consistency condition. 
Since the natural mean-field market clearing condition \eqref{eqn: mc price process mf} appears difficult to solve in our framework, our main existence theorem is formulated for a weaker projected consistency condition \eqref{eqn: mc price process mf v2}. Accordingly, the weak equilibria of Definitions \ref{def: mf price process} and \ref{def: unlifted mf price process} are motivated by market clearing but are not exact market clearing equilibria in the economic sense.
Our main result is Theorem \ref{thm: existence of solutions}, which proves existence of weak lifted mean-field equilibria under the general Assumption \ref{assumption: coefficients} for this projected consistency condition. 
The main mathematical novelty of this result lies not only in the asymmetric information setting, but also in the fact that the equilibrium condition involves the adjoint processes and conditional expectations with respect to the equilibrium price itself.
Under the additional linearity structure of Assumption \ref{assumption: affine target functions} on the running and terminal cost functions, we also prove the existence of a weak unlifted equilibrium; see Theorem \ref{thm: existence of unlifted mean-field price process}. This is obtained by a projection argument which shows that, in this more special regime, the lift can be removed.
Assumption \ref{assumption: affine target functions} can be viewed as a technically convenient but economically more restrictive condition.
}

Section \ref{section: existence of solutions} is devoted to the proof of Theorem \ref{thm: existence of solutions}.
We analyze a discretization of the problem in \S \ref{section: discretization procedure} and prove existence of its solution via a fixed-point argument. Then in \S \ref{section: stability} we show tightness and stability under weak convergence, using the Meyer-Zheng topology and taking care of the compatibility condition.  
{A key point here is that the lifting must involve the adjoint component itself, because the interaction is through the costate rather than the state process.}
The consistency condition for the limit is shown in \S \ref{section: consistency condition} {through an argument tailored to the equilibrium notion considered here. }

{In \S \ref{section: weak market clearing}, by proving Theorem \ref{thm: weak mk clearing condition}, we obtain a conditional asymptotic justification of the mean-field equilibrium as an approximation of the finite-player market clearing relation.
This analysis is carried out under the stronger condition \eqref{eqn: project eq price v1} which is understood as a technical modification of Definition \ref{def: unlifted mf price process}, introduced in order to obtain the probabilistic structure needed for the asymptotic analysis.  
Accordingly, Section \ref{section: weak market clearing} should be interpreted as an asymptotic validation result under stronger assumptions, rather than as a full existence theory for exact market clearing equilibria.
}

Finally, in \S \ref{section: one informed agent}, we analyze a special case in which there is just one informed player in the market, showing how, in this case, her strategy can be explicitly described in terms of the equilibrium of the model and the common noise.

\subsection*{Notation}

For any probability space $(\Om,\G,\prob)$ endowed with a filtration $\mathbb{G}= (\G_t)_{t\in[0,T]}$ we refer to:
\begin{itemize}
\item $\mathbb{L}^2(\G)$ as the set of real-valued $\G$-measurable square integrable random variables;
\item $\mathbb{S}^2(\mathbb{G})$ as the set of real-valued $\mathbb{G}$-adapted \cadlag processes $X$ satisfying:
\begin{displaymath}
||X||_{\mathbb{S}^2}:= \E\bblq \sup_{t\in[0,T]} |X_t|^2 \bbrq^{\frac{1}{2}}<\infty;
\end{displaymath}
\item $\mathbb{H}^2(\mathbb{G})$ is the set of real-valued $\mathbb{G}$-progressively measurable processes $Z$ satisfying:
\begin{displaymath}
||Z||_{\mathbb{H}^2}:= \E\bblq \bbl \int_0^T |Z_t|^2 dt \bbr \bbrq^{\frac{1}{2}}<\infty.
\end{displaymath}
\end{itemize}
Note that we consider real-valued processes, but the whole analysis can also be done  for multidimensional processes, with slightly more complicated notation.  We also introduce the following notation: 
\begin{itemize}
\item For a constant $\Lambda>0$, we adopt the notation $\bar{\Lambda} := \Lambda^{-1}$. 
\item $\mathcal{C}([0,T],\R)$ denotes the space of real-valued continuous functions on $[0,T]$.
\item $\mathcal{D}([0,T],\R)$ denotes the space of real-valued \cadlag functions on $[0,T]$.
\item $\mathcal{M}([0,T],\R)$ denotes the set of real-valued measurable functions on $[0,T]$. We endow $\mathcal{M}([0,T],\R)$ with the Meyer-Zheng topology. The main properties of this space we are going to apply are described in Appendix \ref{Appendix:MZ}. 
\item $\mathcal{L}(X)$ is the law of a random variable $X:\Omega \to \mathcal{S}$, taking values on a Polish space $\mathcal{S}$.
\item $\F^\Theta := (\ItF^{\Theta}_t)_{t\in[0,T]}$ the complete and right-continuous augmentation of the natural filtration generated by the process $\Theta= (\Theta_t)_{t\in[0,T]}$. We refer to this as the usual augmentation. 
\end{itemize}

\section{The finite population setting and its financial motivation}
\label{section: setup}

We start by describing the financial problem which motivates our mathematical setup. As in \cite[Section 3.1]{fujii2022mean}, we consider a market model populated by $N$ agents, belonging to two populations: $N_S$ standard agents (called in the following standard agents) and $N_I$ informed agents. All agents trade a single asset. The goal of every agent is to solve an optimal control problem which depends on the price of the traded asset, denoted by $\w$. We suppose that the informed agents' revenues depend on an additional stochastic process $C$, which can be observed by them, but is inaccessible to the standard agents. The probabilistic setup is represented by a family of stochastic control problems defined as follows.

Index $I$ will be associated with the informed agents, while by index $S$ we refer to the standard agents. We consider a probability space $(\Omega, \mathcal{G}, \mathbb{P})$ supporting the following:
\begin{itemize}
\item For every $p = I,S$ and $j = 1,\dots,N_p$, a random variable $\xi^{p,j}$ and a Brownian motion $(W^{p,j}_t)_{t\in[0,T]}$ independent of $\xi^{p,j}$. The random variable $\xi^{p,j}$ represents the initial value of the state variable of the $j$-th agent of population $p$, while the Brownian motion represents the idiosyncratic noises associated with each agent; 
\item A Brownian motion $(B_t)_{t\in[0,T]}$ and a real-valued {càdlàg} stochastic process $C$, possibly correlated with $B$, which represents the private information of the informed agents. 
\item The processes $(B,C)$, $(\xi^{I,p}, W^{I,p})_{p=1}^{N_I}$ and $(\xi^{S,p}, W^{S,p})_{p=1}^{N_S}$ are independent. We denote by $\mathbb{G}$ the complete and right-continuous augmentation of the natural filtration generated by all noises and initial conditions: $\mathbb{G}=\F^{\xi^{I,1}, \dots, \xi^{I,N_I}, W^{I,1}, \dots, W^{I,N_I}, \xi^{S,1}, \dots, \xi^{S,N_S}, W^{S,1}, \dots, W^{S,N_S}, B,C}$.  
\end{itemize}

\noindent We denote by $\w$ the price process of the traded asset, assumed for the moment to be a generic adapted c\`{a}dl\`{a}g process ${\w}= ({\w}_t)_{t\in[0,T]}$, defined on $(\Omega,\mathcal{G}, \prob, \mathbb{G})$; as a consequence, note that it is adapted to the filtration generated by all the noises introduced above. Thus, we define in this section a \emph{strong formulation} of the economy with $N$ agents. 
We make the following standing assumption on the initial condition of the state variable of every agent, which is supposed to hold throughout the paper.
\begin{assumption}
\label{assumption: 4th moment X0 bound}
The vector $(\xi^{p,1},\dots, \xi^{p,N_p})$ forms a sequence of i.i.d. random variables, such that $\E[|\xi^{p,j}|^4]<\infty$ for every $j =1,\dots, N_p$, every $p = I,S$. 
\end{assumption}

We now describe the optimal control problems which are solved by the agents. At the level of the individual optimization problem, every agent is treated as a {\em price taker}, choosing her strategy by taking the price process $\varpi$ as an exogenous process. 
In the following, we may call the price $\w$ a random environment, in analogy with \cite[Chapter 1]{carmona2018probabilistic2}.  The equilibrium price is then determined through the market-clearing condition introduced below.

\subsection{The optimization problem of the standard agents}
\label{section: lia}
 We introduce the number of shares $X^{p,j}_t$ held by the $j$-th agent of population $p$ at time $t$. In the following, we will refer to this agent as agent $(p,j)$, for $j = 1,\dots,N_p$ and $p = I,S$. The state process $X^{p,j}$ is controlled by the agent through the trading speed, denoted $\alpha^{p,j}$. In particular, $\alpha^{p,j}_t$ represents the number of shares traded by agent $(p,j)$ in the infinitesimal time interval $[t,t+dt]$. In addition, the position $X^{p,j}$ depends also on the trades between the agent and its individual clients. Similarly as in \cite{fujii2022mean}, we suppose that every standard agent has a group of customers who trade the securities with the agent and have no direct access to the market. The random demand of the private customers of agent $(p,j)$ is described by the Brownian motion $W^{p,j}$, multiplied by a factor $\sigma^p$, possibly depending on the price process. Finally, we also allow for the existence of a common random demand which affects all the standard agents in the same way. It is described by a factor $\sigma^{p,0}$, possibly dependent on $\w$, multiplied by the Brownian motion $B$. 

Each standard agent's state dynamics is hence described by the following SDE:
\begin{equation}
\label{eqn: state variable standard}
\begin{cases}
dX^{S,j}_t = (\alpha^{S,j}_t + l^S(t,{\w}_t))dt + \sigma^{S,0}(t,{\w}_t)dB_t + \sigma^{S}(t,{\w}_t)dW^{S,j}_t,\\ 
X^{S,j}_0 = \xi^{S,j},
\end{cases}
\end{equation}
%
for every $j = 1,\dots,N_S$.
The cost functional to be minimized by the $j$-th standard agent is:
\begin{equation}
\label{eqn: objective functional standard}
J^{S}(\alpha^{S,j}):= \E\bblq \int_0^T f^S(t,X^{S,j}_t,\alpha^{S,j}_t,{\w}_t) dt + g^S(X^{S,j}_T,{\w}_T)\bbrq,\quad j=1,\dots,N_S,
\end{equation}
where 
\begin{equation*}
    f^S(t,x,a,{\w}):= {\w}a + \frac{1}{2}  \Lambda^S a^2 + \bar{f}^S(t,x,{\w}),
\end{equation*}
for measurable functions $\bar{f}^S$ and $g^S$ and a positive constant $\Lambda^S$. For every $j = 1,\dots, N_S$, the family of admissible controls is given 
\begin{equation} 
\label{eqn: controls space standard}
\alpha^{S,j} \in \mathbb{A}^{S,j}:=  \mathbb{H}^2(\F^{S,j}), \qquad 
\F^{S,j}:={\F}^{\xi^{S,j},\w,B,W^{S,j}}.
\end{equation}
Every standard agent observes the price process and the Brownian motions $B$, $W^{S,j}$, but is not aware of the presence of additional sources of randomness that are affecting the market.

\subsection{The optimization problem of the informed agents}
\label{section: mia}
The optimal control problem of the informed agents is defined in analogy to the one for the standard agents. The state variable satisfies 
\begin{equation}
\label{eqn: state variable major}
\begin{cases}
dX^{I,j}_t = (\alpha^{I,j}_t + l^I(t,\w_t)) dt + \sigma^{I,0}(t,\w_t) dB_t + \sigma^{I}(t,\w_t) dW^{I,j}_t ,\\
 X^{I,j}_0 = \xi^{I,j}.
 \end{cases}\quad j=1,\dots,N_I.
\end{equation}
As discussed, the gap in the information structure between the informed agents and the standard agents is described by the presence of an additional stochastic factor $C$ which affects the revenues of the informed agents. 
Accordingly, the cost functional to be minimized by the $j$-th informed agent is:
\begin{equation}
\label{eqn: objective functional major}
J^{I}(\alpha^{I,j}):= \E\bblq \int_0^T f^I(t,X^{I,j}_t,\alpha^{I,j}_t,{\w}_t,C_t) dt + g^I(X^{I,j}_T,{\w}_T,C_T)\bbrq,
\end{equation}
where 
\begin{equation*}f^I(t,x,a,{\w},c):= {\w}a + \frac{1}{2}  \Lambda^I a^2 + \bar{f}^I(t,x,{\w},c),
\end{equation*}
for measurable functions $\bar{f}^I$ and $g^I$ and a positive constant $\Lambda^I$. For every $j = 1,\dots, N_I$, the family of admissible controls is given by 
\begin{equation} 
\label{eqn: controls space major}
\alpha^{I,j} \in \mathbb{A}^{I,j}:=  \mathbb{H}^2(\F^{I,j}), \qquad \F^{I,j}:=\F^{\xi^{I,j},\w,B,C,W^{I,j}}.
\end{equation}
This implies that the informed agents can choose the trading strategy depending on the information released by the extra factor $C$, {which is assumed to be a càdlàg process}. 

\begin{remark}[Financial interpretation of the performance functionals]
{
The running cost function of both populations contains the linear term $\varpi a$, which represents the cashflow generated by buying or selling the asset at time $t$, and a quadratic term in $a$, which penalizes large trading rates and can be interpreted as a proxy for execution costs. The functions $\bar{f}^p$ and $g^p$, for $p=S,I$, which do not depend directly on the control variable, describe instead the position-dependent (or inventory, represented by the state variable $x$) component of the agent’s objective, namely the costs or revenues associated with holding the asset over time and at the terminal date. These functions may account for inventory costs, funding costs, hedging motives, and liquidation values.
For the informed agents, these terms may depend on the additional process $C$, which should be interpreted as an extra payoff-relevant factor observed only by that population. In our framework, the asymmetry of information does not necessarily concern the terminal fundamental value of the asset, but may also concern private information about how valuable it is for informed agents to hold a larger or smaller position. This is consistent with the model setup, where $C$ affects the revenues of informed agents and can be used in their admissible strategies. }

{We present two simple financial examples of this setup, which are motivated by the literature on funding constraints and dealer balance-sheet capacity (see, e.g., \cite{BP09,HK13}) and the literature on dealer risk management and hedging-induced trading (see, e.g., \cite{NY03}):
\begin{enumerate}
\item {\em Funding conditions of primary dealers.}
Consider a bond market with two populations: primary dealers (informed agents) and institutional investors (standard agents).
The informed population may represent a sector of primary dealers, with $C_t$ representing their privately observed funding conditions at time $t$, while the common noise $B_t$ represents market-wide shocks to the aggregate demand or supply of the asset.
High values of $C_t$ correspond to tighter intermediation capacity (so that carrying a position is more costly), while lower values indicate more favorable funding conditions. 
In this case, one may use
\[
\bar{f}^I(t,x,\varpi,c)=x\,c^I_{1}(t,\varpi,c), \qquad 
g^I(x,\varpi,c)=x\,c^I_{2}(\varpi,c),
\]
for suitable continuous bounded functions $c^I_{1}$ and $c^I_{2}$, increasing in $c$. 
When funding conditions are tight ($c$ high), the marginal cost of holding inventory is high and dealers optimally reduce their positions. When conditions improve ($c$ low), dealers can warehouse more inventory at lower costs. Hence, observing $C_t$ changes the marginal cost of inventory and, therefore, affects the dealers’ optimal trading strategy.
Through market clearing, primary dealer funding conditions are transmitted into the equilibrium price.
Note also that this example satisfies Assumption \ref{assumption: affine target functions} introduced in Section \ref{section: modification canonical space} below.
\item {\em Hidden hedging demand from an OTC derivatives book.}
The informed agents may represent dealers managing an over-the-counter (OTC) derivatives book not observed by the rest of the market. Let $C_t$ denote the delta-equivalent risk exposure of that private book, measured in units of the traded asset (i.e., $C_t$ is the number of units of the traded asset that would be needed to hedge the exposure). Dealers use the traded asset to hedge, so that it is costly to hold too many or too few units of the asset relative to the hedging target $C_t$. According to this interpretation, a natural specification is
\[
\bar{f}^I(t,x,\varpi,c)=\kappa_1 u(x-c), \qquad
g^I(x,\varpi,c)=\kappa_2  u(x-c)+ v(\varpi)x,
\]
where for instance $u(r) =r^2$ if $|r|\leq R$ and $u(r)=2Rr-R^2$ if $|r|\geq R$, $R>0$ is a big constant threshold, $\kappa_1,\kappa_2>0$ and $h$ is a continuous function.\footnote{
We consider $u$ of this form, instead of just quadratic, because it is convex, globally Lipschitz and continuously differentiable, in accordance with Assumption \ref{assumption: coefficients} below.
} 
In this case, $C_t$ does not represent information about the asset’s fundamental value, but rather a privately observed hedge target. Observing $C_t$ allows the informed agents to choose positions that reduce future hedging costs and, therefore, change their optimal trading strategy.
\end{enumerate}
}
\end{remark}

The following standard growth and regularity assumptions on the coefficients will be in force throughout the paper. Note that the volatility $\sigma^p$ and $\sigma^{p,0}$ might be degenerate, but we assume convexity in $x$ of the costs.  
\begin{assumption}
\label{assumption: coefficients}
There exists a constant $L>0$ such that:
\begin{enumerate}
\item The coefficients $l^p$, $\sigma^p$, $\sigma^{p,0}$ are Borel-measurable functions. For every $t\in[0,T]$ the functions $l^p(t,\cdot)$, $\sigma^p(t,\cdot)$ and $\sigma^{p,0}(t,\cdot)$ are continuous and for every $\w\in \R$
\begin{displaymath}
|l^p(t,\w)| + |\sigma^p(t,\w)| + |\sigma^{p,0}(t,\w)| \leq L(1 + |\w|).
\end{displaymath} \label{A1}
\item For every $t\in [0,T]$ and for $p = I,S$ the coefficients $\bar{f}^S(t,\cdot,\cdot)$, $\bar{f}^I(t,\cdot,\cdot,\cdot)$, $g^S(\cdot,\cdot)$ and $g^I(\cdot,\cdot,\cdot)$ are continuous functions. For every $t\in[0,T]$, $c\in\R$, $\w\in\R$, the functions $\bar{f}^S(t,\cdot,\w)$, $\bar{f}^I(t,\cdot,\w,c)$ and $g^S(\cdot, \w)$, $g^I(\cdot,\w,c)$ are continuously differentiable. Moreover
\begin{displaymath}
|\bar{f}^S(t,x,\w)| + |g^S(x,\w)| + |\bar{f}^I(t,x,\w,c)| + |g^I(x,\w,c)|	\leq L(1 + |x| + |\w|^2+ |c|^2).
\end{displaymath}\label{A2}
\item The functions $\bar{f}^p$ and $g^p$ are convex in the $x$ variable. \label{A3}
\item The functions $l^S(t,\w)$, $l^I(t,\w)$, $\partial_x \bar{f}^S(t,x,\w)$, $\partial_xg^S(x,\w)$, $\partial_x \bar{f}^I(t,x,\w,c)$ and $\partial_xg^I(x,\w,c)$ are continuous. \label{B1}

\item For every $t\in[0,T]$ and for every $c,\w\in\R$, it holds that
\begin{equation*}
|\partial_x \bar{f}^S(t,x,\w)| + |\partial_xg^S(x,\w)|  + |\partial_x \bar{f}^I(t,x,\w,c)| + |\partial_xg^I(x,\w,c)| \leq L.
\end{equation*}  \label{B2}

\item For every $t\in[0,T]$ and for every $x,x',\w,c\in\R$, it holds that
\begin{equation*}
\begin{aligned}
|\partial_x\bar{f}^S(t,x',\w) -\partial_x\bar{f}^S(t,x,\w)| 	 + |\partial_x\bar{f}^I(t,x',\w,c) -\partial_x\bar{f}^I(t,x,\w,c)| &\leq L|x-x'|,\\
|\partial_xg^S(x,\w) - \partial_x g^S(x',\w)|+|\partial_xg^I(x,\w,c) - \partial_x g^I(x',\w,c)|				&\leq L|x-x'|.
\end{aligned}
\end{equation*} \label{B3} 
\end{enumerate}
\end{assumption} 

\subsection{The FBSDE associated with the stochastic maximum principle}
\label{section: FBSDE from SMP}
To determine the candidate optimal controls, we aim at applying the stochastic maximum principle. As discussed in \cite[Section 1.4]{carmona2018probabilistic2}, the application of the stochastic maximum principle requires a compatibility condition, which can be defined in terms of the immersion property.
\begin{defin}[Compatibility]
\label{def: compatibility}
On a probability space $(\Omega,\ItF,\prob)$, let us consider two filtrations $(\ItF_t)_{t\in[0,T]}$ and $(\G_t)_{t\in[0,T]}$. We say that $(\ItF_t)_{t\in[0,T]}$ is \emph{immersed} in $(\G_t)_{t\in[0,T]}$ if
\begin{itemize}
\item $\ItF_t\subseteq\G_t$ for all $t\in[0,T]$.
\item martingales with respect to $(\ItF_t)_{t\in[0,T]}$ remain martingales with respect to $(\G_t)_{t\in[0,T]}$.
\end{itemize}
A stochastic process $\theta$, defined on a probability space $(\Omega,\ItF,\prob)$, is \emph{compatible} with a filtration $\F:= (\ItF_t)_{t\in[0,T]}$ if the natural filtration $\F^\theta$ generated by $\theta$ is immersed in $\F$. 
\end{defin}
 We recall that the immersion of $(\ItF_t)_{t\in[0,T]}$ in $(\G_t)_{t\in[0,T]}$ is equivalent to requiring that
\begin{equation}
\label{eqn: immersion}
\ItF_T \text{ is conditionally independent of } \G_t \text{ given } \ItF_t.
\end{equation}
For a thorough discussion on the notion of compatibility and its equivalent definitions we refer to \cite[\S 1.1]{carmona2018probabilistic2}.
The definition of compatibility is crucial to apply the stochastic maximum principle in the version of \cite[Theorem 1.60]{carmona2018probabilistic2}, which is based on the following definition of admissibility.
\begin{defin}
\label{def: admissibility}
A probability space $\probspace$, endowed with a complete and right-continuous filtration $\F$ and on which a process $\theta:= (\xi, W, \chi)$ is defined, is an \emph{admissible setup} if
\begin{enumerate}
\item $\xi$ is $\mathcal{F}_0$-measurable and independent of $(W,\chi)$;
\item $W$ is an $\F$-Brownian motion;
\item {$\chi \in \mathbb{S}^2(\mathbb{F})$ and} $\theta$ is compatible with $\F$.
\end{enumerate}
We refer to the term admissible probabilistic setup by $(\probspace, \F, \theta)$. 
\footnote{Note that this definition is slightly different from \cite[Def. 1.13]{carmona2018probabilistic2}, as here the environment is not independent of the idiosyncratic noise. However, all the results in that book apply in the same way, provided that $W$ remains a Brownian motion for the filtration $\mathbb{F}$. 
}

\end{defin} 
For convenience of notation, let us denote
\begin{equation}
\label{eqn: chip} 
\chi^{p,\w} := 
\begin{cases}
(\w,C), \quad &\text{if } p = I;\\
\w,\quad &\text{if }p = S.
\end{cases} 
\end{equation}
We state the following assumption, which is in fact an assumption on the joint distribution of $(B,W^{p,j},C)$ : 
\begin{assumption}
\label{assumption: admissibility} 
The probabilistic setups 
\begin{align*}
&((\Omega,\ItF,\prob), \F^{p,j}, (\xi^{p,j}, (B,W^{p,j}) ,\chi^{p,\w}))
\end{align*}
 are admissible in the sense of Definition \ref{def: admissibility} for $p = I,S$.
\end{assumption} 

\begin{remark}
\label{rmk: compatibility example}
Assumption \ref{assumption: admissibility} implies some conditions on the couple $(B,C)$:  since $B$ and $C$ are not assumed to be independent, some relation must hold  between $B$ and $C$, in order to guarantee that $B$ is a $\F^{I,j}$-Brownian motion and $C$ does not reveal future realizations of $B$.
\end{remark}

Under Assumption \ref{assumption: admissibility}, we introduce the reduced Hamiltonians $H^p$ of the stochastic optimal control problem for the agent of population $p\in\{I,S\}$
\begin{equation*}
H^{p}(t,x,y, a, \chi^{p,\w}) = y ( a+ l^p(t,\w)) + f^p(t,x,a,\chi^{p,\w}).
\end{equation*}
Due to Assumption \ref{assumption: coefficients}, the reduced Hamiltonian is convex in $a$ and there exists a unique minimizer, having the following structure
\begin{equation}
\label{eqn: optimal controls}
\hat{\alpha}^{p}(\w,y):= -\bar{\Lambda}^p(y+\w),\quad p = I,S.
\end{equation}
 We now introduce the following FBSDE systems, for $p = I,S$: 
\begin{equation}
\label{eqn: FBSDE:N}
\begin{cases}
dX^{p,j}_t = (-\bar{\Lambda}^p(Y^{p,j}_t+\w_t) + l^p(t,{\w}_t))dt + \sigma^{p,0}(t,{\w}_t)dB_t + \sigma^p(t,{\w}_t)dW^{p,j}_t,&\\ 
X^{p,j}_0 = \xi^{p,j},&\\ 
dY^{p,j}_t = - \partial_x \bar{f}^p(t,X^{p,j}_t, \chi^{p,\w}_t)dt + Z^{p,0,j}_t dB_t + Z^{p,j}_t dW^{p,j}_t + dM^{p,j}_t,&\\ 
Y^{p,j}_T= \partial_x g^p(X^{p,j}_T, \chi^{p,\w}_T),
\end{cases}
\end{equation}
where $M^{p,j}$ denotes a \cadlag martingale. 
Let us remark that $M^{p,j}$ appears {because the underlying filtration $\F^{p,j}$ (given by \eqref{eqn: controls space standard} and \eqref{eqn: controls space major}) might be larger than the Brownian filtration generated by $(B,W^{p,j})$, due to the presence of the random environment $\varpi$ and, for the informed agents, $C$.}
The process $Y^{p,j}$ is called the \emph{adjoint process associated with the optimal control problem}, and it might be discontinuous as well. By the compatibility condition ensured by Assumption \ref{assumption: admissibility},  and by the convexity of $H^p$, we can apply the stochastic maximum principle in the version of \cite[Theorem 1.60]{carmona2018probabilistic2} for optimal control problems in a random environment. 

We remark that the control problem for agent $j$ depends on $j$ only through the idiosyncratic noise $W^{p,j}$, and does not depend on the idiosyncratic noises of the other agents. We thus state a general solvability result for \eqref{eqn: FBSDE}, which also implies uniqueness of the optimal control. This result  will be used several times in the rest of the paper, with respect to different admissible probabilistic setups.

\begin{proposition}
\label{lemma: iter random env}
Suppose that Assumption \ref{assumption: coefficients} holds. Fix $p=I,S$ and let
\[
(\Omega,\ItF,\prob, \F^{p}, (\xi^{p}, (B,W^{p}) ,\chi^{p,\w}))
\]
be an admissible probabilistic setup, in the sense of Definition \ref{def: admissibility}. Then the following hold:
\begin{itemize}
    \item[(i)] The FBSDE 
    \begin{equation}
\begin{cases}
\label{eqn: FBSDE}
dX^{p}_t = (-\bar{\Lambda}^p(Y^{p}_t+\w_t) + l^p(t,{\w}_t))dt + \sigma^{p,0}(t,{\w}_t)dB_t + \sigma^p(t,{\w}_t)dW^{p}_t, \quad  
X^{p}_0 = \xi^{p},\\ 
dY^{p}_t = - \partial_x \bar{f}^p(t,X^{p}_t, \chi^{p,\w}_t)dt + Z^{p,0}_t dB_t + Z^{p}_t dW^{p}_t + dM^{p}_t, \quad  
Y^{p}_T= \partial_x g^p(X^{p}_T, \chi^{p,\w}_T),
\end{cases}
\end{equation}
    admits a unique $\F^p$-adapted solution
\[
(X^{p}, Y^{p}, Z^{p,0}, Z^{p}, M^{p}) \in (\mathbb{S}^2(\F^p) )^2 \times (\mathbb{H}^2(\F^p) )^2 \times \mathbb{S}^2(\F^p).
\]
 $X^p$ is continuous and $M^p$ is a martingale orthogonal to 
$(\int_0^t Z^{p,0}_s dB_s + \int_0^tZ^{p}_s dW^{p}_s )_{t\in[0,T]}$ with $M_0=0$. 
Furthermore, the solution processes are adapted to $\F^{\xi^{p}, B,W^{p} ,\chi^{p,\w}}$.

\item[(ii)] The optimal control problem for a standard (if $p=S$) agent, given by \eqref{eqn: state variable standard}-\eqref{eqn: objective functional standard},  or informed (if $p=I$) agent, given by \eqref{eqn: state variable major}-\eqref{eqn: objective functional major}, defined over controls in $\mathbb{H}^2(\F^p)$, admits a unique optimal control in $\mathbb{H}^2(\F^p)$, given by 
\[
\hat{\alpha}^{p}_t := 
\hat{\alpha}^p(\w_t, Y^{p}_t) = 
-\bar{\Lambda}^p(Y^{p}_t + \w_t), \qquad 
t\in[0,T].
\]

\item[(iii)] Fix $t\in [0,T]$ and consider $x_1,x_2\in\R$
and the same admissible probabilistic setup, but on the interval $[t,T]$ and deterministic initial condition; denote by 
$(X^{p,l},Y^{p,l},Z^{p,0,l},Z^{p,l},M^{p,l})$ the solution to the FBSDE as in (i), for $l=1,2$, with initial condition $X^{p,1}_t=x_1$ and $X^{p,2}_t=x_2$. 
Then, there exists a positive constant $\Gamma_p>0$ depending only on the constants in Assumption \ref{assumption: coefficients} and on $T$ 
such that 
\begin{equation}
\label{eqn: iterrandom}
\mathbb{P}\left(|Y^{p,1}_t - Y^{p,2}_t |\leq \Gamma_p |x_1-x_2|\right) = 1.
\end{equation}
\end{itemize} 
\end{proposition}

\begin{proof} 
This is basically a consequence of \cite[Theorem 1.60]{carmona2018probabilistic2}. For completeness, the central point, which is the Lipschitz continuity of the decoupling field, is proven in Appendix \ref{appendix: proof iter random env}.
\end{proof} 

For $p=I,S$ and any $j$, in the rest of this section we let 
$(X^{p,j},Y^{p,j},Z^{p,0,j},Z^{p,j},M^{p,j})$ be the solution to FBSDE \eqref{eqn: FBSDE:N} in the admissible setup $(\Omega,\ItF,\prob, \F^{p,j}, (\xi^{p,j}, (B,W^{p,j}) ,\chi^{p,\w}))$ and let $\hat{\alpha}^{p,j}_t :=-\bar{\Lambda}^p(Y^{p,j}_t + \w_t)$  be the corresponding optimal control.

\subsection{The market clearing condition} 
\label{section: mcc}
We have so far assumed that the price process $\w$ is exogenously given, i.e., every agent considers $\w$ as an exogenous stochastic process which affects the state variable and the objective functional. Under this assumption, we have a family of stochastic optimal control problems which can be solved separately by each agent. As in \cite{fujii2022mean}, we aim at deriving an equation for the equilibrium price process, i.e., the price determined by the \emph{market clearing condition}, which is defined by the balance between demand and supply, and then show the existence of a solution to that equation. In addition, we aim at understanding if the standard agents can deduce some information regarding the strategy of the informed agents through the observation of the price process. In fact, the strategy of the informed agents has an impact on the equilibrium price, so we expect the latter to reveal part of the information embedded in $C$.

{\em Market clearing} in the finite-player model described at the beginning of this section is expressed by the following condition:
\begin{equation}
\label{eqn: market clearing N player}
 \sum_{h= 1}^{N_I} \hat{\alpha}^{I,h}_t + \sum_{j= 1}^{N_S} \hat{\alpha}^{S,j}_t = 0,\quad \text{dt}\otimes \prob\text{-a.e.}
\end{equation}
We now present the following result, which establishes the relationship between the information available to the informed population and that of the standard one, in terms of the filtration generated by the equilibrium price and the common noise. Its proof is straightforward.
\begin{lemma}
\label{lemma: lemma on intersection of filtrations}
In the setting introduced above, the following holds
\begin{displaymath}
\ItF^{\underline{\xi}^S,\w,B,\underline{W}^S}_t\wedge \ItF^{\underline{\xi}^I,\w,B,C,\underline{W}^I}_t = \ItF^{\w,B}_t,\quad t \in[0,T],
\end{displaymath}
where $\underline{W}^p := (W^{p,1},\dots,W^{p,N_p})$ is a Brownian motion and $\underline{\xi}^p=(\xi^{p,1},\dots,\xi^{p,N_p})$.
\end{lemma}

\begin{remark}
\label{rmk: measurability of empirical mean}
We recall that $\hat{\alpha}^{p,j}\in \mathbb{H}^2(\F^{p,j})$. Assuming that agents are rational, the market clearing condition \eqref{eqn: market clearing N player} implies that $\sum_{h= 1}^{N_I} \hat{\alpha}^{I,h}_t$ must a posteriori have the same measurability properties as $\sum_{j= 1}^{N_S} \hat{\alpha}^{S,j}_t$. 
Hence, $\sum_{h= 1}^{N_I} \hat{\alpha}^{I,h}_t$ is adapted to $(\ItF^{\underline{\xi}^S,\w,B,\underline{W}^{S}}_t\wedge \ItF^{\underline{\xi}^I,\w,B,C,\underline{W}^{I}}_t)_{t\in[0,T]}$. As a consequence, equation \eqref{eqn: market clearing N player}, together with Lemma \ref{lemma: lemma on intersection of filtrations}, enables us to conclude that $(\sum_{h= 1}^{N_I} \alpha^{I,h}_t)_t$ is adapted to $\F^{\w,B}$, whenever $\w$ satisfies the market clearing condition. 

The market clearing condition \eqref{eqn: market clearing N player}, applied with the optimal controls introduced in \eqref{eqn: optimal controls}, leads to an equation for the equilibrium price process. Indeed, it holds that
\begin{displaymath}
0 = \sum_{h= 1}^{N_I} \hat{\alpha}^{I,h}_t + \sum_{j= 1}^{N_S} \hat{\alpha}^{S,j}_t= -\sum_{h= 1}^{N_I} \bar{\Lambda}^I(Y^{I,h}_t + \w_t) - \sum_{j= 1}^{N_S} \bar{\Lambda}^S(Y^{S,j}_t + \w_t) 
\end{displaymath}
This condition, in turn, leads to the following equation for the equilibrium price process: 
\begin{equation}
\label{eqn: mc price process N}
\w_t = -(n_I\bar{\Lambda}^I + n_S\bar{\Lambda}^S)^{-1} \bl \frac{n_I}{N_I}\bar{\Lambda}^I\sum_{h= 1}^{N_I}Y^{I,h}_t + \frac{n_S}{N_S} \bar{\Lambda}^S\sum_{j= 1}^{N_S}Y^{S,h}_t\br,
\end{equation}
where $n_I = \frac{N_I}{N_I + N_S}$ and $n_S = 1-n_I$.

We point out that, as discussed above, $\sum_{h= 1}^{N_I}Y^{I,h}_t$ is  adapted to $\F^{\w,B}$, if $\varpi$ satisfies  \eqref{eqn: mc price process N}. 
{Since the optimal controls are given by $\hat\alpha^{I,h}_t=-\Lambda^I(Y^{I,h}_t+\varpi_t)$, this means that, if market clearing holds, the average candidate optimal control of the informed population can be inferred from the observation of the equilibrium price process $\varpi$.}
\end{remark}

\section{The mean-field model}
\label{section: mfl}
As a consequence of the market clearing condition, an equilibrium price defined as solution to \eqref{eqn: mc price process N} makes the optimal control problems introduced in \S \ref{section: setup} highly recursive. The complexity of the problem is due to the presence of the idiosyncratic noises as well as the asymmetry in the information. To overcome the issue related to the presence of the idiosyncratic noises, we exploit the fact that the agents are price takers and symmetric within the same population (i.e., agents within the same population solve the same optimal control problem, up to their pairwise independent idiosyncratic terms). This implies that the effect of the trading activities of each single agent becomes negligible when $N$ becomes large. 

{We give here a formal justification of} the mean-field limit of \eqref{eqn: mc price process N}. To this end, we apply the Yamada-Watanabe representation theorem in the version of \cite[Theorem 1.33]{carmona2018probabilistic2}, which states that there exist progressively measurable functions 
\begin{equation}
\label{eqn: YW functions}
\begin{cases}
\Psi^S: \R \times \mathcal{C}([0,T],\R^2) \times \mathcal{D}([0,T],\R)\to\mathcal{C}([0,T],\R) \times\mathcal{D}([0,T],\R)\times\mathcal{C}([0,T],\R^2)\times\mathcal{D}([0,T],\R),\\
\Psi^I: \R \times \mathcal{C}([0,T],\R^2) \times \mathcal{D}([0,T],\R^2)\to\mathcal{C}([0,T],\R) \times\mathcal{D}([0,T],\R)\times\mathcal{C}([0,T],\R^2)\times\mathcal{D}([0,T],\R) ,\\
\end{cases}
\end{equation}
such that the adjoint process $Y^{p,j}$ of the $j$-th agent of population $p$, with respect to an exogenously given price process $\w$, satisfies
\begin{equation}
\label{eqn: YW representation of FBSDE}
\Big(X^{p,j},Y^{p,j} , {\Big(\int_0^t(Z^{p,0,j}_s,Z^{p,j}_s) ds \Big)_{0\leq t\leq T} },M^{p,j}\Big)= \Psi^p(\xi^{p,j}, B, W^{p,j}, \chi^{p,\w}),\qquad p = I,S.
\end{equation}
{The two-dimensional component in (3.2) corresponds to the time integrals of $Z^{p,0,j}$ and $Z^{p,j}$.}
Since $(\xi^{p,j})_{j\in N_p}$ is a sequence of i.i.d. random variables and $(W^{p,j})_{j\in N_p}$ is a sequence of pairwise independent Brownian motions, independent of $(\xi^{p,j})_{j\in N_p}$, the sequences $(\xi^{p,j},B, W^{p,j}, \chi^{p,\w})_{j\in\N}$ are exchangeable (see, e.g.,  \cite[Definition 12.1]{klenke2013probability}). Therefore, the sequences $(Y^{p,j}_t)_{j = 1}^{N_p}$ defined by \eqref{eqn: YW representation of FBSDE} are exchangeable too. Applying De Finetti's representation theorem, we have
\begin{displaymath}
\lim_{N_p\rightarrow \infty} \frac{1}{N_p}\sum_{h = 1}^{N_p} Y^{p,h}_t = \E\bblq Y^{p,1}_t\Biggm| \bigcap_{j\geq 1} \sigma\{Y^{p,k}_t,\ k\geq j \}\bbrq\quad \mathbb{P}-a.s.
\end{displaymath}
{This suggests that the tail sigma-algebra $\bigcap_{j\geq 1} \sigma\{Y^{p,k}_t,\ k\geq j \}$ should be identified with the sigma-algebra generated by the common stochastic factors of the random variables $Y^{p,k}_t$, namely $\varpi_{\cdot\wedge t}$ and $B_{\cdot\wedge t}$ for the standard agents, and $\varpi_{\cdot\wedge t}$, $B_{\cdot\wedge t}$ and $C_{\cdot\wedge t}$ for the informed agents.}
As a consequence, it seems natural to suppose that
\begin{align*} 
\lim_{N_S \to\infty} \frac{1}{N_S}\sum_{h = 1}^{N_S} Y^{S,h}_t = \E\blq Y^{S,1}_t\bigm| \ItF^{\w,B}_t\brq, \qquad 
\lim_{N_I \to\infty} \frac{1}{N_I}\sum_{h = 1}^{N_I} Y^{I,h}_t = \E\blq Y^{I,1}_t\bigm| \ItF^{\w,B,C}_t\brq.
 \end{align*}
By relying on this substitution, we can consider a limit market model populated by a single typical standard agent and a typical informed agent. 
{Assuming $N_p/N \to n_p$, for $p=I,S$}, we can pass to the mean-field limit of equation \eqref{eqn: mc price process N}, in a suitable probabilistic setup $\setup$: 
\begin{equation}
\label{eqn: mc price process mf} 
\w_t = -(n_I\bar{\Lambda}^I + n_S\bar{\Lambda}^S)^{-1} \bl n_I\bar{\Lambda}^I\E\blq Y^{I}_t\bigm|\ItF^{\w,B,C}_t\brq + n_S \bar{\Lambda}^S\E\blq Y^{S}_t\bigm| \ItF^{\w,B}_t\brq\br,\   \forall t\in[0,T],\  \mathbb{P}-a.s.
\end{equation}
In \eqref{eqn: mc price process mf}, $Y^p$ is the adjoint process associated with the optimal control problem of a typical agent of population $p\in\{I,S\}$ in the mean-field limit.

We can observe that, in analogy to Remark \ref{rmk: measurability of empirical mean}, the observation of $\w$ given as a solution of \eqref{eqn: mc price process mf} allows the typical standard agent to infer $\E[ Y^{I}_t\mid\ItF^{\w,B,C}_t]$. Moreover, taking conditional expectation with respect to $\ItF^{\w,B}_t$ in \eqref{eqn: mc price process mf}, we deduce that when a solution $\w$ to \eqref{eqn: mc price process mf} exists, then it holds that
\begin{equation}
\label{eqn: mc price process mf v2} 
\w_t = -(n_I\bar{\Lambda}^I + n_S\bar{\Lambda}^S)^{-1} \E\blq n_I\bar{\Lambda}^I Y^{I}_t+n_S\bar{\Lambda}^S  Y^{S}_t\bigm| \ItF^{\w,B}_t\brq, \qquad  \forall t\in [0,T], \quad \mathbb{P}-a.s.
\end{equation}

\begin{remark}
{Equation \eqref{eqn: mc price process mf} is the natural mean-field counterpart of the finite-player market clearing relation, since it retains the dependence of the informed population on the  filtration $\mathbb F^{\varpi,B,C}$. In contrast, \eqref{eqn: mc price process mf v2} is obtained by conditioning \eqref{eqn: mc price process mf} with respect to the smaller filtration $\mathbb F^{\varpi,B}$ and is introduced because it is more tractable for the analysis developed below. 
Indeed, as explained in Remarks \ref{rmk: extra assumption for asymptotic mk clearing} and \ref{rem: 6.3}, existence of a price process satisfying \eqref{eqn: mc price process mf} appears out of reach with the present method. 
If $\varpi$ satisfies \eqref{eqn: mc price process mf v2}, then the gap between the right-hand side of \eqref{eqn: mc price process mf} and $\varpi_t$ is
\[
-(n_I\Lambda^I+n_S\Lambda^S)^{-1} n_I\Lambda^I
\Big(\E\blq Y_t^I\bigm| \mathcal F_t^{\varpi,B,C}\brq -
\E\blq Y_t^I\bigm| \mathcal F_t^{\varpi,B}\brq
\Big).
\]
Therefore, the gap between \eqref{eqn: mc price process mf} and \eqref{eqn: mc price process mf v2} is entirely due to the content of the informed adjoint that is revealed by the private signal $C$ but not by the public filtration $\mathbb F^{\varpi,B}$. 
In this sense, \eqref{eqn: mc price process mf v2} should be viewed as the public-information projection of \eqref{eqn: mc price process mf} and the price process obtained from Definition \ref{def: mf price process} need not satisfy \eqref{eqn: mc price process mf}.} 
\end{remark}

\subsection{Definition of solution}
\label{section: definition of solution} 
We aim at proving the existence of a solution to \eqref{eqn: mc price process mf v2}. Hence, we introduce a general setup for defining a mean-field equilibrium, which is the main object of our analysis. 
We proceed as follows: 
\begin{enumerate}[label=\textbf{S-\Roman*},ref={S-\Roman*}]
\item \label{item: ch2 step 1 fixed point procedure} We consider a filtered probability space $\setup$, of the form
\begin{equation}
\label{eqn: Omega}
\Omega :=  \Om^0 \times \Om^I\times\Om^S,
\end{equation}
carrying a stochastic process $\w$, an initial distribution $\xi^p$ and a Brownian motion $W^p$, defined on $\Omega^p$ for the typical agent of population $p=I,S$ and a Brownian motion $B$ defined on $\Omega^0$. We assume that there exist two sub-filtrations  $\F^I$ and $\F^S$ such that the probabilistic setups $\setupi$ and $\setupn$ are admissible. In view of Proposition \ref{lemma: iter random env}, on these probabilistic setups the optimal control problems of \S \ref{section: lia} and \S \ref{section: mia} admit a solution, denoted by 
\begin{equation}
\label{eqn: solutions FBSDEs}
(\Omega, \ItF, \prob, \F^p, X^p ,Y^p ,Z^{p,0}, Z^p ,M^p),\quad p = I,S. 
\end{equation}
\item \label{item: ch2 step 2 fixed point procedure} We consider the functional $\Phi$ given by  
\begin{equation}
\label{eqn: fixedpointfunctional}
\Phi_t(\w) := -(n_I\bar{\Lambda}^I + n_S\bar{\Lambda}^S)^{-1} \E\blq n_I\bar{\Lambda}^I Y^{I}_t+ n_S\bar{\Lambda}^S  Y^{S}_t\bigm| \ItF^{\w,B}_t\brq,\quad t\in[0,T].
\end{equation}
\end{enumerate}
We aim at proving the existence of a stochastic process $\wmf$ such that: $\Phi_t(\wmf)=\wmf_t$ $\mathbb{P}$-a.s. for every $t\in[0,T]$.
The existence of a fixed point is not straightforward due to the dependence of both $Y^I$ and $Y^S$ on the unknown stochastic process $\wmf$,  as well as the presence of $\wmf$ in the filtration. To overcome this difficulty, we exploit an analogy between the structure of equation \eqref{eqn: mc price process mf v2} and the consistency condition for a weak mean-field game equilibrium in the presence of common noise. This condition, introduced in \cite[Definition 3.1]{carmona2016mean} (see also \cite[Definition 2.24]{carmona2018probabilistic2} and \cite[Definition 2.1]{lacker2016general}) is defined on a canonical space carrying suitable random processes, by imposing that the equilibrium $\mu$ is defined by: 
\begin{equation}
\label{eqn: Lacker weak mf eq}
\mu = {\mathcal{L}}(W,X \mid B, \mu),
\end{equation}
where $W$ is the idiosyncratic noise, $X$ is the optimal state variable and $B$ is the common noise. 

Building upon the results  in \cite[Chapter III]{carmona2018probabilistic2}, which are related to \cite{carmona2016mean}, we develop a forward-backward formulation for the optimal control problems of each agent, extending their applicability to our specific framework. In \cite[Chapter III]{carmona2018probabilistic2}, the construction of the solution to \eqref{eqn: Lacker weak mf eq} is performed by discretizing the canonical space on which the common noise is defined, in order to obtain a sequence of approximated solutions. Following this approach, we consider a random process $\w$ taking values on a suitable functional space $\mathcal{H}$. As discussed at the beginning of this section, if $\w$ is supposed to be exogenously given, the solutions of the optimal control problems exist and are determined by a progressively measurable functional $\Psi^p$ (introduced in equation \eqref{eqn: YW functions}) such that the processes $\FBSDEn$ and $\FBSDEi$, defined in \eqref{eqn: solutions FBSDEs}, can be expressed as in \eqref{eqn: YW representation of FBSDE} where $(\xi^{p,j} ,W^{p,j})$ is replaced by $(\xi^p,W^p)$ for $p\in\{I,S\}$. 
Due to the presence of the \cadlag martingales terms in the dynamics of $Y^I$ and $Y^S$, we consider $\mathcal{H}=\mathcal{D}([0,T],\R)$. We aim at finding a fixed point of the functional \eqref{eqn: fixedpointfunctional} defined on the space $\mathcal{D}([0,T],\R)^{\Omega^0}$. As discussed in \cite[Chapter II]{carmona2018probabilistic2}, when $\Om^0$ is not countable, the compact sets of $\mathcal{D}([0,T],\R)^{\Om^0}$ cannot be easily characterized. As a consequence, we cannot use standard fixed-point arguments, like Schauder's theorem, to provide a solution to equation \eqref{eqn: mc price process mf v2}. To overcome this problem, we adapt the strategies described in \cite[Chapter III]{carmona2018probabilistic2} and \cite[Section 3]{carmona2016mean} and proceed as follows:
\begin{enumerate}
\item We consider an admissible probabilistic setup of the form \eqref{eqn: Omega} on which a random process $(\xi^I,\xi^S, B,W^I,W^S,C)$ is defined. {We discretize the trajectories of the common noise $B$ in space and time, with parameters $l$ and $n$, respectively.
We then construct a fixed point on a finite product of copies of $\mathcal{D}([0,T],\mathbb R)$, indexed by the discretized paths of $B$. By applying Schauder's fixed point theorem, we can then construct a solution to the fixed point problem.}
\item We apply the previous step for each $n$ and $l$ to obtain a sequence of approximated solutions. We then prove that this sequence is tight and determine the conditions which ensure that the weak limit of this sequence solves \eqref{eqn: mc price process mf v2}. 
\end{enumerate}
\begin{remark}[Compatibility condition]
\label{rmk: problem compatibility condition}
The procedure described above to construct the solution to equation \eqref{eqn: mc price process mf v2} involves the issue of compatibility. Indeed, as described in \cite[Section 2.2.2]{carmona2018probabilistic2}, it is not sufficient to require the compatibility condition for the optimal control problems in the discretized setting, because compatibility is in general not preserved when passing to the limit in distribution. As we show in \S \ref{section: compatibility condition weak limit}, we have to lift the sequence of fixed points obtained in the discretized space in a suitable way, in order to ensure that compatibility is preserved in the weak limit. As a consequence, we can formulate the optimal control problems for the two typical agents within the space where the weak limit is defined. This is essential to show that the weak limit can be expressed in terms of the adjoint processes derived from the stochastic maximum principle applied to the optimal control problem on that space.
 \end{remark}
 
In view of Remark \ref{rmk: problem compatibility condition}, we need to change the structure of the optimal control problems considered in \ref{item: ch2 step 1 fixed point procedure}. We must enlarge the filtrations to which the controls of the two typical agents are adapted in order to ensure the compatibility condition.
%

{We are now in the position to introduce the notion of weak {\em lifted} mean-field equilibrium. We emphasize that this notion is not the direct mean-field counterpart of exact finite-player market clearing and, therefore, should not be interpreted as an exact market clearing equilibrium in the economic sense. Rather, it is a weak formulation motivated by the market clearing relation and chosen so as to admit an existence theory in a probabilistic weak sense. In particular, the consistency condition \eqref{eqn: mc price process mf v2} required in Definition \ref{def: mf price process} should be viewed as a tractable relaxation of condition \eqref{eqn: mc price process mf}, for which existence is not established in our work (see also Remark \ref{rem: 6.3}).}

\begin{defin} 
\label{def: mf price process}
We say that
\begin{equation}
(\Omega, \ItF, \prob, \F, (\xi^I, \xi^S), (B,W^I,W^S), (\w,C, \bar{Y}^I, \bar{Y}^S))
\end{equation}
is a \emph{weak lifted mean-field equilibrium} if the following four conditions hold:
\begin{enumerate}
\item $\F := \F^I \vee \F^S$, where 
\begin{equation}
\label{eqn: lifted filtrations}
\F^I := \F^{\xi^{I},(\w,\bar{Y}^I),B,C,W^{I}} 
\quad\text{ and }\quad
\F^S := \F^{\xi^S,(\w, \bar{Y}^S), B, W^{S}}.
\end{equation}
\item For $p=I,S$, $(\Omega, \ItF, \prob, \F^p)$ is an admissible probabilistic setup in the sense of Definition \ref{def: admissibility}, carrying the process 
\begin{equation*}
\theta^p := ( \xi^p, (B,W^p), (\chi^{p,\w},\bar{Y}^p)).
\end{equation*}
\end{enumerate}
In these admissible setups, let $(X^p, Y^p,Z^{p,0}, Z^p, M^p)$ be the solution to the FBSDE \eqref{eqn: FBSDE}, provided by Proposition \ref{lemma: iter random env}.\footnote{
Note that there is no additional ($\overline{Y}^S, \overline{Y}^I$) in the admissible setup in that proposition, but the results apply in the same way.  
}
\begin{itemize}
\item[(3)] The process $\w$ satisfies the consistency condition for the equilibrium price process defined by equation \eqref{eqn: mc price process mf v2}. 
\item[(4)] For $p=I,S$: 
\begin{equation}
\bar{Y}^p_t = Y^p_t, \qquad  \forall t\in [0,T], \quad  \mathbb{P}-a.s.
\end{equation} 
\end{itemize}
\end{defin}

The main result of the paper is the following theorem, proved in Section \ref{section: existence of solutions}.
\begin{theorem}
\label{thm: existence of solutions}
Under Assumption \ref{assumption: coefficients}, 
there exists a weak lifted mean-field equilibrium  in the sense of Definition \ref{def: mf price process}.
\end{theorem}

We remark that the above existence result requires weaker assumptions than those of \cite{fujii2022mean}. In this sense, even without asymmetric information, our result is more general, although we prove existence of weak rather than strong solutions. At our level of generality, there is no uniqueness result, as we do not assume any form of monotonicity condition. 

\subsection{Existence of stronger equilibria under suitable conditions}
\label{section: modification canonical space}

As pointed out above, Definition \ref{def: mf price process} involves additional information given by the processes $\overline{Y}^I, \overline{Y}^S$, which are part of the definition. 
However, for the application of the mean-field equilibrium in the analysis of a weak version of the market clearing condition in the economy with $N$ agents, we need an unlifted version of the definition in which the environment is given by $\chi^{p,\varpi}$ only, as we will discuss in detail in \S \ref{section: weak market clearing}. To guarantee this property, we need an additional assumption.

 \begin{assumption}  
 \label{assumption: affine target functions} For $p=I,S$, the functions $\bar{f}^p$ and $g^p$ introduced in \eqref{eqn: objective functional standard} and \eqref{eqn: objective functional major} satisfy \begin{equation*}
 \bar{f}^p(t,x,\chi^{p,\w})= x c^p(t,\chi^{p,\w}) \qquad \mbox{ and } \quad g^p(x,\chi^{p,\w}) = x \bar{g}^p(\chi^{p,\w}),
 \end{equation*}
 for suitable continuous and bounded functions $c^p$ and $\bar{g}^p$. 
 \end{assumption}

{Assumption \ref{assumption: affine target functions} is introduced for technical reasons, in order to obtain the unlifted weak equilibrium of Theorem \ref{thm: existence of unlifted mean-field price process} and the subsequent asymptotic market clearing analysis of Section \ref{section: weak market clearing}.
From a financial viewpoint, Assumption \ref{assumption: affine target functions} amounts to assuming that the position-dependent running and terminal terms are linear cashflows per unit of the asset held. For the standard agents, these cashflows depend only on $\varpi$, while for the informed agents they may also depend on the private signal $C$. 
The assumption is economically restrictive, since in this specification the adjoint process no longer depends on the state variable and, therefore, the corresponding optimal trading rule does not capture the full inventory risk management channel of intermediaries. 
The main reason for introducing this assumption is that it enables us to establish the unlifted formulation of equilibrium needed for the asymptotic analysis of Section \ref{section: weak market clearing}.
}

{We now introduce a weak {\em unlifted} notion of mean-field equilibrium. As in Definition \ref{def: mf price process}, this notion is not the direct counterpart of exact finite-player market clearing, but rather a weak formulation motivated by it and adapted to the analysis developed below. The notion is weak, since the probability space and the noises are part of the definition, but it is unlifted. Also in this case, the consistency condition is weaker than the original market clearing relation \eqref{eqn: mc price process mf}.}

\begin{defin}
\label{def: unlifted mf price process}
We say that:
\begin{equation}
\big(\Omega, \ItF, \prob, \tilde{\F}, (\xi^I, \xi^S), (B,W^I,W^S), (\w,C ) \big)
\end{equation}
is a \emph{weak (unlifted) mean-field equilibrium} if the following three conditions hold:
\begin{enumerate}
\item $\tilde{\F} := \tilde\F^I \vee \tilde\F^S$, where $ \tilde\F^I= \F^{\xi^I, B,W^I, \chi^{I,\w}}$ and $\tilde\F^S= \F^{\xi^S, B,W^S, \chi^{S,\w}}$.
\item For $p=I,S$, $(\Omega, \ItF, \prob, \tilde{\F}^p)$ is an admissible probabilistic setup in the sense of Definition \ref{def: admissibility}, carrying the process 
$
 ( \xi^p, (B,W^p), \chi^{p,\w}).
$
\end{enumerate}
In these admissible setups, let $(\tilde{X}^p, \tilde{Y}^p,\tilde{Z}^{p,0}, \tilde{Z}^p, \tilde{M}^p)$ be the solution to the FBSDE \eqref{eqn: FBSDE}, provided by Proposition \ref{lemma: iter random env}.
\begin{itemize}
\item[(3)] The process $\w$ satisfies the following consistency condition:
\begin{equation}
\label{eqn: conclusion winf projected}
\w_t = -\constantni\E\big[ n_I\bar{\Lambda}^I\tilde{Y}^{I}_t + n_S \bar{\Lambda}^S\tilde{Y}^{S}_t \bigm| \ItF^{\w, B}_t\big],\quad t\in[0,T], \ \mathbb{P}-a.s.
\end{equation}
\end{itemize}
\end{defin}

We remark that a weak mean-field equilibrium can be transferred to a canonical space of the form of \eqref{eqn: Omega}; see Section \ref{section: weak market clearing}. We can prove the following result: 
\begin{theorem} 
\label{thm: existence of unlifted mean-field price process}
Under Assumption \ref{assumption: coefficients} and \ref{assumption: affine target functions}, there exists a weak (unlifted)  mean-field equilibrium  in the sense of Definition \ref{def: unlifted mf price process}.
\end{theorem} 

\begin{proof}
By Theorem \ref{thm: existence of solutions}, a weak mean-field equilibrium 
\[
(\Omega, \ItF, \prob, \F, (\xi^I, \xi^S), (B,W^I,W^S), (\w,C, \bar{Y}^I, \bar{Y}^S))
\] 
satisfying Definition \ref{def: mf price process} exists. 
Let  $\tilde\F^I= \mathbb{F}^{\xi^I, B,W^I, \chi^{I,\w}}$ and $\tilde\F^S= \mathbb{F}^{\xi^S, B,W^S, \chi^{S,\w}}$; it is clear that  $(\Omega,\ItF,\prob, \F^{p}, (\xi^{p}, (B,W^{p}) ,\chi^{p,\w}))$ is an admissible setup, for $p=I,S$. Let  
$(\tilde{X}^p, \tilde{Y}^p,\tilde{Z}^{p,0}, \tilde{Z}^p, \tilde{M}^p)$ be the solution to the FBSDE \eqref{eqn: FBSDE} in these setups.  
As discussed in the proof of \cite[Theorem 5.1]{el1997backward}, the adjoint process determined by the stochastic maximum principle is defined as 
\begin{equation}
\label{eqn: YtildeBSDE}
\tilde{Y}^{p}_t = \E\bigg[ \bar{g}^p(\chi^{p,\w}_T) + \int_t^T c^p(s,\chi^{p,\w}_s)ds\biggm| \tilde{\ItF}^{p}_t\bigg].
\end{equation} 
By Assumption \ref{assumption: affine target functions}, the {derivatives of the cost functions} do not depend on the state variable $\tilde{X}^{p}$. Consequently, the process $\tilde{Y}^{p}$ depends only on the exogenous processes $\chi^{p,\w}$ and on the other noises generating the filtration, namely $B$ and $W^{p}$.  
Recall that $\bar{Y}^p = Y^{p}$ denotes the adjoint process for the control problem formulated in the filtration $\F^{p}$ in \eqref{eqn: lifted filtrations}. In particular, $Y^p$ admits a representation as a conditional expectation as in \eqref{eqn: YtildeBSDE}, with $\tilde{\ItF}^p_t$ replaced by $\ItF^p_t$.  
It follows that $\tilde{Y}^{p}_t = \E\big[Y^{p}_t \mid \tilde{\ItF}^{p}_t\big]$, and thus \eqref{eqn: conclusion winf projected} is obtained by applying the tower property of conditional expectations to \eqref{eqn: mc price process mf v2}.
\end{proof}

\section{Construction of mean-field equilibrium}
\label{section: existence of solutions}

The aim of this section is to prove Theorem \ref{thm: existence of solutions},  by constructing a tight sequence of discretized solutions defined on the canonical space. Throughout this section, we suppose that Assumption \ref{assumption: coefficients} is in force.
We develop a strategy to construct the solution of \eqref{eqn: mc price process mf v2} as limit in distribution of a sequence of approximated price processes. We proceed as follows:
\begin{enumerate}
\item In \S \ref{section: discretization procedure}, we introduce a discretization procedure which enables us to reduce the space $\mathcal{D}([0,T],\R)^{\Omega^0}$ to $\mathcal{D}([0,T],\R)^{\mathbb{J}^{n}}$, where $\mathbb{J}^n$ is a finite set which we define below and $n$ represents the discretization step in space and time. We prove the existence of a fixed point for a suitable input-output functional defined on $\mathcal{D}([0,T],\R)^{\mathbb{J}^{n}}$. For each $n\in\N$, the fixed point $\wn$ plays the role of a discretized equilibrium price process. 

\item In \S \ref{section: stability}, we consider the sequence $(\wn)_{n\in\N}$ together with the state variables $(\XNn,\XIn)_{n\in\N}$ associated with the optimal control problems for the typical informed agent and the typical standard agent. We show that $(\wn)_{n\in\N}$ and $(\XNn,\XIn)_{n\in\N}$ form tight sequences. Afterwards, we prove that the discretized equilibria are stable, in the sense that the optimal control problems for the typical informed agent and the typical standard agent are solved by the weak limit of $(\XIn)_{n\in\N}$ and $(\XNn)_{n\in\N}$ respectively, when the price process appearing in the coefficients of the problems is the weak limit of $(\wn)_{n\in\N}$.

\item In \S \ref{section: consistency condition}, we conclude that the weak limit of the sequence $(\wn)_{n\in\N}$ satisfies \eqref{eqn: mc price process mf v2}.
\end{enumerate}

\subsection{Discretization procedure}
\label{section: discretization procedure}

\subsubsection{Discretized setup}
\label{section: discretization setup}
In analogy to \cite[Section 3.3]{carmona2018probabilistic2} we shall work on the product between the canonical spaces:
\begin{align*}
\bOm^0	&:= \mathcal{C}([0,T],\R)\times \mathcal{D}([0,T],\R);\\ 
\bOm^p	&:= \R\times \mathcal{C}([0,T],\R), \quad p = I,S.
\end{align*}
We endow $\bOm^0$ with the probability measure $\bprob^0 := \mathcal{L}(B,C)$. By Assumption \ref{assumption: admissibility}, the first marginal of $\bprob^0$ is the one-dimensional Wiener measure $\mathcal{W}$. The canonical process on $\bOm^0$ is denoted by $(b,c)$. 
On the other hand, we endow $\bOm^p$ with the probability measure $\bprob^p :=\mathcal{L}(\xi^p)\otimes \mathcal{W}$ and we denote by $(\eta^p,w^p)$ the canonical process on $\bOm^p$ for $p = I,S$. We also introduce the following space: 
\begin{equation}
\label{eqn: bOm}
\bsetup := \bl\bOm^0\times \bOm^I\times \bOm^S, \bItF^{0}\otimes \bItF^{I}\otimes \bItF^{S}, \bprob^{0}\otimes\bprob^I\otimes\bprob^S,\F^{b,c,\eta^I,\eta^S,w^I,w^S}\br.
\end{equation}
 In the following, we denote the expected value on $\bsetup$ by $\bE$.

We first present the discretization procedure on $\R$. Let us consider two integers $l,n\geq 1$, where $l$ represents the space discretization parameter and $n$ the time discretization parameter.
 Denoting with $\lfloor x \rfloor$ the floor function applied to $x$, we introduce the following function:
 \begin{align*}
 \Pi_{l,1}: \R	&\to \R\\ 
 		x	&\mapsto	
        \begin{cases}
					2^{-l} \lfloor x 2^l \rfloor, \quad &\text{if } |x| \leq 2^l;\\
					2^l \text{ sign }(x),	  \quad &\text{if } |x| > 2^l.
				\end{cases}
\end{align*}
%
Moreover, we consider $\Pi_{l,j}:\R^j\to \R^j$, defined for any $j \geq 2$ by: 
\begin{equation*}
(y^1 ,\dots, y^j )	:=  \Pi_{l,j}(x^1,\dots,x^j), \quad \Pi_{l,j+1}(x^1,\dots,x^{j+1})	:=  (y^1,\dots,y^j, \Pi_{l,1}(y^j + x^{j+1}-x^j)).\\
\end{equation*}
The following result is analogous to \cite[Lemma 3.17]{carmona2018probabilistic2}:
\begin{lemma} 
With the notation introduced above, given $l\in\mathbb{N}$, for every $(x^1,\dots,x^j)\in \R^j$ such that $|x^i|\leq 2^l-1$ for all $ i \in \{1,\dots,j\}$, let $(y^1,\dots,y^j) := \Pi_{l,j}(x^1,\dots,x^j)$. Then we have $|x^i - y^i | \leq \frac{i}{2^l}$ for each $ i\in\{1,\dots,j\}$ and for all $j\leq 2^l$.
\end{lemma}
Given an integer $n$, let $N= 2^n$ and consider the dyadic time mesh $t_i = \frac{iT}{N}$, $i\in \{0,\dots,N\}$. We introduce the random variable $\bar{V}:= (V_1,\dots,V_{{2^n - 1}}) = \Pi_{l,{2^n - 1}}( b_{t_1}, \dots, b_{t_{{2^n - 1}}})$ and adopt the notation
\begin{equation}
\label{eqn: discr bm}
\bar{V}_j := (V_1,\dots,V_j), \quad j = 1,\dots, 2^n-1, 
\end{equation}
where $(b_{t})_{t\in[0,T]}$ is the first component of the canonical process on $\bar{\Omega}^0$.
$\bar{V}$ is a discrete random variable defined on $(\bar{\Omega}^0, \bar{\ItF}^0,\bar{\prob}^0)$. By \cite[Lemma 3.18]{carmona2018probabilistic2}, for every $i = 1,\dots, {2^n - 1}$, the random vector $(V_1,\dots,V_i)$ has support $(\mathbb{J}_l)^i$, where
\begin{equation}
\label{eqn: J}
\mathbb{J}_{l}:= \Bigg{\{} - \Lambda, -\Lambda + \frac{1}{\Lambda} , -\Lambda + \frac{2}{\Lambda}, \dots, \Lambda - \frac{1}{\Lambda}, \Lambda \Bigg{\}},\quad \Lambda:= 2^l.
\end{equation}
We aim at constructing a stochastic process $\w$ on $\Bar{\Omega}^0$ that is adapted to the discretization $\bar{V}$ of the Brownian motion $b$ and that satisfies a discrete version of the equilibrium condition  \eqref{eqn: mc price process mf v2}. For this purpose, we introduce an input-output map, whose fixed point is $\w$. 

We introduce a discretized input process as an object $\boldsymbol{\bar{\theta}}:= (\theta^0,\dots,\theta^{{2^n - 1}})$ such that  $\theta^i:\mathbb{J}^i_{l}\to \mathcal{C}([t_i,t_{i+1}],\R)$, for each $i = 0,\dots,{2^n - 1}$, where $\mathbb{J}^0_l= \{\emptyset\}$. Equivalently, $\boldsymbol{\bar{\theta}}\in\prod_{i=0}^{{2^n - 1}} \mathcal{C}([t_i,t_{i+1}],\R)^{\mathbb{J}^i_{l}}$. Notice that $\boldsymbol{\bar{\theta}}$ determines uniquely an element $({\vartheta}_t)_{t\in[0,T]}\in \mathcal{D}([t_0,t_{N}],\R)^{\mathbb{J}_l^{2^n - 1}}$, defined as follows: 
\begin{equation}
\label{eqn: input map}
\begin{cases}
{\vartheta}_t(v_1,\dots,v_{{2^n - 1}}):= \theta^i_t(v_1,\dots,v_{i}),\quad t\in[t_i,t_{i+1}),\quad i\in\{0,\dots,{2^n - 1}\},\\ 
{\vartheta}_T(v_1,\dots,v_{{2^n - 1}}):= \theta^{{2^n - 1}}_T(v_1,\dots,v_{{2^n - 1}}).
\end{cases}
\end{equation}
We now define the \cadlag stochastic process $\w$ on $(\bar{\Omega}^0,\bar{\ItF}^0,\bar{\prob}^0)$ as follows:
\begin{equation} 
\label{eqn: discr price} 
\w^{\theta}_t:= {\vartheta}_t(V_1,\dots,V_{{2^n - 1}}),\quad t\in[0,T].
\end{equation}

\subsubsection{Existence of a fixed point in the discretized setup}
\label{section: existence fixed point discretized} 
%
First of all, we make the following observation.
\begin{remark}
\label{rmk: compatibility-approx-ocp}
The process $\vt$ is adapted to the filtration generated by the Brownian motion $b$, and thus $(\eta^p,b, w^p,\chi^{p,\vt})$ is adapted to $\bar{\mathbb{F}}^p$, for $p=S,I$, where  
$\bar{\mathbb{F}}^S= \F^{\eta^S, b, w^S}$ and
$\bar{\mathbb{F}}^I= \F^{\eta^I, b, w^I, c}$. 
We recall that adaptedness implies that $(\eta^p,b,w^p,\chi^{p,\varpi^\theta})$ is compatible with $\mathbb F^p$.
%
%
\end{remark} 
We can now introduce the optimal control problem for the typical agent of the two populations $p = I,S$ on $\bsetup$, in the class of controls $\mathbb{H}^2(\bF^p)$. 
Let $(\bOm, \bar{\ItF}, \bprob, \bF^S, \tilde{X}^S, \tilde{Y}^S, \tilde{Z}^{S,0}, \tilde{Z}^S, 0)$  and $(\bOm, \bar{\ItF}, \bprob, \bF^I, \tilde{X}^I, \tilde{Y}^I, \tilde{Z}^{I,0}, \tilde{Z}^I, \tilde{M}^I)$ be the solution to the FBSDE \eqref{eqn: FBSDE}, provided by Proposition \ref{lemma: iter random env}, in the probabilistic setup  
\begin{equation*}
(\bOm, \bItF, \bprob, \bF^p, \eta^p, (b, w^p), \chi^{p,\vt}),\quad p = I,S.
\end{equation*}
Note that, since the random environment $\vt$ is adapted to $b$, the \cadlag orthogonal martingale term $\tilde{M}^S$ is not present in the FBSDE for the typical standard player. The optimal control has the form \eqref{eqn: optimal controls}. 

We introduce the discretized output process $\boldsymbol{\Phi}(\boldsymbol{\bar{\theta}}) := (\vphi^0(\boldsymbol{\bar{\theta}}), \dots, \vphi^{{2^n - 1}}(\btheta))$, defined as follows:
\begin{equation}
\label{output discretized}
\varphi^i_t(\boldsymbol{\bar{\theta}})=\bl-(n_I\bar{\Lambda}^I + n_S\bar{\Lambda}^S)^{-1}\bE\blq n_I\bar{\Lambda}^I\tilde{Y}^I_t + n_S\bar{\Lambda}^S \tilde{Y}^S_t \bigm| V_1 = v_1,\dots, V_i = v_i]\brq\br_{(v_1,\dots,v_i)\in\mathbb{J}^i_{l}}, \; t\in[t_i,t_{i+1}].
\end{equation}
In particular, $\vphi^i(\boldsymbol{\bar{\theta}})[v_1,\dots,v_i]\in \mathcal{C}([t_i,t_{i+1}],\R)$, for each $i = 0,\dots, {2^n - 1}$. Analogously to the definition of $\btheta$ and $(\theta_t)_{t\in[0,T]}$, we can introduce $\bPhi(\btheta)$ as an element of $\mathcal{D}([0,T],\R)^{\mathbb{J}_l^{2^n - 1}}$:
\begin{alignat*}{3}
\bPhi(\btheta) :	&\qquad\mathbb{J}^{2^n - 1}_l 	&&\to \mathcal{D}([0,T],\R)\\
			&(v_1,\dots,v_{{2^n - 1}})	&&\mapsto (\Phi_t(v_1,\dots,v_{{2^n - 1}}))_{t\in[0,T]}
\end{alignat*}	
where 
\begin{displaymath}
\Phi_t(v_1,\dots,v_{{2^n - 1}}):= 	\begin{cases}
						\vphi^i_t(\btheta),		&  t \in [t_i,t_{i+1}),\  i = 0,\dots,{2^n - 1},\\ 
						\vphi^{{2^n - 1}}_t(\btheta),	& t =T.\\
						\end{cases}	
\end{displaymath}
Note that $\bPhi$ is well-defined because $(V_1,\dots,V_i)$ takes values in $\mathbb{J}^i_{l}$. We remark that $\bPhi$ is given by 
\begin{equation}
\label{eqn: prod input-output}
\begin{aligned}
\bPhi : 	\prod_{i = 0}^{{2^n - 1}} \mathcal{C}([t_i,t_{i+1}],\R)^{\mathbb{J}^i_{l}}		&\to \prod_{i = 0}^{{2^n - 1}} \mathcal{C}([t_i,t_{i+1}],\R)^{\mathbb{J}^i_{l}}\\ 
		\btheta 												&\mapsto \bPhi(\btheta) := ( \vphi^0(\btheta),\dots,\vphi^{{2^n - 1}}(\btheta)).
\end{aligned}
\end{equation}
In the following, we shall make use of a more compact notation: we denote $(\bar{\theta}_t)_{t\in[0,T]}$ by $\bar{\theta}$ and the vector $(v_1,\dots,v_i)\in\mathbb{J}^i_{l}$, by $\bv_i$, for every $i = 1,\dots,N$. We prove the existence of a fixed point of the function $\bPhi$ introduced in \eqref{eqn: prod input-output}, which is hence a solution of the discretized problem. In order to apply standard fixed point results, we introduce the following metric on $\prodspace$:
\begin{small}
\begin{equation}
\label{eqn: metric prod input-output}
d(\btheta^1,\btheta^2):= \max_{i = 0,\dots, {2^n - 1}} \Big{\{}\max_{v_{i} \in \mathbb{J}^{i}_l} \sup_{t\in [t_i,t_i+1]}\{ |\theta^{1,i}_t(v_{i}) - \theta^{2,i}_t(v_{i}) | \}\Big{\}},\quad  \btheta^1,\btheta^2 \in \prodspace.
\end{equation}
\end{small}
The next proposition represents a key step and will be proved in \S \ref{Section: discretized equilibria}.  
\begin{proposition}
\label{prop:existence_discrete}
    There exists a fixed point for the functional  $\boldsymbol{\Phi}$ given by \eqref{eqn: prod input-output}.
\end{proposition}

\subsubsection{Solution to the discretized game - proof of Proposition \ref{prop:existence_discrete} }
\label{Section: discretized equilibria}
We apply Schauder's theorem to find a fixed point for $\boldsymbol{\Phi}$.
 In order to apply Schauder's Theorem, we have to prove that:
\begin{enumerate}
\item $\bPhi$ is continuous;
\item there exists a compact and convex subset $K\subseteq \prodspace$, for the topology induced by $d$, such that 
$\bPhi(K)\subseteq K$.
\end{enumerate}

\noindent As a first step, we prove the continuity in the whole space $\prodspace$. 

\begin{proposition}
\label{proposition: continuity of bPhi}
Let  $(\btheta^k)_{k\in\N}$ and $\btheta \in\prodspace$ such that $\lim_{k\to\infty}d(\btheta^k,\btheta)= 0$, then, it holds that
\begin{equation}
\lim_{k\to\infty} d(\bPhi(\btheta^k),\bPhi(\btheta)) = 0, 
\end{equation}
where the metric $d$ is defined in \eqref{eqn: metric prod input-output}. 
\begin{proof}
In order to prove the continuity of $\bPhi$ defined by \eqref{eqn: prod input-output} it is sufficient to show that: 
\begin{equation}
\label{eqn: continuity condition}
\lim_{k\to\infty}\sup_{t\in [t_i,t_i+1]} |\vphi^i_t(\btheta^k)[\bv_{i}] - \vphi^i_t(\btheta)[\bv_{i}]|= 0,\quad \forall \bv_{i} \in \mathbb{J}^i_{l}, \quad \forall i  =0,\dots,{2^n - 1}.
\end{equation} 
We consider the input processes $\w^k$ and $\w$ defined as in \eqref{eqn: discr price} by $\btheta^k$ and $\btheta$ respectively. For $p = I,S$, we denote by $Y^{p,k}$ and $Y^p$ the solutions to the backward components in system \eqref{eqn: FBSDE} with $\w^k$ and $\w$ playing the role of the price process, respectively. The left hand side of \eqref{eqn: continuity condition} is equivalent to
\begin{align}
\label{eqn: continuity phii}
&\sup_{t\in [t_i,t_i+1]} |\vphi^i_t(\btheta^k)[\bv_{i}] - \vphi^i_t(\btheta)[\bv_{i}]|\\	
&\nonumber\qquad= \sup_{t\in [t_i,t_i+1]} \Big|-\constantni\bE\bigl[n_I\bar{\Lambda}^IY^{I,k}_t + n_S\bar{\Lambda}^SY^{S,k}_t \mid V_1 = v_1,\dots, V_i = v_i\bigr]\\ 
&\nonumber\qquad\qquad  +\constantni\bE\bigl[n_I\bar{\Lambda}^IY^I_t + n_S\bar{\Lambda}^SY^S_t \mid V_1 = v_1,\dots,V_i =  v_i\bigr]\Big| \\ 
&\nonumber\qquad\leq \constantni\sup_{t\in[t_i,t_{i+1}]}\Big{\{}n_I\bar{\Lambda}^I\big|\bE\bigl[Y^{I,k}_t - Y^I_t\mid V_1 = v_1,\dots, V_i = v_i\bigr]\big| \\
&\nonumber\qquad\qquad +n_S\bar{\Lambda}^S\big|\bE\bigl[Y^{S,k}_t - Y^S_t\mid V_1 = v_1,\dots, V_i = v_i\bigr]\big| \Big{\}}.
\end{align}
Let us consider the term associated with the adjoint processes of the informed agent {(i.e. $p=I$)}:
\begin{equation*}
(A):= \sup_{t\in[t_i,t_{i+1}]}\big|\bE\bigl[Y^{I,k}_t - Y^I_t\mid V_1 = v_1,\dots, V_i = v_i\bigr]\big|.
\end{equation*}
{The analysis of the term associated with the adjoint processes of the standard agent (i.e. $p=S$) is analogous and, therefore, omitted.}
As discussed in the proof of Proposition \ref{lemma: iter random env}, for $t\in[0,T]$:
\begin{align}
\label{eqn: discr Y}
Y^{I,k}_t&= \bE\bblq \partial_xg^I(X^{I,k}_T, \w^k_T, c_T) + \int_t^T \partial_x \bar{f}^I(s,X^{I,k}_s, \w^k_s,c_s)ds \biggm| \bar{\ItF}^I_t\bbrq\\
&\nonumber=\partial_xg^I(X^{I,k}_T, \w^k_T, c_T) + \int_t^T \partial_x \bar{f}^I(s,X^{I,k}_s, \w^k_s,c_s)ds-  \int_t^T Z^{I,0,k}_sdb_s - \int_t^T Z^{I,k}_sdw^I_s \\
&\nonumber\qquad-( M^{I,k}_T - M^{I,k}_t).
\end{align}
We adopt the notation
\begin{align*}
B_k=   \bE\bblq & \big| \partial_xg^I(X^{I}_T, \w^k_T, c_T) -  \partial_xg^I(X^{I}_T, \w_T, c_T)\big| + \int_t^T \big|\partial_x \bar{f}^I(s,X^{I}_s, \w^k_s,c_s)\\
&-  \partial_x \bar{f}^I(s,X^{I}_s, \w_s,c_s)\big|ds\biggm| V_1 = v_1,\dots, V_i = v_i\bbrq. 
\end{align*}
By Lipschitz continuity of $\partial_xg^I$ and $\partial_xf^I$ in the $x$-variable and the tower property, we have that
 \begin{align*}
 (A)= 	&\sup_{t\in[t_i,t_{i+1}]}\bII \bE \bblq  \partial_xg^I(X^{I,k}_T, \w^k_T, c_T) + \int_t^T \partial_x \bar{f}^I(s,X^{I,k}_s, \w^k_s,c_s)ds-  \partial_xg^I(X^{I}_T, \w_T, c_T) \\
 		&\qquad\qquad -  \int_t^T \partial_x \bar{f}^I(s,X^{I}_s, \w_s,c_s)ds\biggm| V_1 = v_1,\dots, V_i = v_i\bbrq\bII\\
	\leq 	&\sup_{t\in[t_i,t_{i+1}]} \bE \bblq L | X^{I,k}_T - X^I_T| + | \partial_xg^I(X^{I}_T, \w^k_T, c_T) -  \partial_xg^I(X^{I}_T, \w_T, c_T)|+  \int_t^T\blq  L |X^{I,k}_s - X^I_s| \\
		&\qquad \qquad + |\partial_x \bar{f}^I(s,X^{I}_s, \w^k_s,c_s)-  \partial_x \bar{f}^I(s,X^{I}_s, \w_s,c_s)|\brq ds\biggm| V_1 = v_1,\dots, V_i = v_i\bbrq\\
	\leq 	&\bE\bblq L(1+T) \sup_{t\in[t_i,T]} |X^{I,k}_t - X^I_t | \biggm| V_1 = v_1,\dots, V_i = v_i\bbrq+ B_k.
 \end{align*}
 {Since $\varpi^k\to\varpi$ uniformly and $\partial_xg^I$,$\partial_xf^I$ are continuous in the price variable and uniformly bounded, dominated convergence yields $B_k\to0$.}
 As a consequence, it suffices to prove that
\begin{equation}
\label{eqn: stability state variable}
\lim_{k \to \infty} \bE\bblq\sup_{t\in[t_i,T]}|X^{I,k}_t - X^I_t|\biggm|\dcond\bbrq = 0,\quad\forall \bv_{i}\in \mathbb{J}^i_{l},\quad i = 0,\dots,{2^n - 1}.
\end{equation}
We apply \cite[Theorem 1.53]{carmona2018probabilistic2}, whose assumptions are satisfied by Assumption \ref{assumption: coefficients}, except for \eqref{eqn: iterrandom}, which is ensured by Proposition \ref{lemma: iter random env}. Hence, there exists a constant $\Gamma_i$ dependent on $L$, $t_i$ and $T$ such that: 
\begin{small}
\begin{equation}
\label{eqn: stability condition}
\begin{split}
&\bE\bblq \sup_{t\in[t_i,T]} |X^{I,k}_t - X^I_t|^2\biggm| \bItF^I_{t_i}\bbrq \\	
&\quad \leq \Gamma_i \bE\bblq  |X^{I,k}_{t_i} - X^{I}_{t_i}|^2+L^2 |\w^k_T - \w_T|^2+ \int_{t_i}^T \blq|\partial_x \bar{f}^I(t,X^{I}_t,\w_t,c_t) - \partial_x\bar{f}^I(t, X^{I}_t, \w^k_t, c_t)|^2+|\bar{\Lambda}^I|^2 |\w_t - \w^k_t|^2 \\ 
&\qquad\qquad  +  |l^I(t,\w_t) - l^I(t,\w^k_t) |^2 +|\sigma^{I,0}(t,\w^k_t) - \sigma^{I,0}(t,\w_t) |^2  + |\sigma^I(t,\w_t) - \sigma^I(t,\w^k_t) |^2  \brq dt \biggm| \bItF^I_{t_i}\bbrq.
\end{split}
\end{equation}
\end{small}
Assumption \ref{assumption: coefficients} implies that every term apart from $|X^{I,k}_{t_i} - X^I_{t_i}|^2$ in the right-hand side of \eqref{eqn: stability condition} converges to zero as $k\to\infty$. By the tower property: 
\begin{displaymath}
\bE\bblq\sup_{t\in[t_i,T]}|X^{I,k}_{t_i} - X^I_{t_i}|^2\biggm|\dcond\bbrq \leq \Gamma_i \bE\blq  |X^{I,k}_{t_i} - X^{I}_{t_i}|^2\Bigm|\dcond\brq + C_k,		
\end{displaymath}
where $\lim_{k\to\infty} C_k = 0$. Let us notice that in $\bE[  |X^{I,k}_{t_i} - X^I_{t_i}|^2|\dcond]$, the processes $\w$ and $\w^k$ appear in the dynamics of $X^I$ and $X^{I,k}$, respectively, until time $t_i$. Thus, they are fixed by the conditioning event $\{\dcond\}$. By \eqref{eqn: stability condition} for $t_i = t_0 = 0$ we have that
\begin{displaymath}
\lim_{k\to\infty} \bE\bblq \sup_{t\in[0,T]} \big|X^{I,k}_t - X^{I}_t\big|^2\bbrq  = 0.
\end{displaymath}
Moreover, since $\bar{V}_{i}= (V_1,\dots,V_i)$ is a discrete random variable whose support is given by the finite set $\mathbb{J}^i_{l}$, by definition of conditional expectation, the following holds:
\begin{displaymath} 
\bE[  |X^{I,k}_{t_i} - X^I_{t_i}|^2] 
= \sum_{\bv_{i}\in \mathbb{J}^i_{l}} \bE\Bigl[  \big|X^{I,k}_{t_i} - X^I_{t_i}\big|^2\Bigm| \dcond\Bigr]\bprob(\dcond).
\end{displaymath}
By construction, $\bprob(\dcond)>0$ for each $\bv_{i}\in \mathbb{J}^i_{l}$. As a consequence,
\begin{small}
\begin{align*}
\lim_{k\to\infty} \bE\Bigl[  \big|X^{I,k}_{t_i} - X^I_{t_i}\big|^2\Bigm| \dcond\Bigr]	&\leq \lim_{k\to\infty} \frac{1}{ \bprob(\dcond)} \bE\bigl[  |X^{I,k}_{t_i} - X^{I}_{t_i}|^2\bigr]\\ 
&\leq \frac{1}{ \bprob(\dcond)} \lim_{k\to\infty} \bE\bblq\sup_{t\in[0,T]} |X^{I,k}_{t} -X^I_{t}|^2\bbrq =0.
\end{align*}
\end{small}
The claim then follows by Jensen's inequality applied to \eqref{eqn: stability state variable}.
\end{proof}
\end{proposition}

We must now prove that the image of $\bPhi$, introduced in \eqref{eqn: prod input-output}, is contained in a compact set of $\prodspace$. To this end, we apply Ascoli-Arzelà's theorem 
to the set of functions
\begin{equation} 
C^i_{\bv_{i}}:= \Bigg{\{}\vphi^i_t(\btheta)[\bv_{i}]: [t_i,t_{i+1}]\to \R,\quad \btheta \in \prodspacej\Bigg{\}},\quad \forall \bv_{i}\in \mathbb{J}^i_{l}, 
\end{equation}
defined for every $ i = 0,\dots,{2^n - 1}$, where $\vphi$ is introduced in equation \eqref{output discretized}. Indeed, if $C^i_{\bv_{i}}$ has compact closure for each $i$, also the finite product $\prod_{i = 0}^{2^n-1} \prod_{\bv_{i}\in \mathbb{J}^i_{l}} C^i_{\bv_{i}}$ has compact closure. To carry out this program, we first prove the following lemma. 
\begin{lemma} There exists a constant $C$ dependent on $L$ and $T$ such that: 
\begin{align}
\label{eqn: pointwise boundedness varphi}
\sup\Bigg{\{} |\vphi^i_t(\btheta)[\bv_{i}]|:\ \btheta \in \prodspacej\Bigg{\}} &\leq C,\quad \forall t\in [t_i,t_{i+1}],\quad \forall \bv_{i} \in \mathbb{J}^i_{l},\\
\label{eqn: equi-continuity varphi}
|\vphi^i_t(\btheta)[\bv_{i}] - \vphi^i_s(\btheta)[\bv_{i}]| &\leq C|t-s|,\quad \forall t,s \in [t_i,t_{i+1}].
\end{align}
\begin{proof}
First of all, \eqref{eqn: pointwise boundedness varphi} is guaranteed by Assumption \ref{assumption: coefficients} as proved in Lemma \ref{lemma: boundedness fourth moment sequence of processes}.
Indeed, as $Y^I$ and $Y^S$ are bounded by $C_B$,  for all  $\btheta\in \prodspacej$, it holds that
\begin{align*}
 \bigl|\vphi^i_t(\btheta)[\bv_{i}]\bigr| 	&= \Big|\bE\Bigl[-\constantni(\bar{\Lambda}^In_IY^I_t + \bar{\Lambda}^Sn_S Y^S_t) \Big|\dcond\Bigr]\Big| \\
 &\leq \constantni  (\bar{\Lambda}^In_I  + \bar{\Lambda}^Sn_S)C_B =C_B.
\end{align*}
To prove \eqref{eqn: equi-continuity varphi}, we consider $s\leq t$ in the interval $[t_i,t_{i+1}]$ and notice that, by Assumption \ref{assumption: coefficients},
\begin{align*}
|\vphi^i_t&(\btheta)[\bv_{i}] - \vphi^i_s(\btheta)[\bv_{i}]|		=\bI \constantni\bE\blq\bar{\Lambda}^In_IY^I_t + \bar{\Lambda}^Sn_S Y^S_t \Bigm| \dcond\brq\\
&\qquad  - \constantni\bE\blq\bar{\Lambda}^In_IY^I_s + \bar{\Lambda}^Sn_S Y^S_s \Bigm|\dcond\brq\bI  \\ 
&\leq \constantni\Bigg( \sum_{p = I,S}\bII\bar{\Lambda}^pn_p\bE\Bigg[ \int_s^t \partial_x\bar{f}^p(u,X^p_u,\chi^{p,\w}_u)du \Biggm| \dcond\Bigg]\bII\Bigg)\\ 
&\leq \constantni\bE\bblq \int_s^t\Bigg( \sum_{p = I,S}\bar{\Lambda}^pn_p|\partial_x\bar{f}^p(u,X^p_u,\chi^{\w,p}_u)|\Bigg)du  \Biggm|\dcond\bbrq\\
&\leq 2L(t-s).
\end{align*}
The same holds if $t\leq s$, thus proving the lemma. 
\end{proof}
\end{lemma}

{By \eqref{eqn: pointwise boundedness varphi}, for every fixed $i=0,\ldots,2^n-1$ and every $\bar{v}_i\in \mathbb{J}^i_l$, the set $C^i_{\bar{v}_i}$ is uniformly bounded in $\mathcal{C}([t_i,t_{i+1}],\mathbb R)$. Moreover, \eqref{eqn: equi-continuity varphi} implies that $C^i_{\bar{v}_i}$ is equicontinuous on $[t_i,t_{i+1}]$.}
We can then apply Ascoli-Arzel\`{a}'s theorem, which guarantees that $\vphi^i[\bv_{i}]$ has compact closure for the uniform norm, for each choice of the vector $\bv_{i}\in\mathbb{J}^i_{l}$. Since $\mathbb{J}^i_{l}$ is finite, also the function
\begin{displaymath}
\vphi^i:\prodspacej \to \mathcal{C}([t_i,t_{i+1}],\R)^{\mathbb{J}^i_{l}}
\end{displaymath}
has compact closure of its image. We can conclude that the image of $\bPhi = (\vphi^0,\dots\vphi^{{2^n - 1}})$ has compact closure. Therefore, we can restrict the continuous function $\bPhi$ to the compact closure of its image in order to apply Schauder's fixed point theorem and, therefore, the proof of Proposition \ref{prop:existence_discrete} is complete.

\subsection{Stability of the discretized equilibria in the weak limit}
\label{section: stability}

In \S \ref{section: discretization procedure} we proved the existence of a sequence of discretized equilibria. In this section, we show that the sequence admits a limit in distribution which satisfies \eqref{eqn: mc price process mf v2}. 

We consider the canonical space $(\bOm,\bItF, \bprob)$ introduced in \S \ref{section: discretization setup} endowed with the canonical processes $(b,c,\eta^I,w^I,\eta^S,w^S)$. Through the space-time grid determined by $n,l := 2n\in\N$, we denote by $\wn$ the fixed point of the functional $\bPhi$, defined in \eqref{eqn: prod input-output}, by $(\Xpn, \Ypn, \Zopn, \Zpn,\Mpn)$ the solution to  \eqref{eqn: FBSDE}, for $p = I,S$, defined assuming that $\w^n$ plays the role of the price process. In particular, the process $\wn$ is a \cadlag process defined for every $i = 0,\dots, {2^n - 1}$ as follows: 
\begin{equation}
\label{eqn: wn}
\wn_t = -\constantni\bE\bigl[n_I\bar{\Lambda}^I\YIn_t + n_S\bar{\Lambda}^S \YNn_t\big| \bVn_i\bigr],\quad t\in[t_i,t_{i+1}),
\end{equation}
where $\bVn_i = (\Vn_1, \dots \Vn_i)$ is the discretization of the common noise until time $t_i = i T/2^n$.

The strategy at the basis of the proof of Theorem \ref{thm: existence of solutions} can be outlined as follows: 
\begin{enumerate}
\item \label{step: S-I} In \S \ref{section: tightness}, we prove tightness of the sequences $(\XIn,\XNn)_{n\in\N}$ in $\mathcal{C}([0,T],\R^2)$. Moreover, we show that the sequence $(\mathcal{W}^n)_{n\in\N}$, where $\mathcal{W}^n := (\w^n,\YIn,\YNn)$, is tight in $\mathcal{M}([0,T],\R^3)$, where $\mathcal{M}([0,T],\R^3)$ is the Meyer-Zheng space (some properties of $\mathcal{M}([0,T],\R)$ we use are recalled in Appendix \ref{Appendix:MZ}). 
As a consequence, the sequence 
\begin{equation}
\label{eqn: tight sequence}
\Theta^n := (b,c,\eta^I,w^I,\eta^S,w^S,\w^n,\YIn,\YNn,\XIn,\XNn), \quad n\in\N,
\end{equation}
defined on $\bar{\Omega}$, is tight on the space 
\begin{equation}
\label{eqn: Omegainput}
\Omega_{\text{input}} := \mathcal{C}([0,T],\R) \times \mathcal{D}([0,T],\R) \times (\mathcal{C}([0,T],\R)\times \R)^2 \times \mathcal{M}([0,T],\R^3) \times \mathcal{C}([0,T],\R^2).
\end{equation}
As pointed out in Remark \ref{rmk: compatibility-approx-ocp}, for  $p\in{I,S}$, the compatibility condition between the canonical filtration $\bar{\mathbb{F}}^p$ and the process $(\eta^p,b,w^p,\chi^{p,\wn})$ is ensured since $(\eta^p,b,w^p,\chi^{p,\wn})$ is adapted to $\bar{\mathbb{F}}^p$.

\item \label{step: S-II} In \S \ref{section: compatibility condition weak limit}, we prove that the compatibility condition is preserved passing to the weak limit of \eqref{eqn: tight sequence} within the two populations $p = I,S$. To this end, we consider a limit in distribution of $(\Theta^n)_{n\in\N}$ (which we still denote $(\Theta^n)_{n\in\N}$) defined on a suitable complete probability space $(\Omega^{\infty}, \ItF^{\infty}, \prob^{\infty})$ and denoted by 
\begin{equation}
\label{eqn: weak limit thetan}
\Theta^\infty := \processinf.
\end{equation}
On $\setupinf$, we introduce the following filtrations:
\begin{align}
\mathbb{F}^{\infty}		&:= \F^{\Thetainf} 
\label{eqn: limit filtration},\\
\mathbb{F}^{I,\infty}		&:= \F^{\eta^{I,\infty}, \binf, c^{\infty},w^{I,\infty}, \mathcal{W}^{I,\infty}, \XIinf} \label{eqn: limit filtration major},\\ 
\mathbb{F}^{S,\infty}		&:= \F^{\eta^{S,\infty}, \binf, w^{S,\infty}, \mathcal{W}^{S,\infty}, \XNinf}\label{eqn: limit filtration standard},
\end{align}
where $\mathcal{W}^{I,\infty} := (\winf, \YIinf)$ and $\mathcal{W}^{S,\infty} = (\winf, \YNinf)$. Hence, we show that
\begin{itemize}
\item the stochastic process
\begin{equation}
\label{eqn: noise limit process major}
\tilde{\Theta}^{I, \infty}:= ( \eta^{I,\infty}, \binf, c^\infty, w^{I,\infty},\winf,\YIinf),
\end{equation}
taking values on $
\Omega^I_{\text{input}} :=\R \times \mathcal{C}([0,T],\R^2) \times \mathcal{D}([0,T],\R^3)$, is compatible with $\mathbb{F}^{I,\infty}$;
\item the stochastic process
\begin{equation}
\label{eqn: noise limit process standard}
\tilde{\Theta}^{S, \infty}:= ( \eta^{S,\infty}, \binf, w^{S,\infty},\winf,\YNinf), 
\end{equation}
taking values on $\Omega^S_{\text{input}} := \R \times \mathcal{C}([0,T],\R^2) \times \mathcal{D}([0,T],\R^2)$, is compatible with $\mathbb{F}^{S,\infty}$.
\end{itemize}
As we shall explain in \S \ref{section: compatibility condition weak limit} below, adding  the sequence of adjoint processes of the discretized game $(\YIn,\YNn)_{n\in\N}$ to the environment $(\wn)_{n\in\N}$ is necessary to guarantee the compatibility condition in the limit.

\item \label{step: S-III} Once compatibility is verified, for $p=I,S$, by Proposition \ref{lemma: iter random env} we are in the position to introduce the optimal control problems and 
the FBSDE system \eqref{eqn: FBSDE}
\begin{equation*}
(\bX^{p,\infty},\bY^{p,\infty},\bZ^{p,0,\infty},\bZ^{p,\infty},\bM^{p,\infty})
\end{equation*}
in the admissible setup
\[
\bigl(\Omega^\infty,\ItF^\infty,\prob^\infty, \F^{p, \infty}, \eta^{p, \infty}, (b^\infty,w^{p,\infty}) ,(\chi^{p,\varpi^{\infty}}, Y^{p,\infty})\bigr).
\]
Then we show in \S \ref{section: optimality weak equilibria} that
 $\bX^{p,\infty} = X^{p,\infty}$, corresponding to the weak limit of the sequence of solutions of the discretized optimal control problems.

\item \label{step: S-IV} Finally, by the uniqueness of the solutions to the FBSDE systems, we have  that $\bY^{p,\infty} = Y^{p,\infty}$, where $Y^{p,\infty}$ is the weak limit of the sequence $(Y^{p,n})_{n\in\N}$. Applying this property, we show in \S \ref{section: consistency condition} that the following relation holds: 
\begin{equation} 
\label{eqn: consistency condition for winf}
\winf_t = -\constantni\E^{\infty}\blq n_I\bar{\Lambda}^I\YIinf_t  + n_S\bar{\Lambda}^S \YNinf_t \Big| \ItF^{\w^\infty,\binf}_t\brq, \quad \forall t \in[0,T],\, \mathbb{P}^\infty-a.s. 
\end{equation}
\end{enumerate}

\subsubsection{Tightness of ${(X^{p,n},Y^{p,n},\wn)_{n\in\N}}$ in ${\mathcal{C}([0,T],\R) \times \mathcal{M}([0,T],\R^2)}$} 
\label{section: tightness} 

First of all, we prove the tightness of the sequences $(X^{p,n})_{n\in\N}$ in $\mathcal{C}([0,T],\R)$, for $p = I,S$.

\begin{proposition}
\label{prop:tightness:Xn}
For  $p = I,S$, the sequence $(\Xpn)_{n\in\N}$ is tight on $\mathcal{C}([0,T],\R)$.
\end{proposition} 

\begin{proof}
By Lemma \ref{lemma: boundedness fourth moment sequence of processes}, the processes $Y^{p,n}$ and $\varpi^n$ are uniformly bounded by $C_B$. As a consequence, by Assumption \ref{assumption: coefficients} and {the Burkholder-Davis-Gundy inequality, there exists a constant $C_4$ such that}, for all $0\leq t < s\leq T$,
\begin{align*}
 \bE\bigg[ |\Xpn_s - \Xpn_t|^4\bigg] 
 &\leq C_4 \bE \bigg[\Big(\int_t^s \bigl( \, |\bar{\Lambda}^p(Y^{p,n}_r+\wn_r) 
 +|l^p(r,\wn_r)| \bigr)dr  \Big)^4\bigg]   \\
 &\quad
 +  C_4 \bE \bigg[\Big( \int_t^s \bigl(|\sigma^{p,0}(r,\wn_r)|^2 + 
 |\sigma^p(r,\wn_r)|^2 \bigr)dr \Big)^2 \bigg]\\
 &\leq C_4 \bE\bigg[ \Big(\int_t^s \bigl( \, \bar{\Lambda}^p(|Y^{p,n}_r|+|\wn_r|) 
 +L(1+|\wn_r| ) \bigr)dr  \Big)^4\bigg]    \\
&\quad +  C_4 \bE \bigg[\Big( \int_t^s 2 L^2(1+|\wn_r| )^2 \big]dr \Big)^2\bigg] \\
 &\leq  C_4 \bigl(\bar{\Lambda}^p 2C_B + L(1+C_B) \bigr)^4 |s-t|^4 
 + C_4 2L^2 (1+C_B)^2 |s-t|^2.
\end{align*}
Therefore, Kolmogorov's criterion, together with the fact that the distribution of the initial condition is constant, provides the tightness of  $(\Xpn)_{n\in\N}$ on $\mathcal{C}([0,T],\R)$.
\end{proof}

We now prove the tightness of the sequences $(\wn)_{n\in\N}$, $(\Ypn)_{n\in\N}$, for $p = I,S$, in the Meyer-Zheng space. To do so, we apply \cite[Theorem 5.8]{kurtz1991random}, which is recalled in Appendix \ref{Appendix:MZ}. 

\begin{proposition} 
\label{prop: tightness Y and w}
Let  $\mathbb{G}^n$ be the subfiltration of $\bF$ given by $\G^n_t := \sigma\{\bar{V}_j\}$, for all $t \in[t_j,t_{j+1})$, and $j = 0,\dots,2^n-1$, where the vector $\bar{V}_j$ is given by \eqref{eqn: discr bm}.
Then, for every $n\in\N$, $\wn$ satisfies \eqref{eqn: Kurtz condition} for the filtration $\mathbb{G}^n$. Moreover, $\Ypn$ satisfies \eqref{eqn: Kurtz condition} for the filtration $\bF^{p}$, for all $n\in\N$ and $p=I,S$. 
\begin{proof}
Notice that $\wn$ is adapted to $\mathbb{G}^n$ as it satisfies \eqref{eqn: wn}. We prove that
\begin{equation}
\sup_{n\in\N} \left\{\bE\big[|\wn_T|\big] + V^{n}_T(\wn) \right\}< \infty,
\end{equation}
where
\begin{align*}
{V}^{n}_T(\wn)	&=  \sup_{K\geq 1} \sup_{0\leq s_0\leq \dots\leq s_K\leq T} \bE\bblq \sum_{j = 0}^{K-1} \Big|\bE\bigl[ \wn_{s_{j+1}} - \wn_{j}\mid\G^n_{s_j}\bigr]\Big|\bbrq.
\end{align*}
We first observe that $(\G^{n}_t)_{t\in[0,T]}$ is contained in $\bF^S\wedge \bF^I$. This condition, together with Assumption \ref{assumption: coefficients}, implies that
\begin{align*}
 \bI&\bE\blq\wn_{s_{j+1}} - \wn_{s_j}\bigm|\G^n_{s_j}\brq\bI 	
 \\
 &\leq \bI\bE\blq \constantni \Big(n_I\bar{\Lambda}^ I \bE\bigl[ \YIn_{s_{j+1}} \mid \G^n_{s_{j+1}}\bigr] + n_S\bar{\Lambda}^S \bE\bigl[\YNn_{s_{j+1}}\mid \G^n_{s_{j+1}}\bigr]   
 - n_I\bar{\Lambda}^I  \bE\bigl[ \YIn_{s_{j}} \mid\G^n_{s_j}\bigr] \\
 &\qquad -n_S\bar{\Lambda}^S \bE\bigl[\YNn_{s_{j}}\mid \G^n_{s_{j}}\bigr] \Big)\Bigm|\G^n_{s_{j}}\brq\bI\\ 
 &\leq \constantni\bI  n_S\bar{\Lambda}^S\ \bE\bigl[\YNn_{s_{j+1}} - \YNn_{s_{j}}\mid \G^n_{s_{j}}\bigr] + n_I\bar{\Lambda}^I\bE\bigl[\YIn_{s_{j+1}} - \YIn_{s_{j}}\mid \G^n_{s_{j}}\bigr] \bI\\
& = \constantni \bII n_S \bar{\Lambda}^S \ \bE\bblq \int_{s_{j}}^{s_{j+1}}\partial_x \bar{f}^S(s,\XNn_s, \wn_s)ds \Biggm| \G^n_{s_{j}}\bbrq \\
&\qquad+ n_I \bar{\Lambda}^I \ \bE\bblq \int_{s_{j}}^{s_{j+1}}\partial_x \bar{f}^I(s,\XIn_s, \wn_s,c_s)ds \Biggm| \G^n_{s_{j}}\bbrq\bII\\
&\leq \constantni \bE\bblq \int_{s_{j}}^{s_{j+1}}\bl n_S\bar{\Lambda}^S|\partial_x \bar{f}^S(s,\XNn_s, \wn_s)| + n_I\bar{\Lambda}^I|\partial_x \bar{f}^I(s,\XIn_s, \wn_s,c_s)|\br ds\Biggm| \G^n_{s_{j}}\bbrq \\ 
&\leq 2L(s_{j+1} - s_j).
 \end{align*}
This implies that ${V}^{n}_T(\wn)\leq 2LT$. 
Similarly, we prove that the sequences $(\YNn)_{n\in\N}$ and $(\YIn)_{n\in\N}$ satisfy condition \eqref{eqn: Kurtz condition} with respect to filtrations $(\F^{p})_{n\in\N}$, for $p=I,S$. 
By assumption, it holds that $\bE[|\Ypn_T|] = \bE[|\partial_xg^p(\Xpn_T, \chi^{p,\varpi^n}_T)|] \leq L$, for every $n\in\N$. Hence,  
\begin{align*}
{V}^{n}_T(\Ypn)	&= \sup_{K\geq 1} \sup_{s_0\leq \dots \leq s_K\leq T} \bE\bblq \sum_{j= 0}^{K - 1} \Big|\bE[\Ypn_{s_{j+1}} - \Ypn_{s_j}\mid\ItF^{p}_{s_j} ]\Big|\bbrq\\ 
&= \sup_{K\geq 1} \sup_{s_0\leq \dots \leq s_K\leq T}\bE\bblq \sum_{j= 0}^{K - 1}\bII\bE\bblq \int_{s_j}^{s_{j+1}} \partial_x\bar{f}^p (s,\Xpn_s, \chi^{p,\wn}_s) ds \mid\ItF^{p}_{s_j}\bbrq\bII\bbrq\\
 &\leq  \sum_{j= 0}^{K - 1}(s_{j+1} - s_j) L = TL.
\end{align*}
Since the estimate does not depend on $n$, we can take the supremum and obtain the claim.											
\end{proof}
\end{proposition}

Combining Propositions \ref{prop:tightness:Xn} and \ref{prop: tightness Y and w} and \cite[Theorem 5.8]{kurtz1991random}, we obtain the following result.

\begin{proposition}
\label{prop:tightness}
The sequence $\Theta^n := (b,c,\eta^I,w^I,\eta^S,w^S,\w^n,\YIn,\YNn,\XIn,\XNn)$
  is tight on the space 
$\Omega_{\text{input}}$ defined in \eqref{eqn: Omegainput}. 
\end{proposition}

\subsubsection{Compatibility for the limit optimal control problems}
\label{section: compatibility condition weak limit}
We are now allowed to introduce a limit in distribution of a subsequence of \eqref{eqn: tight sequence}, which we still denote $(\Theta^n)_n$; such limit takes the form \eqref{eqn: weak limit thetan}. As discussed in Remark \ref{rmk: problem compatibility condition}, in order to introduce the optimal control problems for the typical standard agent and the typical informed agent, we need to guarantee the compatibility condition between $(\eta^{p,\infty},\binf,w^{p,\infty},\chi^{p,\winf})$ and the filtration generated by the weak limit $(\eta^{p,\infty},\binf,w^{p,\infty},\chi^{p,\winf},X^{p,\infty})$. This property does not hold in general, because of the presence of $X^{p,\infty}$ in the filtration. For this reason, we consider the sequence of processes 
\begin{displaymath}
\mathcal{W}^{p,n}:= (\wn, \Ypn), \quad n\in \N.
\end{displaymath} 
In analogy to \cite[Chapter 3]{carmona2018probabilistic2}, we call $\mathcal{W}^{p,n}$ \emph{lifted environment of population $p$}, for $p=S,I$. As we are going to show below, lifting the environment $\w^n$ enables us to guarantee the compatibility condition in the limit in distribution. 

By Proposition \ref{prop:tightness}, the sequence $(\mathcal{W}^{p,n})_{n\in\N}$ is tight. As a consequence, we can consider a weak limit in $\mathcal{M}([0,T],\R^2)$. The weak limit $\mathcal{W}^{p,\infty} = (\winf, Y^{p,\infty})$ admits a \cadlag version (\cite[Theorem 5.8]{kurtz1991random}). We point out that, for the moment, we cannot conclude that $Y^{p,\infty}$ is the adjoint process of the solution $X^{p,\infty}$ of the optimal control problem of the typical agent defined on $\setupinf$. Indeed, since the compatibility condition is not yet guaranteed, we cannot define the optimal control problem on $\setupinf$. In particular, we first need to prove the following result.\footnote{In the proof of Lemma \ref{lemma: compatibility condition limit ocp}, we point out that Proposition \ref{prop: Xpinf satisfies FSDE} is used only to identify the limiting filtration and prove compatibility, not to prove the convergence of the sequence of controls.}

\begin{lemma}
\label{lemma: compatibility condition limit ocp}
Let us consider the above setup. Then, for $p = I,S$, let
\begin{equation}
\label{eqn: optimal controls limit game}
\hat{\alpha}^{p,\infty}_t := -\overline{\Lambda}^p(\Ypinf_t +\winf_t)
\end{equation} 
and let $\hat{\alpha}^{p,n}$ be the optimal control for the discretized setting, for $n\in\N$.  Then the sequence $(\hat{\alpha}^{p,n})_{n\in\N}$ converges in distribution on $\mathcal{M}([0,T],\R)$ to a \cadlag process that is $\mathbb{P}^\infty$-indistinguishable from  $\hat{\alpha}^{p,\infty}$.
Moreover, the process $\tilde{\Theta}^{p,\infty}$ introduced in \eqref{eqn: noise limit process major} and \eqref{eqn: noise limit process standard} is compatible with the filtration $\F^{p,\infty}$ which coincides with $\F^{\tilde{\Theta}^{p,\infty}, X^{p,\infty}}$. 
\end{lemma}

\begin{proof}
As in the proof of \cite[Proposition 3.12]{carmona2018probabilistic2}, we aim at expressing the filtration $\F^{p,\infty}$ as the completion of the natural filtration of a process that does not explicitly depend on $X^{p,\infty}$. To do so, we exploit the lifted environment $\mathcal{W}^{p,\infty}$. The  optimal controls of the typical agents of the two populations obtained in the discretized setting form tight sequences in $\mathcal{M}([0,T],\R)$ since those controls are of the form of $\hat{\alpha}^{p,n}_t := \hat{\alpha}^p(\Ypn_t, \wn_t)$, where the function $\hat{\alpha}^p$ is given in \eqref{eqn: optimal controls}. 
For the same reason, applying the continuous mapping theorem and 
\cite[Lemma 3.5]{carmona2018probabilistic2}, 
the sequence $(\hat{\alpha}^{p,n})_{n\in\N}$ converges in distribution on $\mathcal{M}([0,T],\R)$ to a process $\tilde{\alpha}^{p,\infty}$ such that 
$\hat{\alpha}^{p,\infty}_t := \tilde{\alpha}^{p,\infty}_t$, for  a.e. $t\in[0,T]$, $\mathbb{P}^\infty$-a.s.  
Moreover, since both $(\hat{\alpha}^{p,\infty}_t)_{t\in[0,T]}$ and $(\tilde{\alpha}^{p,\infty})_{t\in[0,T]}$ are \cadlag processes, they coincide for every $t \in[0,T]$ (see \cite[Section IV. 44]{dellacherie1978probabilities}). This proves the first part of the proposition and also implies that $\hat{\alpha}^{p,\infty}$ is adapted to $\F^{p,\infty}$. 

For the second claim, we rely on Proposition \ref{prop: Xpinf satisfies FSDE} below, which shows in particular that  $X^{p,\infty}$ is adapted to $\F^{\Psi^{p,\infty}}$, with $\Psi^{p,\infty}:=  (\tilde{\Theta}^{p,\infty}, \hat{\alpha}^{p,\infty})$ taking values on $\Omega^p_{\text{input}} \times \mathcal{M}([0,T],\R)$. 
Since $\hat\alpha^{p,\infty}$ is a measurable functional of $\tilde\Theta^{p,\infty}$, we have $\mathbb F^{\Psi^{p,\infty}}=\mathbb F^{\tilde\Theta^{p,\infty}}$. Hence, the adaptedness of $X^{p,\infty}$ to $\mathbb F^{\Psi^{p,\infty}}$ implies that $\F^{p,\infty}=\mathbb F^{\tilde\Theta^{p,\infty},X^{p,\infty}}=\mathbb F^{\tilde\Theta^{p,\infty}}$. This yields the desired compatibility.
\end{proof}

We are left to prove the following main convergence result.

\begin{proposition}
\label{prop: Xpinf satisfies FSDE}
In the above setting, $X^{p,\infty}$ satisfies the following forward equation: 
\begin{small}
\begin{equation}
\label{eqn: limit SDE state}
X^{p,\infty}_t =  \eta^{p,\infty} + \int_0^t (\hat\alpha^{p,\infty}_s + l^p(s,\winf_s))ds  + \int_0^t \sigma^{p,0}(s,\winf_s) d\binf
_s + \int_0^t \sigma^p(s,\winf_s) dw^{p,\infty}_s,
\end{equation}
\end{small} 
where $\hat\alpha^{p,\infty}$ is given by \eqref{eqn: optimal controls limit game} and is the limit in distribution of $\hat{\alpha}^{p,n}$. As a consequence, $X^{p,\infty}$ is $\F^{\eta^{p,\infty}, \binf, w^{p,\infty}, \winf, \hat\alpha^{p,\infty}}$-adapted. 
\end{proposition}

    The proof is given in Appendix \ref{appendix: measurability of Xinf} and is based on the steps described in \cite[Proposition 3.11]{carmona2018probabilistic2}. The main difference is given by the convergence in distribution of the random environment $(\wn)_{n\in\N}$, that in the case presented here is on the Meyer-Zheng space, while in the framework of \cite{carmona2018probabilistic2} the convergence of the random environment $(\mu^n)_{n\in\N}$ is on $\mathcal{D}([0,T],\R)$ endowed with the $J1$ topology. This difference affects mainly the first step of the proof of \cite[Proposition 3.11]{carmona2018probabilistic2} regarding the structure of $X^{p,\infty}$, which must satisfy the state equation.

\subsubsection{Optimality of the weak limit} 
\label{section: optimality weak equilibria}
In this section, we prove stability of the discretized equilibria when the number of agents goes to infinity. What is left to prove is that $\hat{\alpha}^{p,\infty}$ and the corresponding process $X^{p,\infty}$, obtained as limits in distribution, are indeed optimal for the control problem defined with respect to the filtration $\F^{p,\infty}$.  
Consider then the problem


\begin{equation}
\label{eqn: limit ocp}
\begin{aligned}
&\inf_{\alpha^p \in \mathbb{H}^2(\mathbb{F}^{p,\infty})} J^{p,\winf}(\alpha^p),\\
&J^{p,\winf}(\alpha^p) := \E^{\infty} \bblq \int_0^T f^p(s, X_s, \winf_s, \alpha_s,\chi^{p,\winf}_s) ds + g^p(X_T, \chi^{p,\winf}_T) \bbrq,\\
&\text{subject to} \\
&
\begin{cases} 
dX^p_t = (\alpha^p_t + l^p(t,\winf_t))dt + \sigma^{p,0}(t,\winf_t) d\binf_t + \sigma^{p}(t,\winf_t) dw^{p,\infty}_t,\\
X^p_0 = \eta^{p,\infty}.
\end{cases}
\end{aligned}
\end{equation}
In view of Proposition \ref{lemma: iter random env}, the unique optimal control is given by 
\begin{equation}
\label{eqn:optimal_bar}
    \bar{\alpha}^{p,\infty}_t = -\bar{\Lambda}^p (\bar{Y}_t^{p,\infty} +\varpi_t^{\infty}),
\end{equation}
where 
$
(\bX^{p,\infty},\bY^{p,\infty},\bZ^{p,0,\infty},\bZ^{p,\infty},\bM^{p,\infty})
$
solves the FBSDE 
\begin{equation}
\label{eqn: FBSDE typical agent limit game}
\begin{cases} 
d\bar X^{p,\infty}_t 	&= \big(-\bar{\Lambda}^p (\bar{Y}_t^{p,\infty} +\varpi_t^{\infty}) + l^p(t,\winf_t) \big) dt + \sigma^{p,0}(t,\winf_t) d\binf_t + \sigma^p(t,\winf_t) dw^{p,\infty}_t,\\ 
 \bar X^{p,\infty}_0 	&= \eta^{p, \infty},\\ 
d \bar Y^{p,\infty}_t	&= -\partial_x\bar{f}^p(t, \bar X^{p,\infty}_t, \chi^{p,\winf}_t) dt + \bar Z^{p,0,\infty}_t db^{\infty}_t + \bar Z^{p,\infty}_t dw^{p,\infty}_t + d\bar M^{p,\infty}_t,\\ 
 \bar Y^{p,\infty}_T 	&= \partial_xg^p(\bar X^{p,\infty},\chi^{p,\varpi^\infty}_T),
\end{cases}
\end{equation}
in the admissible setup $(\Omega^\infty,\ItF^\infty,\prob^\infty, \F^{p, \infty}, \eta^{p, \infty}, (b^\infty,w^{p,\infty}) ,(\chi^{p,\varpi^{\infty}}, Y^{p,\infty}))$. Observe that, by Definition \ref{def: admissibility}, in order to prove that the setup is admissible, it remains to show that $(b^\infty,w^{p,\infty})$ is a Brownian motion for the filtration $\F^{p, \infty}$. This follows by standard arguments (we refer the reader to \cite[Lemma 2.41]{tesilanaro} for a detailed proof).

\begin{proposition}
\label{prop: stability weak equilibria}
For $p=I,S$, consider the optimal control problem in the discretized admissible setups $(\bOm, \bItF, \bprob, \bF^p, \eta^p, (b, w^p), (\chi^{p,\varpi^n}))$ defined above and let $\hat{\alpha}^{p,n}$ be the optimal control; note that the filtration does not depend on $n$. Let the limit in distribution and the subsequence be given as above.
Then the sequence $(\hat{\alpha}^{p,n})_{n\in\N}$ converges in distribution on $\mathcal{M}([0,T],\R)$ to a càdlàg process $\tilde{\alpha}^{p,\infty}$ (defined on the same limit probability space) such that $\tilde{\alpha}^{p,\infty}_t= \bar{\alpha}^{p,\infty}_t$ for every $t\in[0,T]$, $\mathbb{P}^\infty$-a.s, where $\bar{\alpha}^{p,\infty}$ is given by \eqref{eqn:optimal_bar}. 
As a consequence, we have that 
\begin{equation}
X_t^{p,\infty}= \bar X_t^{p,\infty}  \mbox{ and } 
Y_t^{p,\infty}= \bar Y_t^{p,\infty} 
\qquad \forall t\in[0,T], \quad \mathbb{P}^\infty - a.s.
\end{equation}
\end{proposition}

\begin{proof}
 The proof is analogous to that of \cite[Proposition 3.11]{carmona2018probabilistic2}. Since the steps are quite similar, we refer to \cite[Proposition 2.28]{tesilanaro} for the detailed technical computations. We just remark that, once compatibility is verified thanks to Lemma \ref{lemma: compatibility condition limit ocp}, a key step is to use Proposition \ref{prop: Xpinf satisfies FSDE}, which we prove in Appendix \ref{appendix: measurability of Xinf}, for the optimal control and for another control in $\F^{p, \infty}$. Note also that Lemma  \ref{lemma: boundedness fourth moment sequence of processes} implies the required uniform square integrability of the processes. 
 To prove the last claim, note that the uniqueness of the optimal control implies that $\tilde{\alpha}^{p,\infty}_t= \bar{\alpha}^{p,\infty}_t$, for a.e. $t\in[0,T]$ $\mathbb{P}^\infty$-a.s, and thus also for all $t\in[0,T]$, as they are càdlàg processes. Then Lemma \ref{lemma: compatibility condition limit ocp} implies that $\tilde{\alpha}^{p,\infty}_t= \hat{\alpha}^{p,\infty}_t$ which yields 
 \[
 \bar{Y}_t^{p,\infty} +\varpi_t^{\infty} = 
 Y_t^{p,\infty} +\varpi_t^{\infty}, \qquad \forall t\in[0,T], \quad \mathbb{P}^\infty - a.s.
 \]
 Therefore, we obtain that $Y_t^{p,\infty}= \bar Y_t^{p,\infty}$ and hence $X_t^{p,\infty}= \bar X_t^{p,\infty}$ $\mathbb{P}^\infty$ - a.s. for every $t\in[0,T]$ by pathwise uniqueness of the solution to the SDE \eqref{eqn: limit SDE state}. 
\end{proof}

\subsection{Consistency condition for the limit game} 
\label{section: consistency condition}

In this section, we finally show that the procedure described at the beginning of \S \ref{section: stability} provides a weak lifted mean-field equilibrium.  To achieve this, it remains to show that the consistency condition \eqref{eqn: consistency condition for winf} holds for the limit price process $\varpi^\infty$. We remark that such condition is quite different from the consistency condition for standard mean-field games and, therefore, our proof differs from the one in \cite{carmona2018probabilistic2}.  


\begin{theorem}
\label{thm: consistency condition weak limit}
The component $\winf$ of $\Thetainf$ introduced in \eqref{eqn: weak limit thetan}, defined as the \cadlag version of the limit in distribution on $\mathcal{M}([0,T],\R)$ of $(\wn)_{n\in\N}$, satisfies \eqref{eqn: consistency condition for winf}.
\end{theorem}

\begin{proof}
We recall that $\wn$ was defined as the solution of \eqref{eqn: wn}. Let us introduce the process
\begin{equation}\begin{aligned}
\label{eqn: ch2 - mathcalV}
\mathcal{V}^{n}_t	&:= \bar{V}^{n}_i,\quad t\in[t_i,t_{i+1}),\ \forall i = 0,1\dots,2^n -1,\\
\mathcal{V}^{n}_T	&:= \bar{V}^{n}_{2^n - 1}.
\end{aligned}
\end{equation}
As a consequence, 
\[
\wn_t= -\constantni \bE\Bigl[n_I\bar{\Lambda}^I\YIn_t + n_S\bar{\Lambda}^S \YNn_t \bigm|\mathcal{H}^{n}_t\Bigr], 
\quad\text{ for all }t\in[0,T],
\]
where $\mathcal{H}^n_t:= \sigma\{ \mathcal{V}^{n}_s:\ s\leq t\}$. In particular, $\wn_t$ is $\mathcal{H}^{n}_t$-measurable for each $t\in[0,T]$. Thus, we can define $\mathcal{G}^{n}_t:= \sigma\{ \wn_s,\mathcal{V}^{n}_s:\ s\leq t\} = \mathcal{H}^{n}_t$, for each $t\in[0,T]$. We notice that
 \begin{displaymath}
 \bE\blq\wn_t +\constantni( n_I\bar{\Lambda}^I\YIn_t + n_S\bar{\Lambda}^S\YNn_t) \bigm|\mathcal{G}^{n}_t\brq = 0, \quad\forall t\in[0,T],\ n\in\N.
 \end{displaymath}
This is equivalent to
\begin{equation}
 \label{eqn: consistency condition limit}
\bE\blq 
h( \wn_{t_1}, \dots, \wn_{t_M}, \mathcal{V}^n_{t_1}, \dots, \mathcal{V}^n_{t_M})
\big(\wn_t +\constantni( n_I\bar{\Lambda}^I\YIn_t + n_S\bar{\Lambda}^S\YNn_t)\big)\brq = 0,
\end{equation}
for any $n$ and any $t\in[0,T]$, for every $M\in \N$, for each $0\leq t_1< \dots< t_M\leq t$ and for any bounded and continuous function $h:\R^{2M}\rightarrow \R$.  
We pass to the limit in the above equation. 
We recall that $\mathcal{V}^n$ is a discretization of the Brownian motion $b$ defined on the canonical space $\bOm^0$ in \S \ref{section: discretization setup}. It is not hard to show that $\mathcal{V}^n$ converges to $b$ in probability in $\mathcal{D}([0,T], \R)$ (see \cite[Lemma 2.31]{tesilanaro} for a detailed proof).
Since $( \wn, Y^{I,n}, Y^{S,n},b)$ converges to $( \w^\infty, Y^{I,\infty}, Y^{S,\infty},b^\infty)$ in law on $\mathcal{M}([0,T], \R^3) \times \mathcal{C}([0,T], \R)$, we obtain that, up to a subsequence, $\Phi^n:=( \wn, Y^{I,n}, Y^{S,n}, \mathcal{V}^n, b)$ converges to $\Phi^\infty:=( \w^\infty, Y^{I,\infty}, Y^{S,\infty}, b^\infty, b^\infty)$ in law on $\mathcal{M}([0,T], \R^3) \times \mathcal{D}([0,T], \R) \times \mathcal{C}([0,T], \R)$.  
Hence, \cite[Thm 5]{meyer1984tightness} gives  that there exists a further subsequence $(n_k)_k$ and a set $I\subset [0,T]$ of full Lebesgue measure such that the finite dimensional distributions of $(\Phi^{n_k}_s)_{s\in I}$ converge to those of $(\Phi^{\infty}_s)_{s\in I}$, as $k\to \infty$.
Therefore we can pass to the limit in \eqref{eqn: consistency condition limit} and obtain 
\begin{equation*}
\E^\infty\blq 
h( \w^\infty_{t_1}, \dots, \w^\infty_{t_M}, b^\infty_{t_1}, \dots, b^\infty_{t_M})
\big(\w^\infty_t +\constantni( n_I\bar{\Lambda}^I Y^{I,\infty}_t + n_S\bar{\Lambda}^SY^{S,\infty}_t)\big)\brq = 0,
\end{equation*}
 for every $M\in \N$ and $t\in[0,T]$,  
 for each $0\leq t_1< \dots< t_M\leq t$ with $t, t_1,\dots,t_M \in I$ and for any bounded and continuous $h:\R^{2M}\rightarrow \R$. Since trajectories are right continuous, the same property also holds for any $t, t_1, \dots, t_M$ in $[0,T]$, thus proving the claim. 
\end{proof}

We have now all the ingredients to prove Theorem \ref{thm: existence of solutions}.

\begin{proof}[Proof of Theorem \ref{thm: existence of solutions}]
In Proposition \ref{prop:tightness}, we proved that the sequence $(\Theta^n)_{n\in\N}$ introduced in \eqref{eqn: tight sequence} is tight on $\Omega_{\text{input}}$. Hence, we are allowed to introduce a weak limit of a subsequence, defined on a suitable probability space $(\Omega^{\infty}, \ItF^{\infty}, \prob^{\infty})$ and denoted by $\Theta^{\infty}$ as the one introduced in \eqref{eqn: weak limit thetan}. 
Hence, 
\begin{equation*}
(\Omega^{\infty}, \ItF^{\infty}, \prob^{\infty}, \F^{,\infty}, (\eta^{I,\infty},\eta^{S,\infty}), (\binf,w^{I,\infty}, w^{S,\infty}, c^{\infty}), (\winf,Y^{I,\infty}, Y^{S,\infty}))
\end{equation*}
 is a weak lifted mean-field equilibrium in the sense of Definition \ref{def: mf price process}: 
properties (1)–(2) follow from the construction and Lemma \ref{lemma: compatibility condition limit ocp},
property (3) follows from Theorem \ref{thm: consistency condition weak limit} and property (4) follows from Proposition \ref{prop: stability weak equilibria}.
\end{proof}

\section{An asymptotic version of the market clearing condition} 
\label{section: weak market clearing}

In this section, we investigate the relation between a weak mean-field equilibrium 
and the market clearing condition in an economy populated by finitely many agents. More precisely, we study whether the mean-field equilibrium price provides an asymptotic version of the finite-player market clearing relation when the number of agents tends to infinity and each agent solves her stochastic optimal control problem taking the mean-field price process $\w$ 
as exogenous.
{The purpose of this section is instead to justify, under stronger assumptions, the interpretation of the mean-field equilibrium price as an asymptotic market clearing price.}

We show that the asymptotic version of the market clearing condition is satisfied by a suitable modification of the (weak unlifted) equilibrium price process introduced in Definition \ref{def: unlifted mf price process}. More precisely, we consider here a price process $\w$ satisfying conditions (1) and (2) of Definition \ref{def: unlifted mf price process} and a different consistency condition. In particular, we focus on the modified version of \eqref{eqn: mc price process mf} in which the adjoint processes $Y^I$ and $Y^S$ are respectively replaced by the projected adjoint processes $\widetilde Y^I$ and $\widetilde Y^S$ of Definition \ref{def: unlifted mf price process}, namely
\begin{equation}
    \label{eqn: project eq price v1}
    \w_t = -(n_I\bar{\Lambda}^I + n_S\bar{\Lambda}^S)^{-1} \bl n_I\bar{\Lambda}^I\E\bigl[ \tilde{Y}^{I}_t\bigm| \ItF^{\w,B,C}_t\bigr] + n_S\bar\Lambda^S \E\bigl[ \tilde{Y}^{S}_t\bigm| \ItF^{\w,B}_t\bigr]\br,\   \forall t\in[0,T],\  \mathbb{P}-a.s.
\end{equation}

\begin{remark}
{The consistency condition in \eqref{eqn: project eq price v1} should be understood as a technical modification of Definition \ref{def: unlifted mf price process}, introduced in order to obtain the probabilistic structure needed for the asymptotic analysis below. In particular, as explained in Remark \ref{rmk: extra assumption for asymptotic mk clearing} below, it yields the conditional i.i.d. structure exploited in the proof of Theorem \ref{thm: weak mk clearing condition}. We stress, however, that the resulting notion of equilibrium is narrower than Definition \ref{def: mf price process} and should not be interpreted as the most general economic specification of our framework. 
}
%
\end{remark}

\subsection{Weak formulation of the finite-player economy under asymmetric information}

In order to pass from the mean-field equilibrium to a weak $N$-player economy, the probabilistic structure of the equilibrium must be carefully handled. 
Since the total randomness affecting the economy (and correlated with the process $\varpi$ satisfying \eqref{eqn: project eq price v1}) cannot be fixed a priori, we must define the $N$-agent economy on a probability space sufficiently rich to guarantee the existence of $\w$. Therefore, we shall define a \emph{weak formulation} of the market with $N$ agents.
To this end, it is fundamental to consider the unlifted notion of equilibrium of Definition \ref{def: unlifted mf price process}.
Let then 
\begin{equation*}
\big(\Omega, \ItF, \prob, \F, (\xi^I, \xi^S), (B,W^I,W^S), (\w,C ) \big)
\end{equation*}
be a weak unlifted mean-field equilibrium in the sense of Definition \ref{def: unlifted mf price process}, where, as discussed before, condition \eqref{eqn: conclusion winf projected} is replaced by \eqref{eqn: project eq price v1}. In particular, since  there is no additional source of randomness given by the presence of the adjoint processes $\bar{Y}^I,\bar{Y}^S$, we can follow the approach of \cite[Theorem 3.13]{carmona2018probabilistic2} to transport the mean-field equilibrium introduced in Definition \ref{def: unlifted mf price process} on the extended canonical space 
$(\bOm:=\bOm^0\times\bOm^I\times\bOm^S,\bar{\ItF},\bar{\prob})$ defined as 
\begin{equation} 
\label{eqn: canonical space}
\begin{cases}
\bar{\Omega}^0 := \mathcal{C}([0,T],\R) \times \mathcal{D}([0,T],\R^2),\\ 
\bar{\Omega}^p := \R\times \mathcal{C}([0,T],\R),\quad p= I,S,
\end{cases}
\end{equation}
where we denote the canonical process on $\bOm$ by $(b,c,\wmf,\eta^I,w^I,\eta^S,w^S)$ and by $\bprob$ the law of the equilibrium:
\begin{equation}
\label{eqn: transfer canonical distr}
\prob\circ( B,C, \w,\xi^I, W^I, \xi^S, W^S)^{-1} = \bprob\circ(b,c,\wmf,\eta^I,w^I,\eta^S,w^S)^{-1}.
\end{equation} 
Since there are no additional stochastic processes $\bar{Y}^I$  and $\bar{Y}^S$, the environment $\chi^{p,\w}$, for $p = I,S$, is independent of the idiosyncratic noises of the typical agents of the two subpopulations. This implies that the probability measure defined by the mean-field equilibrium on the canonical space is given by the product $\bprob^0 \otimes \bprob^I\otimes\bprob^S$, where
\begin{equation*}
\bprob^0 := \prob^{-1}\circ(B,\w,C),\quad
\bprob^I := \prob^{-1}\circ(\xi^I,W^I),\quad
\bprob^S := \prob^{-1}\circ(\xi^S,W^S).
\end{equation*}
As a consequence, we can define the economy with finitely many agents on (copies of) the same probability space on which the mean-field equilibrium price process (in the sense of Definition \ref{def: unlifted mf price process}) is defined. Applying Proposition \ref{lemma: iter random env}, we denote by 
\begin{equation*}
({X}^p,{Y}^p,{Z}^{p,0},{Z}^p,{M}^p) 
\quad p = I,S,
\end{equation*}
the solution to the FBSDE \eqref{eqn: FBSDE} defined on the admissible setup 
$(\bOm^0\times\bOm^I\times\bOm^S,\bItF,\bprob, \bF^{p})$
carrying the process $(\eta^p, (b,w^p), \chi^{p,\wmf})$, where $\bF^p= \F^{\eta^p, b, w^p, \chi^{p, \wmf}}$ and $\bF= \bF^I \vee \bF^S$ is the (augmented) canonical filtration. 
Since the distribution of $({X}^p,{Y}^p,{Z}^{p,0},{Z}^p,{M}^p)$ coincides with that of the mean-field equilibrium, for $p = I,S$, we deduce that $\bar{\mathbb{P}}$-a.s.
\begin{equation}
\label{eqn: equation wmf weak mc}
\wmf_t = -\constantni\Big(n_I\bar{\Lambda}^I\bE\bigl[  Y^I_t \mid \ItF^{b,\varpi^{\rm mf},c}_t \bigr] +  n_S\bar{\Lambda}^S \bE \bigl[Y^S_t\mid \ItF^{b,\varpi^{\rm mf}}_t\bigr]\Big) ,\quad \forall t\in[0,T].
\end{equation}
 In order to build the $N_I + N_S$ agents economy on a probability space rich enough to contain the solution of \eqref{eqn: equation wmf weak mc}, we consider ${N_p}$ copies of the space $(\bOm^p,\bItF^p,\bprob^p,\bar{\F}^p)$ 
 introduced in \eqref{eqn: canonical space}, for $p = I,S$,
 denoted by $(\Omega^{p,j},\ItF^{p,j},\bprob^{p,j},\bar{\F}^{p,j})_{j=1}^{N_p}$, where 
 $\bF^{p,j}=\F^{b,\eta^{p,j},w^{p,j},\chi^{p,\wmf}}$. 
 By construction, for each $j = 1,\dots, N_p$, the space $(\bOm^{p,j},\bItF^{p,j},\bprob^{p,j},\bar{\F}^{p,j})$ is rich enough to support a one-dimensional Brownian motion $w^{p,j}=(w^{p,j}_t)_{t\in[0,T]}$ and a random variable $\eta^{p,j}$ distributed as $\eta^p$ and independent of $w^{p,j}$. Thus, we can define the product space:
\begin{equation}
\label{eqn: N player setup}
\begin{cases}
\bar{\Omega}_N&:= \bOm^0\times\bOm^{I,1}\times \bOm^{I,2}\times\dots\times\bOm^{I,N_I}\times\bOm^{S,1}\times \bOm^{S,2}\times\dots\times\bOm^{S,N_S},\\ 
(\bar{\ItF}_N,\bar{\mathbb{P}}_N)		
&:= (\bItF^0\otimes \bItF^{I,1}\otimes\dots\otimes \bItF^{S,N_S},\bprob^0\otimes\dots \bprob^{S,N_S}),\\ 
\bar{\mathbb{F}}_N 				
&:= (\bItF^0_t\otimes\dots\otimes \bItF^{S,N_S}_t)_{t\in[0,T]}.
\end{cases}
\end{equation}
The $j$-th agent of population $p$ 
must solve her control problem applying controls belonging to $ \mathbb{H}^2(\bF^{p,j})$.
Thus, her optimal control is given by 
\begin{equation}
\label{eqn: optimal control N agent}
\hat{\alpha}^{\text{mf};p,j}_t	:=-\bar{\Lambda}^p(Y^{p,j}_t + \wmf_t),\quad t\in[0,T],
\end{equation}
where $({X}^{p,j},{Y}^{p,j},{Z}^{p,0,j},{Z}^{p,j},{M}^{p,j})$ solves the FBSDE \eqref{eqn: FBSDE:N}, with $\wmf$ playing the role of $\w$
in the admissible setup 
$((\bOm_N,\bar{\ItF}_N,\bar{\mathbb{P}}_N, \bF^{p,j}), (\eta^{p,j}, (b,w^{p,j}), \chi^{p,\wmf}) )$.

\subsection{The weak asymptotic market clearing condition}

We have now all the ingredients to prove the main result of this section. Let the tuple 
\[
\big(\bOm_N,\bar{\ItF}_N,\bar{\mathbb{P}}_N, \bF^{p,j}, \eta^{p,j}, (b,w^{p,j}), \chi^{p,\wmf}, {X}^{p,j},{Y}^{p,j},{Z}^{p,0,j},{Z}^{p,j},{M}^{p,j}
\big),
\quad \text{for $p=I,S$ and $j=1,\dots, N_p$},
\]
be the weak solution of the economy with $N=N_I+N_S$ agents, constructed in the above subsection. {We recall that the price process $\wmf$ is assumed to arise from a mean-field equilibrium in the sense of Definition \ref{def: unlifted mf price process} with consistency condition \eqref{eqn: project eq price v1}.
The next theorem shows that $\wmf$ satisfies an asymptotic market clearing condition in the corresponding weak finite-player economy with $N=N_I+N_S$ agents, thus providing a conditional asymptotic validation of the mean-field equilibrium price. Unlike Theorems \ref{thm: existence of solutions} and \ref{thm: existence of unlifted mean-field price process}, however, the next theorem does not provide an existence result for equilibria satisfying \eqref{eqn: project eq price v1}.
The proof only uses the existence of a weak unlifted mean-field equilibrium satisfying Definition \ref{def: unlifted mf price process} with the modified consistency condition \eqref{eqn: project eq price v1}. 
The proof of the next theorem does not use Assumption \ref{assumption: affine target functions} directly; it only uses the existence of an unlifted weak equilibrium satisfying \eqref{eqn: project eq price v1}.
}

\begin{theorem}
\label{thm: weak mk clearing condition}
{Under Assumption \ref{assumption: coefficients}, suppose that there exists a weak unlifted mean-field equilibrium in the sense of Definition \ref{def: unlifted mf price process}, with condition \eqref{eqn: conclusion winf projected} replaced by \eqref{eqn: project eq price v1}, and let $\varpi^{mf}$ be the associated price process. 
Then, for the weak $N$-agent economy constructed above,} 
 there exists a constant $C>0$ such that 
\begin{equation}
\label{eqn: asymptotic market clearing}
\bE_N\bblq \int_0^T \Bigg|\frac{1}{N_I+N_S}\sum_{p = I,S} \sum_{j=1}^{N_p}\hat{\alpha}^{\text{mf};p,j}_{t}\Bigg|^2 dt \bbrq\leq \frac{C}{N_I + N_S},
\end{equation}
where $\hat{\alpha}^{\text{mf};p,j}$ is given by \eqref{eqn: optimal control N agent} 
and $\bar{\E}_N$ denotes the expectation with respect to the measure $\bar{\mathbb{P}}_N$.
\end{theorem}
\begin{proof}
First of all, we notice that, by \eqref{eqn: equation wmf weak mc},
\begin{align}
\label{eqn: N equilibrium market clearing}
\frac{1}{N}\sum_{p = I,S} \sum_{j=1}^{N_p}\hat{\alpha}^{\text{mf};p,j}_{t}		
&= \sum_{p = I,S}\frac{n_p}{N_p} \sum_{j=1}^{N_p}\hat{\alpha}^{\text{mf};p,j}_{t}	\\
&= \sum_{p = I,S}\frac{n_p}{N_p} \sum_{j = 1}^{N_p}-\bar{\Lambda}^p(Y^{p,j}_t + \wmf_t)\nonumber\\
&= \sum_{p = I,S} -n_p\bar{\Lambda}^p \bbl \frac{1}{N_p}\sum_{j = 1}^{N_p} Y^{p,j}_t \bbr+\bar{\E}\bigl[n_I \bar{\Lambda}^IY^I_t \mid \ItF^{\varpi^{\rm mf},b,c}_t\bigr] + \bar{\E} \bigl[n_S\bar{\Lambda}^SY^S_t\mid\mffil_t\bigr] \nonumber\\ 
&= \sum_{p = I,S}- n_p\bar{\Lambda}^p \bbl \frac{1}{N_p}\sum_{j = 1}^N Y^{p,j}_t - \bar{\E}\bigl[Y^p_t \mid\ItF^{b,\chi^{p,\wmf}}_t\bigr]\bbr.\nonumber
\end{align}
Applying the Yamada-Watanabe theorem (\cite[Theorem 1.33]{carmona2018probabilistic2}) to the FBSDEs \eqref{eqn: optimal control N agent} defined on $(\bOm_N,\bar{\ItF}_N,\bar{\mathbb{P}}_N, \bF^{p,j})$, for $p = I,S$, we obtain the existence of two measurable functions $\Psi^I$ and $\Psi^S$, introduced in \eqref{eqn: YW functions}, such that 
\begin{equation}
\label{eqn: YW dynamics}
(X^{p,j}_t,Y^{p,j}_t, Z^{p,0,j},Z^{p,j},M^{p,j})_{t\in[0,T]}:= \Psi^p(\eta^{p,j}, b,\chi^{p,\wmf},w^{p,j}),\quad j = 1,\dots,N_p.
\end{equation}
Conditionally on the common environment $\F^{b,\chi^{p,\wmf}}$, the processes $(Y^{p,j})_{j = 1,\dots,N_p}$ are i.i.d., since by \eqref{eqn: YW dynamics} each $Y^{p,j}$ is a measurable functional of the common variables $(b,\chi^{p,\varpi^{\rm mf}})$ and of the idiosyncratic pair $(\eta^{p,j},w^{p,j})$. 
Hence, for $p = I,S$,
\begin{align*}
\bar{\E}_N\Big[Y^{p,j}_t\Bigm|\ItF^{b,\chi^{p,\wmf}}_t\Big] &= \bar{\E}_N\Big[Y^{p,1}_t\Bigm|\ItF^{b,\chi^{p,\wmf}}_t\Big]
= \bar{\E}\Big[Y^{p}_t\Bigm|\ItF^{b,\chi^{p,\wmf}}_t\Big]
,\quad \forall t\in[0,T],\quad \forall j=1,\dots,N_p.
\end{align*}
We can then replace $\bar{\E}[Y^{p,j}_t\mid\ItF^{b,\chi^{p,\wmf}}_t]$ in the last term of \eqref{eqn: N equilibrium market clearing}, obtaining
\begin{displaymath}
\frac{1}{N}\sum_{p = I,S} \sum_{j=1}^{N_p}\hat{\alpha}^{\text{mf};p,j}_{t}=-\sum_{p = I,S} n_p\bar{\Lambda}^p\bbl \frac{1}{N_p} \sum_{j = 1}^{N_p} Y^{p,j}_t - \bar{\E}_N \bigl[ Y^{p,1}_t \mid\ItF^{b,\chi^{p,\wmf}}_t\bigr]\bbr.
\end{displaymath}
We notice that the function $F^p:\bOm^0\times \bOm^{p,1}\times \dots \times\bOm^{p,N_p}\rightarrow \R$ defined by
\begin{displaymath}
F^p(t,(\omega^0,\omega^{p,1}\dots,\omega^{p,N_p})):=\bII\frac{1}{N_p}\sum_{j=1}^{N_p} Y^{p,j}_t(\omega^0,\omega^{p,j}) - \bar{\E}_N\bigl[Y^{p,1}_t\mid\ItF^{b,\chi^{p,\wmf}}_t\bigr](\omega^0)\bII^2
\end{displaymath}
is measurable (recall that $Y^{p,j}$ is progressively measurable, see \cite[Remark 1.34]{carmona2018probabilistic2}). By Lemma \ref{lemma: boundedness fourth moment sequence of processes}, Fubini's theorem, Cauchy-Schwarz and Jensen's inequalities, together with conditional independence, we get
\begin{small}
\begin{align*}
\bar{\E}_N &\bblq  \int_0^T F^p(t,(\omega^0,\omega^{p,1}\dots,\omega^{p,N_p}))dt\bbrq = \int_0^T\frac{1}{N_p^2}\bar{\E}_N\bblq\bI\sum_{j=1}^{N_p}\bl Y^{p,j}_t - \bar{\E}_N[Y^{p,1}_t\mid\ItF^{b,\chi^{p,\wmf}}_t]\br\bI^2\bbrq dt\\ 
&= \int_0^T\frac{1}{N_p^2}\Bigg(\bar{\E}_N\bblq\sum_{j=1}^{N_p}\bI Y^{p,j}_t - \bar{\E}_N[Y^{p,j}_t\mid\ItF^{b,\chi^{p,\wmf}}_t]\bI ^2\bbrq\\
    &\quad+ 2 \bar{\E}_N\bblq\sum_{h,k=1, h\neq k}^{N_p}\bl Y^{p,h}_t- \bar{\E}_N[Y^{p,h}_t\mid\ItF^{b,\chi^{p,\wmf}}_t]\br\bl Y^{p,k}_t  - \bar{\E}_N[Y^{p,k}_t\mid\ItF^{b,\chi^{p,\wmf}}_t]\br\bbrq \Bigg)dt\\
	&= \int_0^T\frac{1}{N_p^2}\bar{\E}_N\bblq\sum_{j=1}^{N_p}\bI Y^{p,j}_t - \bar{\E}_N[Y^{p,j}_t\mid\ItF^{b,\chi^{p,\wmf}}_t]\bI ^2\bbrq dt\\
    &\leq\frac{4}{N_p^2} \sum_{j=1}^{N_p}  \int_0^T \bar{\E}_N\bigl[\big| Y^{p,j}_t\big|^2\bigr] dt\leq\frac{4T}{N_p} C_B^2. 
\end{align*}
\end{small}
Finally, we have that
\begin{align*}
\bar{\E}_N\bblq \int_0^T\Bigg| \frac{1}{N_I +N_S} \sum_{p = I,S}\sum_{j=1}^{N_p}\hat{\alpha}^{\text{mf};p,j}_{t}\Bigg|^2 dt\bbrq 	&\leq 2\sum_{p = I,S}(n_p\bar{\Lambda}^p)^2\bar{\E}_N\bblq\int_0^T F^p(t,(\omega^0,\omega^{p,1}\dots,\omega^{p,N_p}))dt\bbrq	\\
&\leq\frac{8TC^2_B\sum_{p = I,S}n_p (\bar{\Lambda}^p)^2}{N_I + N_S}. 
\qedhere
\end{align*}
\end{proof}

\begin{remark}[On the equilibrium consistency condition]
\label{rmk: extra assumption for asymptotic mk clearing} 
{As discussed at the beginning of this section, we replace the consistency condition \eqref{eqn: conclusion winf projected} in the definition of a weak mean-field equilibrium (Definition \ref{def: unlifted mf price process}) by the more restrictive condition \eqref{eqn: project eq price v1}.
The reason is probabilistic: the proof of Theorem \ref{thm: weak mk clearing condition} relies on the conditional i.i.d. property of the family $(Y^{p,j})_{j = 1,\dots,N_p}$ with respect to $\F^{b,\wmf,c}$, for $p = I,S$. Condition \eqref{eqn: project eq price v1} is designed precisely to guarantee this property, which is used in the proof of Theorem \ref{thm: weak mk clearing condition} to handle the term $\frac{1}{N_I}\sum_{j = 1}^{N_I} Y^{I,j}_t-\E[Y^{I,1}_t \mid \ItF^{b,\wmf,c}_t]$.}

{If one instead works with the original consistency condition \eqref{eqn: conclusion winf projected}, then the proof would require replacing $\E[Y_t^{I,1}\mid \ItF_t^{b,\wmf,c}]$ by $\E[Y_t^{I,1}\mid \ItF_t^{b,\wmf}]$. The difficulty is that, in general, the family $(Y^{I,j})_{j=1,\ldots,N_I}$ is not conditionally i.i.d. with respect to $\mathbb F^{b,\wmf}$. 
Nevertheless, suppose that there exists a price process $\wmf$ satisfying \eqref{eqn: conclusion winf projected} and consider the corresponding finite-player economy. In analogy with Remark \ref{rmk: measurability of empirical mean}, the finite-player market clearing relation suggests that the the empirical average of the informed adjoint processes $(\frac1{N_I}\sum_{j=1}^{N_I}Y_t^{I,j})_{t\in[0,T]}$ should be adapted to $\mathbb F^{b,\wmf}$. In this case, one would formally obtain}
\begin{equation*}
    \E\bigg[\frac{1}{N_I}\sum_{j = 1}^{N_I} Y^{I,j}_t\biggm| \ItF^{b,\wmf,c}_t\bigg] =  \E\bigg[\frac{1}{N_I}\sum_{j = 1}^{N_I} Y^{I,j}_t\biggm| \ItF^{b,\wmf}_t\bigg] = \E[Y^{I,1}_t\bigm| \ItF^{b,\wmf}_t], \quad \forall t\in[0,T], \quad \prob-a.s.
\end{equation*}
{We stress, however, that this is only an a posteriori heuristic observation and cannot be used in an existence proof. Its role is only to motivate the additional structural assumptions introduced below, under which the proof of Theorem \ref{thm: weak mk clearing condition} can be adapted to the case of \eqref{eqn: conclusion winf projected}.}
{Using the representation of the adjoint process under Assumption \ref{assumption: affine target functions}}, we then deduce that 
\begin{equation}
\label{eqn: condition on cost functions}
\E\bigg[ \bar{g}^p(\chi^{p,\wmf}_T) + \int_t^T c^p(s, \chi^{p,\wmf}_s)ds \biggm| \ItF^{b,\wmf,c}_t\bigg] = \E\bigg[ \bar{g}^p(\chi^{p,\wmf}_T) + \int_t^T c^p(s, \chi^{p,\wmf}_s)ds \biggm| \ItF^{b,\wmf}_t\bigg].
\end{equation}
For $t = T$, equation \eqref{eqn: condition on cost functions} implies that $\bar{g}^I(\wmf_T,c_T)$ is a measurable function of $(b,\wmf)$. As a consequence, it is natural to assume that 
\begin{equation}
\label{eqn:str_ass_g}
\bar{g}^I(\varpi^{\rm mf}_T, c_T)
=\tilde g^I(\varpi^{\rm mf}_T, b_T) 
\end{equation}
for a suitable continuous function $\tilde{g}^I:\R^2 \to \R$. Reasoning analogously, one may suppose that there exists a continuous function $\tilde{c}^I:[0,T]\times \R^2 \to \R$ such that 
\begin{equation}
\label{eqn:str_ass_c}
c^I(s,\wmf_s, c_s) = \tilde{c}^I(s,\wmf_s, b_s),
\qquad\text{ for all }s \in[0,T].
\end{equation} 
Under these additional structural conditions, the optimal control problem of every informed agent depends only on the equilibrium price process and the common noise. Hence, in \eqref{eqn: YW dynamics}, the adjoint processes $(Y^{I,j})_{j = 1}^{N_I}$ are measurable functions of $(\eta^{I,j}, b, \wmf, w^{I,j})$, hence they are conditionally i.i.d. with respect to $\F^{b,\wmf}$. Therefore, under the additional structural conditions \eqref{eqn:str_ass_g}-\eqref{eqn:str_ass_c}, Theorem \ref{thm: weak mk clearing condition} can be proved analogously also in the case of the original consistency condition \eqref{eqn: conclusion winf projected}. We refer to \S \ref{section: one informed agent} for a single-informed-agent example in which such a structural restriction is imposed.
\end{remark}

\begin{remark}
\label{rem: 6.3}
 Let us further comment on why we do not provide existence of equilibria satisfying a consistency condition of the form \eqref{eqn: project eq price v1}, {which is the condition suggested by the formal limit of the finite-player market clearing relation}. The main challenge lies in the presence of two different filtrations in the terms defining the equilibrium. In particular, the discretization procedure outlined in \S \ref{section: existence of solutions} does not seem to be directly applicable in this setting. To employ a similar approach, one would need to construct an input-output map $\bPhi$, analogous to \eqref{eqn: prod input-output}, but defined by the conditional expectations of the adjoint processes of the representative agents of the two subpopulations with respect to a discretization of $(B,C)$. 
The discretization of $C$ introduces a dependence of the input-output map inside the conditional expectation of the representative standard agent, and this makes it difficult to prove continuity of $\bPhi$.
{At present, our discretization method does not seem to extend to that setting and establishing existence for this modified equilibrium notion remains an open problem within the present approach.}

\end{remark}

\section{A special case: a single informed agent} 
\label{section: one informed agent}
    
In this section, we present an example  in which there is a single informed agent and a population of symmetric standard agents. We consider a family of $N_S$ standard agents and single informed agent ($N_I = 1$), who can access an additional stochastic process $C$, given by a Brownian motion correlated with the common noise, i.e., $C_t= \rho^2 B_t + \sqrt{1-\rho^2} B^2_t$, which satisfies the properties discussed in Remark \ref{rmk: compatibility example}. We suppose that Assumption \ref{assumption: coefficients} and Assumption \ref{assumption: affine target functions} hold. In addition, we suppose that the initial values of the state variables are deterministic $\xi^I$, $\xi^{S,j}\in\R$. As a consequence, the filtrations to which the controls of the agents are adapted are given by $\F^I := \F^{B, \w,C}$ and $\F^{S,j} := \F^{B,W^{S,j},\w}$, for every $j = 1,\dots,N_S$. The optimal control problem of the informed agent is specified as follows:
\begin{equation}
\label{eqn: major ocp}
\inf_{\beta \in \mathbb{H}^2(\mathbb{F}^{I})} J^{I}(\beta), \qquad 
J^{I}(\beta) := \E \bblq \int_0^T f^I_{N_S}(s, X_s, \w_s, \beta_s,C_s) ds + g^I(X_T, \w_T, C_T) \bbrq,
\end{equation}
subject to
\begin{equation*}
\begin{cases} 
dX^I_t = (\beta_t + l^I(t,\w_t))dt + \sigma^{I,0}(t,\w_t) dB_t,\\
X^I_0 = \xi^I,
\end{cases}
\end{equation*}
where 
\begin{align*} 
f^I_{N_S} ( s,x,\beta, \w,c) 	&:= \w\beta + \frac{1}{2} \frac{\Lambda^I}{N_S}\beta^2 + x c^I(t,\w,c),\\
g^I (x,\w,c) 			&= x \bar{g}^I(\w,c).
\end{align*}
In analogy to \cite[Remark 3.11]{fujii2022equilibrium}, we notice that for the analysis with a fixed $N_S$, the scaling term appearing in $f^I_{N_S}$ is arbitrary and irrelevant. On the other hand, this factor is fundamental to study the large population limit $N_S\to\infty$. Indeed, the denominator which appears in the quadratic term of the cost function $f^I_{N_S}$ is necessary to guarantee that the impact of the informed agent does not become negligible in the mean-field limit. In addition, note that there is no idiosyncratic term in the dynamics of the informed agent. 

\begin{remark}
We highlight that the framework described in this section differs from that of \cite{fujii2022equilibrium} in the following key aspects. 
{First, the major agent is a price taker at the level of her individual optimization problem and does not solve a strategic price-impact problem. As in the rest of the paper, her control problem is formulated for a given price process $\varpi$, which is treated as an exogenous random environment at the optimization stage. Nevertheless, because her contribution enters the market clearing condition with non-vanishing weight, her strategy affects the equilibrium price, as made clear by \eqref{eqn: mc price process N+1} below. Second, the major agent has access to private information, whereas in \cite{fujii2022equilibrium} she is assumed to observe only the common noise.}
\end{remark} 

The stochastic maximum principle implies that the optimal controls are of the form:
\begin{align*}
\hat\alpha^{S}(y,\w):= -\bar\Lambda^S(y +\w),\qquad
\hat\beta(y,\w):= - \bar\Lambda^IN_S (y+ \w).
\end{align*}
As discussed in \S \ref{section: FBSDE from SMP}, the optimal control problem of the standard agents is solved by the FBSDE system in \eqref{eqn: FBSDE:N} for $p = S$, while the solution of the optimal control problem of the informed agent is characterized by 
\begin{equation}
\label{eqn: FBSDE major}
\begin{cases}
dX^{I}_t = (-\bar{\Lambda}^I(Y^{I}_t+\w_t) + l^I(t,{\w}_t))dt + \sigma^{I,0}(t,{\w}_t)dB_t,\quad &X^{I}_0 = \xi^I,\\ 
dY^{I}_t = -c^I(t,\w_t, C_t)dt + Z^{I,0}_t dB_t+ dM^{I}_t,\quad &Y^{I}_T= \bar{g}^I(\w_T,C_T).
\end{cases}
\end{equation}
The market clearing condition \eqref{eqn: market clearing N player}, in the case of a large population of standard agents, is given by $\frac{1}{N_S+1}(\sum_{j = 1}^{N_S} \hat\alpha^{S,j}_t + \hat\beta_t) = 0$, where $\hat\alpha^{S,j}_t := \hat\alpha^S(Y^{S,j}_t, \w_t)$ and $\hat\beta_t = \hat\beta(Y^{I}_t ,\w_t)$. This condition, when applied to the candidate optimal control, provides an equation for the equilibrium price process: 
\begin{equation}
\label{eqn: mc price process N+1}
\w_t = -(\bar\Lambda^I + \bar\Lambda^S)^{-1}\bbl\frac{1}{N_S}\bar\Lambda^S\sum_{j = 1}^{N_S} Y^{S,j}_t+\bar\Lambda^I Y^{I}_t\bbr,\quad t\in[0,T].
\end{equation}

{In this section, we do not prove the existence of a price process solving \eqref{eqn: mc price process N+1}. Similarly to the setup considered in \S \ref{section: mcc}, also condition \eqref{eqn: mc price process N+1} does not appear to be solvable due to the highly recursive structure of the problem. 
However, in analogy to Remark \ref{rmk: extra assumption for asymptotic mk clearing}, if there exists a process $\varpi$ satisfying \eqref{eqn: mc price process N+1}, then the latter condition implies that $Y^{I}_t$ must be $\ItF^{B,\w}_t$- measurable.} 
This means that the optimal control $\hat\beta_t$ of the informed agent is fully determined by $(B,\w)$, where $\w$ solves \eqref{eqn: mc price process N+1}. 
As a consequence, thanks to the linearity in the $x$-variable of the target functions ensured by Assumption \ref{assumption: affine target functions} together with the fact that $Y^{I}_T$ is $\ItF^{\w,B}_T$-measurable, it holds that $Y^{I}_T = \tilde{g}^I(\w,B)$, for a suitable measurable function $\tilde{g}^I$ defined on $\mathcal{D}([0,T],\R)\times \mathcal{C}([0,T],\R)$. Reasoning analogously for the coefficient $c^I(t,\w,c)$, {this motivates the structural restriction introduced below, which is necessary to apply the methodology developed in \S \ref{section: existence of solutions}}.

\begin{assumption}
\label{assumption: coefficients major agent by mcc}
{There exist bounded and continuous functions
\[
\tilde{c}^I: [0,T] \times \R^2, \qquad
\tilde{g}^I : \R^2\to \R,
\] 
such that, for every $t\in[0,T]$ and $\varpi\in\R$, it holds $\mathbb{P}$-a.s. that
\begin{align*}
   c^I(t,\varpi,C_t) &= \tilde{c}^I(t,\varpi,B_t), \\
   \bar{g}^I(\w,C_T) &= \tilde{g}^I(\w,B_T).
\end{align*}
Moreover, the standard agents' cost functions satisfy the linear structure of Assumption \ref{assumption: affine target functions}. 
}
\end{assumption}

\begin{remark}
We emphasize that the influence of the private signal $C$ does not disappear, even though $C$ no longer appears explicitly in the coefficients after Assumption \ref{assumption: coefficients major agent by mcc}. The informed strategy $\hat\beta$ may still depend on $C$ through the equilibrium price process $\varpi$, which, contrary to the standard assumption in the literature (see e.g. \cite{fujii2022mean}), need not be adapted to the filtration generated by $B$ alone.
More precisely, the gap between the filtrations generated by $B$ and by $(\w, B)$ captures the influence of the unobserved signal $C$ on the equilibrium dynamics.
\end{remark}

Similarly to \S \ref{section: mcc}, it is convenient  to study the mean-field limit of the equilibrium price equation \eqref{eqn: mc price process N+1}. To this end, we let $N_S\to \infty$ and consider the optimal control problem of a unique typical standard agent and the optimal control problem of the informed agent, defined as follows: 
\begin{align}
\label{eqn: major ocp mf}
&\inf_{\beta \in \mathbb{H}^2(\mathbb{F}^{\w,B})} J^{I}(\beta),\qquad J^{I}(\beta) := \E \bblq \int_0^T \tilde{f}^I(s, X_s, \beta_s, \w_s, B_s) ds + X_T\,\tilde{g}^I(\w_T,B_T) \bbrq,
\end{align}
subject to 
\begin{equation*}
\begin{cases} 
dX^I_t = (\beta_t + l^I(t,\w_t))dt + \sigma^{I,0}(t,\w_t) dB_t,\\
X^I_0 = \xi^I,
\end{cases}
\end{equation*}
where $\tilde{f}^I (t,x,\beta, \w, b) := \w\beta + \frac{1}{2} \Lambda^I\beta^2 + 	x\tilde{c}^I(t,\w,b)$. 
The candidate optimal control of the informed agent is given by $\hat\beta_t := -\bar\Lambda^I(Y^I_t+\w_t)$, where $Y^I_t$ solves the backward component of the FBSDE associated with the stochastic maximum principle applied to \eqref{eqn: major ocp mf}: 
\begin{equation}
\label{eqn: FBSDE major v2}
\begin{cases}
dX^{I}_t = (-\bar{\Lambda}^I(Y^{I}_t+\w_t) + l^I(t,{\w}_t))dt + \sigma^{I,0}(t,{\w}_t)dB_t,\quad &X^{I}_0 = \xi^I,\\ 
dY^{I}_t 	= -\tilde{c}^I(t,\w_t,B_t)  dt + Z^{I}_t dB_t + dM^{I}_t, \quad &Y^{I}_T 	=\tilde{g}^I(\w_T,B_T).
\end{cases}
\end{equation}
$M^{I}$ is a martingale adapted to $\F^{\w,B}$. By a reasoning analogous to that of \S \ref{section: definition of solution} and recalling Assumption \ref{assumption: affine target functions}, we can define a weak (unlifted) mean-field equilibrium in analogy to Definition \ref{def: unlifted mf price process}. We recall that we are considering here deterministic initial conditions $\xi^I, \xi^S$. 

\begin{defin}[Mean-field equilibrium]
\label{def: unlifted mf price process example}
We say that 
\begin{equation*}
\tilde\Theta := (\Omega, \ItF, \prob, \F,(B,W^S), \w)
\end{equation*}
is a \emph{weak mean-field equilibrium} {for the single-informed-agent model} if
\begin{itemize}
\item $((\Omega, \ItF, \prob), \tilde\F^S,(\xi^S, (B,W^S), \w))$ and $((\Omega, \ItF, \prob), \tilde\F^I,(\xi^I, B, \w))$ are admissible probabilistic setups in the sense of Definition \ref{def: admissibility}, 
where $\tilde\F^S := \F^{ B,W^S, \w}$ and $\tilde\F^I = \F^{B,\w}$;
\item the process $\w$ solves the equation
\begin{equation}
\label{eqn: eq price process mf 1 major agent}
\w_t = -(\bar\Lambda^I+\bar\Lambda^S)^{-1}\bl\E[\bar\Lambda^S\tilde{Y}^{S}_t \mid \ItF^{B,\w}_t] + \bar\Lambda^I\tilde{Y}^I_t\br,\quad \qquad \forall t\in[0,T], \quad \mathbb{P}-a.s.,
\end{equation}
where $\tilde{Y}^{p}$ is the solution of the backward component of 
the FBSDEs
for $p = S,I$.
\end{itemize}
\end{defin}

{
We can then state as follows the main existence result in the setup of this section. 
We stress that the theorem below concerns the mean-field equilibrium equation \eqref{eqn: eq price process mf 1 major agent}, not the finite-player equilibrium price equation \eqref{eqn: mc price process N+1}. This result should therefore be viewed as a benchmark mean-field analysis for a single informed agent, conditionally on the validity of Assumption \ref{assumption: coefficients major agent by mcc}.
\begin{theorem}
    Under Assumption \ref{assumption: coefficients major agent by mcc}, there exists a mean-field equilibrium for the single-informed-agent model in the sense of Definition \ref{def: unlifted mf price process example}.  
\end{theorem}
\begin{proof} 
Following the arguments of \S \ref{section: existence of solutions} we can prove that there exists a stochastic process defined on a suitable filtered probability space $\setup$ which satisfies 
\begin{equation}
\label{eqn:price_major}
\w_t = -(\bar\Lambda^I+\bar\Lambda^S)^{-1}\E[\bar{\Lambda}^S\tilde{Y}^{S}_t +\bar{\Lambda}^I\tilde{Y}^I_t\mid \ItF^{B,\w}_t], \qquad \forall t\in[0,T], \quad \mathbb{P}-a.s.
\end{equation}
We recall that, by Assumption \ref{assumption: coefficients major agent by mcc}, we can consider the adjoint process satisfying the following conditions: $\tilde{Y}^I_t$ is $\ItF^{B,\w}_t$-measurable and $\tilde{Y}^S_t$ is $\ItF^{B,\w,W^S}_t$-measurable. 
Therefore, equations \eqref{eqn: eq price process mf 1 major agent} and \eqref{eqn:price_major} are equivalent and we can  conclude that there exists a mean-field equilibrium $(\Omega, \ItF, \prob, \F,(B,W^S), \w)$ as in Definition \ref{def: unlifted mf price process example}. 
\end{proof} 
}

{In the economy considered in this section, the informed agent’s informational advantage becomes revealed through the equilibrium price. 
In this special setting, the mean-field equilibrium price process ${\w}$ solving \eqref{eqn: eq price process mf 1 major agent} is sufficiently informative that the informed agent’s optimal strategy is measurable with respect to the public filtration $\mathbb F^{B,\varpi}$ and, hence, can be identified from public information.
Indeed, for every $t\in[0,T]$, $\mathbb{P}$-a.s.\begin{displaymath}
\begin{split}
\hat{\beta}_t 	&= -\bar{\Lambda}^I(\tilde{Y}^I_t + {\w}_t)= (\bar{\Lambda}^S+ \bar{\Lambda}^I){\w}_t + \E[\bar{\Lambda}^S\tilde{Y}^S_t \mid \mathcal{F}^{{B},{\w}}_t] - \bar{\Lambda}^I{\w}_t =\bar{\Lambda}^S({\w}_t + \E[\tilde{Y}^S_t \mid \mathcal{F}^{{B},{\w}}_t]),
\end{split}
\end{displaymath}
showing that $\hat\beta_t$ is $\mathcal F_t^{B,\w}$-measurable. Therefore, the private signal still affects the equilibrium, but in this particular example its effect is fully transmitted through the equilibrium price  in such a way that the informed agent’s optimal strategy becomes identifiable from public information.}

\appendix

\section{Proof of Proposition \ref{lemma: iter random env}} 
\label{appendix: proof iter random env}
The well-posedness result follows the proof of \cite[Theorem 1.60]{carmona2018probabilistic2}. It first shows well-posedness of FBSDE \eqref{eqn: FBSDE} for short time horizon and then, in order to establish well-posedness for any time horizon, a key step is to establish Lipschitz continuity of the decoupling field. We thus prove here point (iii) of the statement, in particular we provide an explicit bound for the Lipschitz constant in condition \eqref{eqn: iterrandom}. As in that result, this is mainly due to the convexity of the cost coefficients.  

First, we need a preliminary stability estimate for two solutions of system \eqref{eqn: FBSDE}.

\begin{lemma}
\label{lemma: stability conditions result}
Let us consider two solutions of \eqref{eqn: FBSDE}, denoted by $(X^{p,1},Y^{p,1},Z^{p,0,1},Z^{p,1},M^{p,1})$ and $(X^{p,2},Y^{p,2},Z^{p,0,2},Z^{p,2},M^{p,2})$, defined by two different compatible processes $(\xi^{p,1}, B,W^p,\chi^{p,\w^1})$ and $(\xi^{p,2},B,W^p,\chi^{p,\w^2})$. Then, there exists a constant $C(T,L)>0$ (depending on $L$ introduced in Assumption \ref{assumption: coefficients} and on $T$) and a constant $C(T,\Lambda)>0$ (depending on $\Lambda^p$ and $T$) such that:
\begin{align}
\label{eqn: estimation stability}
&\E\bblq\sup_{t\in[0,T]}|Y^{p,1}_t - Y^{p,2}_t|^2 \biggm|\ItF^p_0\bbrq \\ 
&\qquad\leq C(T,L) \E\bblq \sup_{t\in[0,T]} |X^{p,1}_t - X^{p,2}_t|^2 +  |\partial_xg^p(X^{p,1}_T, \chi^{p,\w^1}_T) - \partial_x g^p(X^{p,1}_T,\chi^{p,\w^2}_T)|^2\nonumber \\ 
&\qquad\quad+ \int_0^T |\partial_x\bar{f}^p(t,X^{p,1}_t,\chi^{p,\w^1}_t) - \partial_x\bar{f}^p(t,X^{p,1}_t, \chi^{p,\w^2}_t) |^2dt  \biggm| \ItF^p_0\bbrq, \nonumber\\
\label{eqn: estimation stability X}
&\E\bblq \sup_{t\in[0,T]} |X^{p,1}_t - X^{p,2}_t|^2\biggm| \ItF^p_0\bbrq\\
&\qquad \leq C(T,\Lambda)\Bigg(|\xi^{p,1} - \xi^{p,2} |^2 + \E\bblq  \sup_{t\in[0,T]} |Y^{p,1}_t - Y^{p,2}_t|^2 +\int_0^T \Big(
|\w^1_t - \w^2_t|^2
\nonumber\\
&\qquad\quad
+|l^p(t,\w^1_t) - l^p(t,\w^2_t)|^2
+ |\sigma^{p,0}(t,\w^1_t) -\sigma^{p,0}(t,\w^2_t)|^2 + |\sigma^{p}(t,\w^1_t) -\sigma^{p}(t,\w^2_t)|^2\Big) dt \biggm| \ItF^p_0 \bbrq\Bigg)\nonumber.
\end{align}
\begin{proof}
We recall that, by \cite[Theorem 1.60]{carmona2018probabilistic2}, the stochastic integrals appearing in the process $Y^{p,j}_t$ and the process $M^p$ are true martingales. By integration by parts, we have that
\begin{align*}
&\bI Y^{p,1}_t - Y^{p,2}_t\bI^2	\\
&\qquad= \bI\partial_xg^p(X^{p,1}_T,\chi^{p,\w^1}_T) -\partial_xg^p( X^{p,2}_T,\chi^{p,\w^1}_T)\bI^2 \\
&\qquad\quad- 2\Bigg(\int_t^T(Y^{p,1}_s - Y^{p,1}_s)(-\partial_x \bar{f}^p(s,X^{p,1}_s,\w^{p,2}_s) + \partial_x\bar{f}^p(s,X^{p,2}_s,\w^{p,2}_s))ds \\ 
&\qquad\quad +\int_t^T(Y^{p,1}_s - Y^{p,2}_s)\big((Z^{p,0,1}_s - Z^{p,0,2}_s)dB_s + (Z^{p,1}_s - Z^{p,2}_s) dW^p_s+ dM^p_s\big) \Bigg) \\ 
&\qquad\quad - \int_t^T \big( |Z^{p,0,1}_s - Z^{p,0,2}_s |^2 + |Z^{p,1}_s - Z^{p,2}_s|^2\big)ds - [M^{p,1}- M^{p,2}]_T + [M^{p,1} - M^{p,2}]_t.
\end{align*}
We define $\gamma_t: =  \int_t^T\big( |Z^{p,0,1}_s - Z^{p,0,2}_s |^2 + |Z^{p,1}_s - Z^{p,2}_s|^2\big)ds + [M^{p,1}- M^{p,2}]_T-[M^{p,1}- M^{p,2}]_t$, for $t\in [0,T]$. We then take the conditional expectation with respect to $\ItF^p_0$:
\begin{align}
\label{eqn: cond expect square of difference}
&\E\big[ |Y^{p,1}_t - Y^{p,2}_t|^2 + \gamma_t \big|\ItF^p_0\big] \\		
\nonumber&\qquad\leq \E\bigg[ |\partial_xg^p(X^{p,1}_T, \chi^{p,\w^1}_T) - \partial_x g(X^{p,2}_T,\chi^{p,\w^2}_T) |^2 +  \int_t^T \bl |Y^{p,1}_s - Y^{p,2}_s|^2 + |\partial_x \bar{f}^p(s,X^{p,1}_s, \chi^{p,\w^1}_s) \\ 
\nonumber&\qquad \quad -\partial_x\bar{f}^p(s,X^{p,2}_s, \chi^{p,\w^2}_s)|^2\br ds \bigg|\ItF^p_0\bigg] \\ 
\nonumber&\qquad\leq \E\bigg[ |\partial_xg^p(X^{p,1}_T, \chi^{p,\w^1}_T) - \partial_x g^p(X^{p,2}_T,\chi^{p,\w^2}_T) |^2 +  \int_t^T |Y^{p,1}_s - Y^{p,2}_s|^2 ds +2 \int_t^T\Big(  |\partial_x \bar{f}^p(s,X^{p,1}_s, \chi^{p,\w^1}_s)\\ 
\nonumber&\qquad\qquad -\partial_x \bar{f}^p(s,X^{p,1}_s, \chi^{p,\w^2}_s)|^2 +L^2| X^{p,1}_s - X^{p,2}_s|^2 \Big) ds \bigg|\ItF^p_0\bigg].
\end{align}
We now multiply by $e^{\alpha t}$, for $\alpha>0$, and integrate over the interval $[0,T]$. Hence, \eqref{eqn: cond expect square of difference} becomes
\begin{align}
\label{eqn: estimate on Y and gamma}
&\E\bblq\int_0^Te^{\alpha t} |Y^{p,1}_t - Y^{p,2}_t|^2 dt +\int_0^Te^{\alpha t} \gamma_t dt \Biggm|\ItF^p_0\bbrq\\
\nonumber &\leq \E\bblq \bl\alphaexpinT\br |\partial_xg^p(X^{p,1}_T, \chi^{p,\w^1}_T) - \partial_x g^p(X^{p,2}_T,\chi^{p,\w^2}_T) |^2 +  \int_0^T \bl \alphaexpint \br|Y^{p,1}_t - Y^{p,2}_t|^2 dt \\ 
\nonumber&\qquad + 2 \int_0^T\bl \alphaexpint \br \Big( |\partial_x \bar{f}^p(t,X^{p,1}_t, \chi^{p,\w^1}_t)-\partial_x \bar{f}^p(t,X^{p,1}_t, \chi^{p,\w^2}_t)|^2 +L^2| X^{p,1}_t - X^{p,2}_t|^2 \Big) dt \Biggm| \ItF^p_0\bbrq.
\end{align}
For $\alpha = 1$, by Assumption \ref{assumption: coefficients} we have that
\begin{align}
\label{eqn: estimate integral deltaY}
&\E\bblq \int_0^T | Y^{p,1}_t - Y^{p,2}_t|^2 dt \Biggm|\ItF_0^p \bbrq 	\\
&\nonumber\qquad\leq \E\bblq  \int_0^T e^t | Y^{p,1}_t -	 Y^{p,2}_t|^2 dt + \int_0^T e^t\gamma_t dt \Biggm| \ItF^p_0\bbrq\\ 
\nonumber&\qquad\leq 2\E\bblq ( e^T-1) \Bigg( L^2|X^{p,1}_T - X^{p,2}_T|^2 + |\partial_xg^p(X^{p,1}_T,\chi^{p, \w^1}_T)  - \partial_x g^p(X^{p,1}_T,\chi^{p,\w^2}_T) |^2 \\ 
\nonumber&\qquad\quad+\int_0^T \Big( L^2|X^{p,1}_t - X^{p,2}_t|^2 + |\partial_x \bar{f}^p(t,X^{p,1}_t, \chi^{p,\w^1}_t)-\partial_x \bar{f}^p(t,X^{p,1}_t,\chi^{p, \w^2}_t)|^2\Big) dt\Bigg) \Biggm| \ItF^p_0\bbrq\\ 
\nonumber&\qquad\leq 2 (e^T\!-1) \E\bblq L^2(1+T)\sup_{t\in[0,T]}|X^{p,1}_t - X^{p,2}_t|^2+  |\partial_xg^p(X^{p,1}_T,\chi^{p, \w^1}_T)  - \partial_x g^p(X^{p,1}_T,\chi^{p,\w^2}_T) |^2\\
\nonumber&\qquad\quad + \int_0^T|\partial_x \bar{f}^p(t,X^{p,1}_t, \chi^{p,\w^1}_t)-\partial_x \bar{f}^p(t,X^{p,1}_t,\chi^{p, \w^2}_t)|^2 dt \Biggm| \ItF^p_0\bbrq.
\end{align}
We now apply estimate \eqref{eqn: estimate integral deltaY} to \eqref{eqn: cond expect square of difference} at $t=0$:
\begin{align}
\label{eqn: inequality gamma}
&\E\bigl[ \gamma_0  \big| \ItF^p_0 \bigr]	\\ 			
&\nonumber\qquad \leq \E\Big[ |Y^{p,1}_0 - Y^{p,2}_0|^2+ \gamma_0  \Big| \ItF^p_0\Big]\\ 
\nonumber	&\qquad
\leq \E\bigg[ |\partial_xg^p(X^{p,1}_T, \chi^{p,\w^1}_T) - \partial_x g^p(X^{p,2}_T,\chi^{p,\w^2}_T) |^2 +  \int_0^T |Y^{p,1}_s - Y^{p,2}_s|^2 ds \\ 
\nonumber	&\qquad\quad
+2 \int_0^T( |\partial_x \bar{f}^p(s,X^{p,1}_s, \chi^{p,\w^1}_s) -\partial_x \bar{f}^p(s,X^{p,1}_s, \chi^{p,\w^2}_s)|^2 +L^2| X^{p,1}_s - X^{p,2}_s|^2 )ds \bigg|\ItF^p_0\bigg]\\ 
\nonumber	&\qquad
\leq 2 e^T \E\bigg[ |\partial_xg^p(X^{p,1}_T, \chi^{p,\w^1}_T) - \partial_x g^p(X^{p,1}_T,\chi^{p,\w^2}_T) |^2 
+L^2(1 +T)\sup_{t\in[0,T]}|X^{p,1}_t - X^{p,2}_t|^2 \\
\nonumber	&\qquad\quad
+\int_0^T\Big|\partial_x \bar{f}^p(s,X^{p,1}_s, \chi^{p,\w^1}_s) -\partial_x \bar{f}^p(s,X^{p,1}_s, \chi^{p,\w^2}_s)\Big|^2ds  \bigg|\ItF^p_0\bigg].
\end{align}
We now apply Doob's conditional maximal inequality in the form 
\[
\E \Big[ \sup_{t\in [0,T]} |M_T|^2  \Big| \ItF^p_0\Big] 
\leq 4\E \Big[  |M_T|^2  \Big| \ItF^p_0\Big] 
= 4 \E \Big[  [M]_T  \Big| \ItF^p_0\Big],
\] 
for a càdlàg martingale $M$ with $M_0=0$. We then make use of the bound 
$\sup_{t\in [0,T]} |M_T-M_t|^2 \leq 2 |M_T|^2 + 2\sup_{t\in [0,T]} |M_t|^2$.
Using also 
the elementary inequality $(\sum_{i=1}^na_i)^2 \leq n\sum_{i=1}^n a^2_i$ to $(Y^{p,1}_t - Y^{p,2}_t)$, Jensen's inequality and It\^o isometry, we  obtain 
\begin{align}
\label{eqn: cs ineq Y}
&\E\bblq\sup_{t\in[0,T]}|Y^{p,1}_t - Y^{p,2}_t|^2\Biggm| \ItF^p_0\bbrq	\\
&\nonumber\qquad\leq 5\E\bblq\sup_{t\in[0,T]}\Bigg(|\partial_xg^p(X^{p,1}_T, \chi^{p,\w^1}_T) - \partial_x g^p(X^{p,2}_T,\chi^{p,\w^2}_T)|^2 \\
\nonumber&\qquad\quad
+ \Bigg| \int_t^T (\partial_x\bar{f}^p(s,X^{p,1}_s,\chi^{p,\w^1}_s)
- \partial_x\bar{f}^p(s,X^{p,2}_s,\chi^{p,\w^2}_s))ds\Bigg|^2 \\
& \nonumber\quad\qquad + \Bigg|  \int_t^T (Z^{p,0,1}_s - Z^{p,0,2}_s) db_s\Bigg|^2+ \Bigg|  \int_t^T (Z^{p,1}_s - Z^{p,2}_s) dw^p_s\Biggm|^2\\
& \nonumber\quad\qquad  +\big| M^{p,1}_T- M^{p,2}_T  - M^{p,1}_t+ M^{p,2}_t\big|^2 \bigg)\Biggm| \ItF^p_0\bbrq\\
\nonumber&\qquad\leq 10\E\bblq L^2(1+T^2) \sup_{t\in[0,T]} |X^{p,1}_t - X^{p,2}_t|^2 + |\partial_xg^p(X^{p,1}_T, \chi^{p,\w^1}_T)- \partial_x g^p(X^{p,1}_T,\chi^{p,\w^2}_T) |^2 \\
\nonumber&\qquad\quad+T\int_0^T  \big|\partial_x \bar{f}^p(s,X^{p,1}_s, \chi^{p,\w^1}_s)-\partial_x \bar{f}^p(s,X^{p,1}_s, \chi^{p,\w^2}_s)\big|^2ds  + \gamma_0 + 4\gamma_0 \Biggm| \ItF^p_0\bbrq \\
\nonumber&\qquad\leq 10 \E\bblq \big(L^2(1+T^2) + 10e^T L^2(1+T) \big) \sup_{t\in[0,T]} |X^{p,1}_t - X^{p,2}_t|^2 \\
\nonumber&\qquad\quad
+ (1+10 e^T)|\partial_xg^p(X^{p,1}_T, \chi^{p,\w^1}_T) - \partial_x g^p(X^{p,1}_T,\chi^{p,\w^2}_T) |^2 \\
\nonumber&\qquad\quad+(T+10e^T)\int_0^T  \big|\partial_x \bar{f}^p(s,X^{p,1}_s, \chi^{p,\w^1}_s)-\partial_x \bar{f}^p(s,X^{p,1}_s, \chi^{p,\w^2}_s)\big|^2ds  \Biggm| \ItF^p_0\bbrq, \\
\end{align}
which gives \eqref{eqn: estimation stability} defining $C(T,L) := 10\max\{ 
L^2(1+T^2) + 10e^T L^2(1+T), 1+10 e^T, T+10e^T 
\}$. 
Moreover, \eqref{eqn: estimation stability X} comes directly by the application of It\^o isometry and Doob's maximal inequality, with $C(T,\Lambda) :=6 \max\{4, |T\bar{\Lambda}^p|^2, T|\bar{\Lambda}^p|^2, T\}$.
\end{proof}
\end{lemma}

We now have all the ingredients to prove Proposition \ref{lemma: iter random env}. 

\begin{proof}[Proof of \eqref{eqn: iterrandom}]
We consider the two solutions introduced in Proposition \ref{lemma: iter random env}, which differ only for the initial value at $t$ of the state process.  The probabilistic setup is called  the $t$-initialized probabilistic setup (see \cite[Definition 1.39]{carmona2018probabilistic2}). By Assumption \ref{assumption: coefficients},
\begin{displaymath}
g^p(X^{p,1}_T,\chi^{p,\w}_T) - g^p(X^{p,2}_T,\chi^{p,\w}_T) \leq \partial_xg^p(X^{p,1}_T,\chi^{p,\w}_T)(X^{p,1}_T - X^{p,2}_T) = Y^{p,1}_T(X^{p,1}_T - X^{p,2}_T).
\end{displaymath}
We denote by $\hat\alpha^{p,l}$, for $l=1,2$, the candidate optimal control for the control problem defined on the $t$-initialized probabilistic setup, i.e., $\hat\alpha^{p,l}_t = -\bar\Lambda^p(Y^{p,l}_t+ \w_t)$. Hence, by integration by parts, and using that the process $X^{p,1} - X^{p,2}$ has bounded variation, we obtain that 
\begin{align*}
&\E[g^p(X^{p,1}_T,\chi^{p,\w}_T) - g^p(X^{p,2}_T,\chi^{p,\w}_T) \mid \ItF^p_t]\\
&\qquad\leq \E [Y^{p,1}_T(X^{p,1}_T - X^{p,2}_T)\mid\ItF^p_t]	\\ 
&\qquad = \E\bigg[ Y^{p,1}_t(X^{p,1}_t - X^{p,2}_t) + \int_t^T Y^{p,1}_s d(X^{p,1}_s - X^{p,2}_s) 
    + \int_t^T (X^{p,1}_s - X^{p,2}_s) dY^{p,1}_s  \biggm| \ItF^p_t \bigg]\\
&\qquad= Y^{p,1}_t(x_1-x_2) + \E\left[ \int_t^T Y^{p,1}_s(\hat{\alpha}_s^{p,1} - \hat\alpha^{p,2}_s) ds -\int_t^T (X^{p,1}_s - X^{p,2}_s) \partial_x\bar{f}^p(s,X^{p,1}_s, \chi^{p,\w}_s) ds \biggm| \ItF^p_t\right]\\
&\qquad\leq Y^{p,1}_t(x_1-x_2) + \E\bigg[ \int_t^T Y^{p,1}_s(\hat\alpha_s^{p,1} - \hat\alpha^{p,2}_s) ds +\int_t^T(\bar{f}^p(s,X^{p,2}_s, \chi^{p,\w}_s)- \bar{f}^p(s,X^{p,1}_s, \chi^{p,\w}_s)) ds \biggm| \ItF^p_t\bigg].
\end{align*}
Recalling that $Y^{p,l}_t = -(\Lambda^p\hat\alpha^{p,l}_t + \w_t) =-\partial_a{f}^p(t, X^{p,l}_t,\hat\alpha^{p,l}_t , \chi^{p,\w}_t)$, we have that
 \begin{align*}
& Y^{p,1}_t(x_2-x_1) 	\\
&\qquad\leq \E\bblq g^p(X^{p,2}_T , \chi^{p,\w}_T) - g^p(X^{p,1}_T , \chi^{p,\w}_T)  + \int_t^T\bigl( \bar{f}^p(s,X^{p,2}_s,\chi^{p, \w}_s)-\bar{f}^p(s,X^{p,1}_s,\chi^{p, \w}_s)\bigr) ds \Biggm| \ItF^p_t\bbrq \\
&\qquad\qquad +\E\bblq \int_t^T  \partial_af^p(s,X^{p,1}_s,\hat\alpha^{p,1}_s,\chi^{p,\w}_s)(\hat\alpha^{p,2}_s - \hat\alpha^{p,1}_s)ds \Biggm| \ItF^p_t\bbrq\\
&\qquad\leq \E\bblq g^p(X^{p,2}_T , \chi^{p,\w}_T) + \int_t^T {f}^p(s,X^{p,2}_s,\hat\alpha^{p,2}_s,\chi^{p,\w}_s) ds \bI \ItF^p_t\bbrq  \\ 
&\qquad\qquad -\E\bblq g^p(X^{p,1}_T , \chi^{p,\w}_T) + \int_t^T {f}^p(s,X^{p,1}_s, \hat\alpha^{p,1}_s,\chi^{p,\w}_s) ds \Biggm| \ItF^p_t\bbrq \\
&\qquad\qquad- \frac{1}{2} {\Lambda}^p\E\bblq \int_t^T|\hat\alpha^{p,2}_s - \hat\alpha^{p,1}_s|^2ds \Biggm| \ItF^p_t\bbrq,
\end{align*}
where the last equivalence holds because
\begin{align*}
&\frac{1}{2}\Lambda^p(\hat\alpha^{p,1}_s - \hat\alpha^{p,2}_s)^2\\
&\qquad=f^p(s,X^{p,1}_s,\hat\alpha^{p,2}_s,\chi^{p,\w}_s) - f^p(s,X^{p,1}_s,\hat\alpha^{p,1}_s,\chi^{p,\w}_s) - (\hat\alpha^{p,2}_s- \hat\alpha^{p,1}_s) \partial_af^p(s,X^{p,1}_s,\hat\alpha^{p,1}_s, \chi^{p,\w}_s).
\end{align*} 
Exchanging the role of $x_1$ and $x_2$, we obtain that
\begin{equation}
\label{eqn: iterrandom ineq 5}
 -(Y^{p,2}_t-Y^{p,1}_t)(x_2-x_1) 	\leq -{\Lambda}^p \E\bblq\int_t^T|\hat\alpha^{p,1}_s - \hat\alpha^{p,2}_s|^2 ds\bI\ItF^p_t\bbrq.
\end{equation}
We now observe that, by Jensen's inequality,
\begin{align*}
\E\bblq\sup_s|X^{p,1}_s - X^{p,2}_s|^2\Biggm|\ItF^p_t\bbrq 	
&\leq 2\E\bblq |x_1-x_2|^2 + \sup_{s\in[t,T]} \Biggm|\int_t^s(\hat\alpha^{p,1}_u - \hat\alpha^{p,2}_u)du \Biggm|^2\Biggm| \ItF^p_t\bbrq\\ 
&\leq 2\max\{1,T\} \Bigg( |x_1-x_2|^2 + \E\bblq \int_t^T |\hat\alpha^{p,1}_s - \hat\alpha^{p,2}_s|^2 ds \Biggm| \ItF^p_t\bbrq\Bigg)
\end{align*} 
In conclusion, by Lemma \ref{lemma: stability conditions result} and \eqref{eqn: iterrandom ineq 5}, we have that, $\prob$-a.s.
\begin{align*}
|Y^{p,1}_t - Y^{p,2}_t|^2 
&\leq \E\bblq\sup_{s\in[t,T]} |Y^{p,1}_s - Y^{p,2}_s|^2 \bI\ItF^p_t\bbrq \\
&\leq C(T,L)\E\bblq\sup_{s\in[t,T]}|X^{p,1}_s - X^{p,2}_s|^2\bI\ItF^p_t\bbrq \\ 
&\leq  2\max\{1,T\} C(T,L)\bblq |x_1-x_2|^2 + \E\bblq \int_t^T |\hat\alpha^{p,1}_s - \hat\alpha^{p,2}_s|^2 ds \Biggm| \ItF^p_t\bbrq\bbrq\\
&\leq 2\max\{1,T\}C(T,L)\bigl( |x_1-x_2|^2 + (\Lambda^p)^{-1}(Y^{p,2}_t-Y^{p,1}_t)(x_2-x_1) \bigr)\\ 
&\leq 
 \big(2\max\{1,T\}C(T,L)+ |\max\{1,T\}C(T,L) \bar{\Lambda}^p|^2 \big) |x_1-x_2|^2 + \frac12 |Y^{p,2}_t-Y^{p,1}_t|^2.
\end{align*}
This implies that: 
\begin{equation*}
|Y^{p,1}_t - Y^{p,2}_t| \leq  \Gamma_p |x_1-x_2|,\quad \prob-a.s,
\end{equation*}
where $\Gamma_p^2:=4\max\{1,T\}C(T,L)+2|\max\{1,T\}C(T,L)\bar{\Lambda}^p|^2$, 
thus proving the validity of \eqref{eqn: iterrandom}.
\end{proof}


\section{Proof of Proposition \ref{prop: Xpinf satisfies FSDE} }
\label{appendix: measurability of Xinf}


We first need a preliminary result.
 \begin{lemma}
\label{lemma: boundedness fourth moment sequence of processes}
For $p = I,S$,  $(\Ypn)_{n\in\N}$ and $(\wn)_{n\in\N}$ are uniformly bounded by a constant $C_B$ depending only on $T$ and on the constants in Assumption \ref{assumption: coefficients}, and 
 $(\Xpn)_{n\in\N}$ have uniformly bounded fourth moments.  
\end{lemma}
\begin{proof}
The bound on $\Ypn$ follows immediately from the representation of solution to the BSDE as 
\[
\Ypn_t =  \bE \bigg[ \partial_x\bar{g}^p(\Xpn_T,\chi^{p,\wn}_T) + \int_t^T \partial_x\bar{f}^p(s,\Xpn_s, \chi^{p,\wn}_s) ds \biggm|{\ItF}^p_t \bigg]
\]
and the boundedness of $\partial_x\bar{g}^p$ and $\partial_x\bar{f}^p$. Since $\Ypn$ are bounded, $\wn$ are bounded as well, thanks to their definition  \eqref{eqn: wn}. The bound on the fourth moments of $\Xpn$ is obtained by using  the (forward) SDE, the boundedness of $\Ypn$, $\wn$, and the bound of the fourth moment of the initial condition.
\end{proof}

To prove the convergence result, we adopt the strategy described in the proof of \cite[Proposition 3.11]{carmona2018probabilistic2}.

\begin{proof}[Proof of Proposition \ref{prop: Xpinf satisfies FSDE}]
 As we proved in  Lemma \ref{lemma: compatibility condition limit ocp}, the sequence of optimal controls $(\hat\alpha^{p,n})_{n\in\N}$ converges in distribution on $\mathcal{M}([0,T],\R)$ to a real-valued stochastic process $\hat\alpha^{p,\infty}$ with \cadlag trajectories, gyven by \eqref{eqn: optimal controls limit game}. This implies that, up to subsequences, the process
\begin{equation}
\label{eqn: Theta n}
\Theta^{p,n}:= (\eta^{p},b,w^p, \chi^{p,\wn}, \Ypn , \Xpn,\hat\alpha^{p,n}) \in \Om^p_{\text{input}}\times\mathcal{C}([0,T],\R) \times \mathcal{M}([0,T],\R)
\end{equation}
admits a weak limit $\Theta^{p,\infty}:= (\eta^{p,\infty}, \binf,w^{p,\infty},\chi^{p,\varpi^{\infty}}, \Ypinf, \Xpinf, \hat\alpha^{p,\infty})$. To show that the weak limit $\Xpinf$ satisfies system \eqref{eqn: limit SDE state}  we first notice that every weak limit of the initial condition 
has the same distribution of $\xi^p$. Moreover, by Assumption \ref{assumption: admissibility} together with the fact that $(\binf, w^{p,\infty}, c^{\infty})$ is distributed as $(B,W^p,C)$, we can conclude that $(\binf,w^{p,\infty})$ is a two-dimensional $\F^{\Theta^{p,\infty}}$-Brownian motion (see \cite[Lemma 2.41]{tesilanaro} for a proof). 
We notice that $X^{p,\infty}\in \mathbb{S}^2(\F^{\Theta^{p,\infty}})$ by Lemma \ref{lemma: boundedness fourth moment sequence of processes}. We also recall that $\alpha^{p,n}$ is bounded by Lemma \ref{lemma: boundedness fourth moment sequence of processes} and hence uniformly square integrable.
We then introduce the following processes:
\begin{align*}
B^{p,n}_t := \int_0^t(\hat\alpha^{p,n}_s + l^p(s,\wn_s)) ds,\
\Sigma^{p,0,n}_t := \int_0^t \sigma^{p,0}(s,\wn_s) db_s,\
\Sigma^{p,n}_t := \int_0^t \sigma^{p}(s,\wn_s) dw^p_s, 
\end{align*}
for $t\in[0,T]$.
We apply \cite[Lemma 3.6]{carmona2018probabilistic2} to prove that $B^{p,n}$ converges weakly to $B^{p,\infty}:= \int_0^{\cdot} (\hat\alpha^{p,\infty}_s + l^p(s,\winf_s))ds$ on $\mathcal{C}([0,T],\R)$. For $\Sigma^{p,0,n}$ and $\Sigma^{p,n}$, we proceed differently. We consider $\sigma^n_t\in \{ \sigma^{I}(t,\wn_t),\sigma^{I,0}(t,\wn_t),\sigma^{S}(t,\wn_t),\sigma^{S,0}(t,\wn_t)\}$ and denote by $\sigma\in\{ \sigma^I,\sigma^{I,0},\sigma^S,\sigma^{S,0}\}$. Since $\wn$ converges weakly to $\winf$ in $\mathcal{M}([0,T],\R)$ and $\sigma$ is continuous in $\w$-variable, by \cite[Lemma 3.5]{carmona2018probabilistic2}, $\sigma(t,\wn_t))_{t\in[0,T]}$ converges in distribution to $(\sigma(t,\winf_t))_{t\in[0,T]}$ on $\mathcal{M}([0,T],\R^2)$. Moreover, by Lemma \ref{lemma: boundedness fourth moment sequence of processes} and Assumption \ref{assumption: coefficients}, we have that $\sup_{n\in\N} \sup_{t\in[0,T]}|\sigma(t,\wn_t)|^2 \leq L$ for a suitable constant $L>0$. By Skorokhod's representation theorem, there exists a probability space $(\hat\Omega,\hat\ItF,\hat\prob)$ on which a sequence $(\hat\w^n, \hat{b}^n)_{n\in\N\cup\{\infty\}}$ is defined such that $(\hat\w^n,\hat{b}^n) \stackrel{d}{=} (\wn,b)$ for all $n\in\N$ and $\lim_{n\to \infty}( \hat{\w}^n, \hat{b}^n) = (\hat{\w}^\infty, \hat{b}^\infty)$ \quad $\hat\prob$-a.s. on $\mathcal{M}([0,T],\R) \times \mathcal{C}([0,T],\R)$. 
Since uniform boundedness of ${\w}^n$ is transferred to $\hat\w^n$, dominated convergence gives
\begin{equation*}
    \lim_{n\to\infty} \int_0^T | \sigma(t,\hat\w^n_t) - \sigma(t,\hat\w^\infty_t) |^2 dt = 0\quad \hat\prob-a.s.
\end{equation*}
and also 
\begin{equation}
\label{eqn: l2 convergence volatilities}
\lim_{n\to\infty} \hat\E\bigg[ \int_0^T |\sigma(t,\hat\w^n_t) - \sigma(t,\hat\w^\infty_t) |^2 dt\bigg] =0.
\end{equation}
We shall prove that
\begin{equation}
\label{eqn: C convergences stochastic integrals}
    \lim_{n\to\infty}\hat\E \bigg[ \sup_{t\in[0,T]} \bigg| \int_0^t \sigma(t,\hat\w^n_t) d\hat{b}^n_t - \int_0^t \sigma(t,\hat\w^\infty_t) d\hat{b}^\infty_t\bigg| \bigg] = 0.
\end{equation}
We write the proof for the common noise $b$, but it holds in the same way for the idiosyncratic noise $w^p$. 
Since $(\sigma(t,\winf_t))_{t\in[0,T]}\in \mathbb{H}^2(\F^{\winf})$, there exists a sequence $(\sigma^k)_{k\in\N}$ of uniformly bounded $\F^{\winf}$-progressively measurable step functions of the form $\sigma^k_t := \sum_{h = 0}^{M_k-1} \sigma^k_h \charfunc_{[t_h,t_{h+1})}(t)$, for a subdivision $0\leq t_0 <\dots < t_{M_k} \leq T$, where $\sigma^k_h$ is $\mathcal{F}^{\winf}_{t_h}$-measurable and $\sigma^k_h \leq 2L$ $\hat\prob$-a.s., such that
\begin{equation*}
\hat\E \bigg[ \int_0^T |\sigma^k_t  - \sigma(t,\hat\w^{\infty}_t)|^2 dt\bigg]\leq \frac{1}{k},\quad k\in\N.
\end{equation*}
As a consequence, for every $k \in\N$, by Doob's martingale inequality we have that
\begin{align}
\label{eqn: convergence supnorm stochastic integrals} 
&\hat\E \bigg[ \sup_{t\in[0,T]} \bigg| \int_0^t \sigma(t,\hat\w^n_t) d\hat{b}^n_t - \int_0^t \sigma(t,\hat\w^\infty_t) d\hat{b}^\infty_t\bigg| \bigg]\\
 \nonumber &\qquad\leq \hat\E \bigg[ \sup_{t\in[0,T]} \bigg| \int_0^t (\sigma(t,\hat\w^n_t)-\sigma(t,\hat\w^\infty_t))d\hat{b}^n_t\bigg| \bigg] +\hat\E \bigg[ \sup_{t\in[0,T]} \bigg| \int_0^t \sigma(t,\hat\w^\infty_t) d\hat{b}^n_t - \int_0^t \sigma(t,\hat\w^\infty_t) d\hat{b}^\infty_t\bigg| \bigg] \\
 \nonumber&\qquad\leq 2\bigg(\hat\E\bigg[ \int_0^T |\sigma(t,\hat\w^n_t) - \sigma(t,\hat\w^\infty_t) |^2 dt\bigg]\bigg)^{\frac{1}{2}} + \hat\E\bigg[\sup_{t\in[0,T]}\bigg(\bigg| \int_0^t (\sigma(t,\hat{\w}^\infty_t) -\sigma^k_t) d\hat{b}^n_t \bigg| \\
 \nonumber &\qquad \qquad+ \bigg| \int_0^t \sigma^k_t d\hat{b}^n_t - \int_0^t \sigma^k_t d\hat{b}^\infty_t \bigg| + \bigg|\int_0^t (\sigma^k_t - \sigma(t,\hat{\w}^\infty_t)) d\hat{b}^\infty_t\bigg|\bigg)\bigg].
\end{align}
Applying again martingale Doob's inequality we notice that
\begin{equation}
\label{eqn: convergence step function vol}
    \hat\E\bigg[ \sup_{t\in[0,T]} \bigg| \int_0^t (\sigma(t, \hat\w^\infty_t) - \sigma^k_t) d\hat{b}^n_t \bigg|\bigg] \leq 2\bigg(\hat\E\bigg[ \int_0^T |\sigma(t,\hat\w^\infty_t)- \sigma^k_t |^2dt\bigg]\bigg)^{\frac{1}{2}} \leq \frac{2}{\sqrt{k}}
\end{equation}
and the same holds for $\hat\E[ \sup_{t\in[0,T]}| \int_0^t (\sigma^k_t- \sigma(t, \hat\w^\infty_t) ) d\hat{b}^\infty_t |]$. On the other hand, we observe that 
\begin{align}
\label{eqn: convergence step fucntion vol II}
\hat\E\bigg[ \sup_{t\in[0,T]} \bigg| \int_0^t \sigma^k_t d\hat{b}^n_t - \int_0^t \sigma^k_t d\hat{b}^\infty_t \bigg|\bigg] &\leq\hat\E \bigg[ \sup_{t\in[0,T]} \bigg| \sum_{h = 1}^{M_k} \sigma^k_{h} \big[ \big( \hat{b}^n_{t_{h+1}} - \hat{b}^\infty_{t_{h+1}}\big) - \big( \hat{b}^n_{t_{h}} - \hat{b}^\infty_{t_{h}}\big)\big]\bigg|\bigg]\\
 \nonumber&\leq 4L M_k\hat\E \bigg[ \sup_{t\in[0,T]} |\hat{b}^n_t - \hat{b}^\infty_t|\bigg].
\end{align}
Applying \eqref{eqn: convergence step function vol} and \eqref{eqn: convergence step fucntion vol II} to \eqref{eqn: convergence supnorm stochastic integrals}, we notice that, for every $k\in\N$,
\begin{align*}
    &\hat\E \bigg[ \sup_{t\in[0,T]} \bigg| \int_0^t \sigma(t,\hat\w^n_t) d\hat{b}^n_t - \int_0^t \sigma(t,\hat\w^\infty_t) d\hat{b}^\infty_t\bigg| \bigg] \\
    &\qquad\leq  2\bigg(\hat\E\bigg[ \int_0^T |\sigma(t,\hat\w^n_t) - \sigma(t,\hat\w^\infty_t) |^2 dt\bigg]\bigg)^{\frac{1}{2}} + \frac{4}{\sqrt{k}} + 4L M_k\hat\E \bigg[ \sup_{t\in[0,T]} |\hat{b}^n_t - \hat{b}^\infty_t|\bigg].
\end{align*}
By \eqref{eqn: l2 convergence volatilities} and uniform integrability of $(\hat{b}^n_t - \hat{b}^\infty_t)_{t\in[0,T]}$, we observe that 
\begin{align*}
&\lim_{n\to\infty}\bigg(2\bigg(\hat\E\bigg[ \int_0^T |\sigma(t,\hat\w^n_t) - \sigma(t,\hat\w^\infty_t) |^2 dt\bigg]\bigg)^{\frac{1}{2}}+ 4LM_k\E \bigg[ \sup_{t\in[0,T]} |\hat{b}^n_t - \hat{b}^\infty_t|\bigg]\bigg) = 0. 
\end{align*}
This suffices to conclude that, for every $k\in\N$,
\begin{equation*}
    \limsup_{n\to\infty}\hat\E \bigg[ \sup_{t\in[0,T]} \bigg| \int_0^t \sigma(t,\hat\w^n_t) d\hat{b}^n_t - \int_0^t \sigma(t,\hat\w^\infty_t) d\hat{b}^\infty_t\bigg| \bigg] \leq \frac{4}{\sqrt{k}}.
\end{equation*}
Letting $k\to\infty$ we obtain \eqref{eqn: C convergences stochastic integrals}, which implies that $\lim_{n\to\infty} \int_0^t \sigma(t,\hat\w^n_t) d\hat{b}^n_t \stackrel{d}{=} \int_0^t \sigma(t,\hat\w^\infty_t) d\hat{b}^\infty_t$ on $\mathcal{C}([0,T],\R)$. As a consequence, we can conclude that the sequence
\begin{displaymath}
\tilde\Sigma^{p,n} := \Big(\sigma^p(\cdot,\wn_{\cdot}),\sigma^{p,0}(\cdot,\wn_{\cdot}),w^p,b,\int\sigma^{p}(t,\wn_{t}) dw^p_t,\int\sigma^{p,0}(t,\wn_{t}) db_t\Big),\quad n\in\N,
\end{displaymath} 
is tight on $\mathcal{M}([0,T],\R^2) \times \mathcal{C}([0,T],\R^4)$. Thus, up to subsequences,  $(\tilde\Sigma^{p,n})_{n\in\N}$ converges in distribution on $\mathcal{M}([0,T],\R^2) \times \mathcal{C}([0,T],\R^4)$ to 
\begin{displaymath}
\tilde\Sigma^{p,\infty} := \Big(\sigma^p(\cdot,\winf_{\cdot}),\sigma^{p,0}(\cdot,\winf_{\cdot}), w^p, \binf,\int\sigma^{p}(t,\winf_{t}) d\binf_t,\int\sigma^{p,0}(t,\winf_{t}) d\binf_t\Big).
\end{displaymath} 
Hence, up to subsequences, $(\Theta^{p,n},B^{p,n},\Sigma^{p,0,n},\Sigma^{p,n})_{n\in\N}$ converges to $(\Theta^{p,\infty},B^{p,\infty}, \Sigma^{p,0,\infty}, \Sigma^{p,\infty})$ on $\Omega^p_{\text{input}}\times \mathcal{C}([0,T],\R) \times \mathcal{M}([0,T],\R) \times  (\mathcal{C}([0,T],\R))^3$ in distribution. Finally, we consider the function $h:\Omega^p_{\text{input}}\times \mathcal{C}([0,T],\R) \times \mathcal{M}([0,T],\R) \times  (\mathcal{C}([0,T],\R))^3\to \mathcal{C}([0,T],\R)$ 
given by
\begin{align*}
(\Theta,B,\Sigma^0,\Sigma)&\mapsto (h_t(\Theta, B, \Sigma^0, \Sigma))_{t\in[0,T]} := (X_t - (X^0 + B_t + \Sigma^0_t + \Sigma_t))_{t\in[0,T]} .
\end{align*}
Since $h$ is continuous, by the continuous mapping theorem, $(h(\Theta^{p,n},B^{p,n},\Sigma^{p,0,n},\Sigma^{p,n}))_{n\in\N}$ is convergent in distribution on $\mathcal{C}([0,T],\R)$. Moreover, $h_t(\Theta^{p,n},B^{p,n},\Sigma^{p,0,n},\Sigma^{p,n}) = 0$. Therefore, we obtain $\prob^{\infty}$-a.s.
\begin{displaymath}
X^{p,\infty}_t - (\xi^{p,\infty} +B^{p,\infty}_t+ \Sigma^{p,0,\infty}_t+\Sigma^{p,\infty}_t) = 0, \quad \forall t\in[0,T],
\end{displaymath} 
which corresponds to \eqref{eqn: limit SDE state}. 
\end{proof}

\section{Meyer-Zheng space}
\label{Appendix:MZ} 

We recall here some results about the Meyer-Zheng topology which are used in the paper, referring to \cite[\S 3.2.2]{carmona2018probabilistic2} for a more complete account. The Meyer-Zheng topology on $\mathcal{M}([0,T], \R)$ is the topology of convergence in $dt$-measure  for functions from $[0,T]$ to $\R$. The set $\mathcal{M}([0,T], \R)$ is a Polish space with the metric 
\[
d_{\mathcal{M}}(x,y) = \int_0^T 1\wedge |x_t-y_t| dt.  
\]
Note that the topology of $\mathcal{M}([0,T], \R)$ is weaker than the Skorokhod topology on $\mathcal{D}([0,T], \R)$. 

The main criterion for tightness in the Meyer-Zheng space can be found in \cite[Theorem 5.8]{kurtz1991random}. This result states that if a sequence of \cadlag processes $(A^n)_{n\in\N}$ adapted to a filtration $\mathbb{G}^n$ on $(\Omega^n,\mathcal{G}^n, \mathbb{P}^n)$ satisfies
\begin{equation}
\label{eqn: Kurtz condition}
\sup_{n\in\N} \Big{\{}\E^n \big[|A^n_T| \big]+ V^{n}_T(A^n)\Big{\}}< \infty,
\end{equation}
then the laws of the processes are tight on $\mathcal{M}([0,T],\R)$ and any limit point is the law of 
 a process admitting a \cadlag version. The term $V^{n}_T(A^n)$ in \eqref{eqn: Kurtz condition} stands for the \emph{conditional variance of $A^n$ with respect to $\mathbb{G}^n$}, defined as 
\begin{equation}
\label{eqn: conditional variance}
V^{n}_t(A^{n}):= \sup_{\Delta \subset [0,t]} \E^n\bblq \sum_{i = 1}^{N_\Delta}\bI\E^n[ A^{n}_{t_{i+1}} - A^{n}_{t_i}\mid\mathcal{G}^{n}_{t_i}]\bI\bbrq, 
\end{equation}
for $t\in[0,T]$, where the supremum is taken over all partitions $\Delta$ of the time interval $[0,t]$, denoting by $N_\Delta$ the number of elements in the partition $\Delta$. 

\end{document}

%% file: newcommands.tex
\newcommand{\F}{\mathbb{F}}
\newcommand{\bF}{\overline{\mathbb{F}}}
\newcommand{\ItF}{\mathcal{F}}
\newcommand{\bItF}{\overline{\mathcal{F}}}
\newcommand{\R}{\mathbb{R}}

\newcommand{\G}{\mathcal{G}}
\newcommand{\prob}{\mathbb{P}}
\newcommand{\bprob}{\overline{\mathbb{P}}}
\newcommand{\vphi}{\varphi}
\newcommand{\w}{\varpi}
\newcommand{\E}{\mathbb{E}}
\newcommand{\bE}{\overline{\mathbb{E}}}

\newcommand{\blq}{\Bigl{[}}
\newcommand{\brq}{\Bigr{]}}
\newcommand{\bblq}{\Biggl{[}}
\newcommand{\bbrq}{\Biggr{]}}
\newcommand{\bl}{\Bigl{(}}
\newcommand{\bbl}{\Bigg{(}}
\newcommand{\br}{\Bigr{)}}
\newcommand{\bbr}{\Bigg{)}}
\newcommand{\bI}{\Big{|}}
\newcommand{\bII}{\Bigg{|}}

\newcommand{\setup}{(\Omega,\mathcal{F},\mathbb{P},\mathbb{F})}
\newcommand{\bsetup}{(\overline{\Omega},\overline{\mathcal{F}},\overline{\mathbb{P}},\overline{\mathbb{F}})}

\newcommand{\btheta}{\boldsymbol{\overline{\theta}}}
\newcommand{\bPhi}{\boldsymbol{\Phi}}

\newcommand{\bv}{\overline{v}}
\newcommand{\prodspace}{\prod_{i = 0}^{2^n-1} \mathcal{C}([t_i,t_{i+1}];\R)^{\mathbb{J}^i_l}}
\newcommand{\prodspacej}{\prod_{j = 0}^{2^n-1} \mathcal{C}([t_j,t_{j+1}];\R)^{\mathbb{J}^j_l}}

\newcommand{\vt}{\varpi^{\theta}}

\newcommand{\dcond}{V_1 = v_1,\dots, V_i = v_i}
\newcommand{\alphaexpint}{\frac{e^{\alpha t} - 1}{\alpha}}
\newcommand{\alphaexpinT}{\frac{e^{\alpha T} - 1}{\alpha}}

\newcommand{\N}{\mathbb{N}}
\newcommand{\Xnl}{X^{n,l}}

\newcommand{\XIn}{X^{I,n}}
\newcommand{\YIn}{Y^{I,n}}

\newcommand{\XNn}{X^{S,n}}
\newcommand{\Xpn}{X^{p,n}}
\newcommand{\YNn}{Y^{S,n}}
\newcommand{\Ypn}{Y^{p,n}}
\newcommand{\Ypinf}{Y^{p,\infty}}
\newcommand{\Xpinf}{X^{p,\infty}}

\newcommand{\Zopn}{Z^{p,0,n}}
\newcommand{\Zpn}{Z^{p,n}}

\newcommand{\Mpn}{M^{p,n}}

\newcommand{\bVn}{\overline{V}^{n}}

\newcommand{\Vn}{{V}^{n}}

\newcommand{\bX}{\overline{X}}
\newcommand{\bY}{\overline{Y}}
\newcommand{\bZ}{\overline{Z}}
\newcommand{\bM}{\overline{M}}

\newcommand{\Xinf}{X^{\infty}}

\newcommand{\xiinf}{\xi^{\infty}}
\newcommand{\etainf}{\eta^{\infty}}

\newcommand{\YIinf}{Y^{I,\infty}}
\newcommand{\YNinf}{Y^{S,\infty}}
\newcommand{\XIinf}{X^{I,\infty}}
\newcommand{\XNinf}{X^{S,\infty}}

\newcommand{\binf}{b^{\infty}}

\newcommand{\winf}{\varpi^{\infty}}
\newcommand{\wbminf}{w^{\infty}}

\newcommand{\Thetainf}{\Theta^{\infty}}
\newcommand{\Einf}{\mathbb{E}^{\infty}}

\newcommand{\wn}{\varpi^{n}}

\newcommand{\wnn}{\varpi^{n,2n}}

\newcommand{\Ynn}{Y^{n,2n}}
\newcommand{\Yonn}{Y^{0,n,2n}}

\newcommand{\tstop}{\cdot \wedge t} 
 
\newcommand{\charfunc}{\mathds{1}}

\newcommand{\Prob}{\mathbb{P}}
\newcommand{\wmf}{\varpi^{\text{mf}}}
\newcommand{\Om}{\Omega}
\newcommand{\bOm}{\overline{\Omega}}

\newcommand{\mffil}{\mathcal{F}^{\varpi^{\text{mf}}, b}}

\newcommand{\cadlag}{c\`{a}dl\`{a}g }

\newcommand{\processinf}{(b^{\infty},c^{\infty},\eta^{I,\infty},w^{I,\infty},\eta^{S,\infty},w^{S,\infty},\w^\infty,Y^{I,\infty},Y^{S,\infty},X^{I,\infty},X^{S,\infty})}

\newcommand{\setupinf}{(\Omega^{\infty}, \mathcal{F}^{\infty}, \mathbb{P}^{\infty})}

\newcommand{\probspace}{(\Omega,\mathcal{F},\mathbb{P})}

\newcommand{\be}{\begin{equation}}
\newcommand{\ee}{\end{equation}}
\newcommand{\ba}{\begin{aligned}}
\newcommand{\ea}{\end{aligned}}
\newcommand{\setupn}{(\Omega,\mathcal{F},\mathbb{P},\mathbb{F}^S)}
\newcommand{\setupi}{(\Omega,\mathcal{F},\mathbb{P},\mathbb{F}^I)}

\newcommand{\FBSDEn}{(X^S, Y^S, Z^{S,0}, Z^S, M^S)}
\newcommand{\FBSDEi}{(X^I, Y^I, Z^{I,0}, Z^I, M^I)}
\newcommand{\constantni}{(n_I\overline{\Lambda}^I + n_S \overline{\Lambda}^S)^{-1}}
\renewcommand{\tilde}{\widetilde}
\renewcommand{\hat}{\widehat}
\renewcommand{\bar}{\overline}